\documentclass[10pt,leqno]{amsart}
\usepackage{amsmath,amsthm,amsfonts,eucal,eufrak}
\usepackage{latexsym}
\usepackage{amssymb}
\usepackage[dvips]{epsfig}
\usepackage[active]{srcltx}
\renewcommand{\thesection}{\arabic{section}}

\newtheorem{theorem}{Theorem}[section]
\newtheorem*{thma}{Theorem A}
\newtheorem*{thmb}{Theorem B}

\newtheorem*{thmi}{Theorem $I$}

\newtheorem*{thmta}{Theorem $\tilde A$}
\newtheorem*{thmtb}{Theorem $\tilde B$}
\newtheorem*{thmtc}{Theorem $\tilde C$}
\newtheorem*{thmtd}{Theorem $\tilde D$}

\newtheorem*{thmti}{Theorem $\tilde I$}
\newtheorem{lemma}[theorem]{Lemma}
\newtheorem{prop}[theorem]{Proposition}

\newtheorem{corollary}[theorem]{Corollary}
\theoremstyle{definition}
\newtheorem{remark}[theorem]{Remark}
\theoremstyle{definition}
\newtheorem{defi}[theorem]{Definition}

\renewcommand{\theequation}{\thesection .\arabic{equation}}
\let\subs\subsection
\renewcommand\subsection{\setcounter{equation}{0}
\gdef\theequation{\thesubsection \arabic{equation}}\subs}
\let\sect\section
\renewcommand\section{\setcounter{equation}{0}
\gdef\theequation{\thesection .\arabic{equation}}\sect}

\newcommand{\cT}{{\mathcal{T}}}
\newcommand{\cS}{{\mathcal{S}}}

\newcommand{\IC}{{\mathbb{C}}}

\newcommand{\IR}{{\mathbb{R}}}

\newcommand{\IZ}{{\mathbb{Z}}}
\newcommand{\zv}{\IZ^\nu}

\newcommand{\be}{\begin{equation}}
\newcommand{\ee}{\end{equation}}
\newcommand{\nn}{\nonumber}

\newcommand{\dist}{\mathop{\rm{dist}}}

\newcommand{\spec}{\mathop{\rm{spec}}}
\newcommand{\sgn}{\mathop{\rm{sgn}}}

\newcommand{\La}{\Lambda}

\newcommand{\ve}{\varepsilon}

\newcommand{\vp}{\varphi}
\newcommand{\ka}{\kappa}

\def\Ga{\Gamma}

\def\zero{{(0)}}

\def\one{{(1)}}

\def\es{{(s)}}
\def\ar{{(r)}}
\def\esone{{(s-1)}}

\DeclareMathOperator*{\supp}{supp}

\newcommand{\eqdef}{\overset{\mathrm{def}}=}
\newcommand{\R}{\mathbb{R}}

\newcommand{\Z}{\mathbb{Z}}

\begin{document}

\smallskip

\title[The isospectral torus of quasi-periodic Schr\"odinger operators]{The isospectral torus of quasi-periodic Schr\"odinger operators via periodic approximations}

\author{David Damanik}

\address{Department of Mathematics, Rice University, Houston TX 77005, U.S.A.}

\email{damanik@rice.edu}

\author{Michael Goldstein}

\address{Department of Mathematics, University of Toronto, Bahen Centre, 40 St. George St., Toronto, Ontario, CANADA M5S 2E4}

\email{gold@math.toronto.edu}

\author{Milivoje Lukic}

\address{Department of Mathematics, Rice University, Houston TX 77005, U.S.A. and Department of Mathematics, University of Toronto, Bahen Centre, 40 St. George St., Toronto, Ontario, CANADA M5S 2E4}

\email{mlukic@math.toronto.edu}

\thanks{D.~D.\ was partially supported by NSF grants DMS--1067988 and DMS--1361625. M.~G.\ was partially supported by NSERC. M.~G.\ expresses his gratitude for the hospitality during a stay at the Institute of Mathematics at the University of Stony Brook in May 2014. M.~L.\ was partially supported by NSF grant DMS--1301582.}

\begin{abstract}
We study the quasi-periodic Schr\"odinger operator
$$
-\psi''(x) + V(x) \psi(x) = E \psi(x), \qquad x \in \IR
$$
in the regime of ``small'' $V(x) = \sum_{m\in\zv}c(m)\exp (2\pi i m\omega x)$, $\omega = (\omega_1, \dots, \omega_\nu) \in \mathbb{R}^\nu$, $|c(m)| \le \ve \exp(-\kappa_0|m|)$. We show that the set of reflectionless potentials isospectral with $V$ is homeomorphic to a torus. Moreover, we prove that any reflectionless potential $Q$ isospectral with $V$ has the form $Q (x) = \sum_{m \in \mathbb{Z}^\nu} d(m) \exp (2\pi i m\omega x)$, with the same $\omega$ and with $\lvert d(m) \rvert \le \sqrt{2 \ve} \exp(-\frac{\kappa_0}{2} |m|)$.

Our derivation relies on the study of the approximation via Hill operators with potentials $\tilde V (x) = \sum_{m \in \zv} c(m) \exp (2 \pi i m \tilde \omega x)$, where $\tilde \omega$ is a rational approximation of $\omega$. It turns out that the multi-scale analysis method of \cite{DG} applies to these Hill operators. Namely, in \cite{DGL} we developed the multi-scale analysis for the operators dual to the Hill operators in question. The main estimates obtained in  \cite{DGL} allow us here to establish estimates for the gap lengths and the Fourier coefficients in a form that is considerably stronger than the estimates known in the theory of Hill operators with analytic potentials in the general setting. Due to these estimates, the approximation procedure for the quasi-periodic potentials is effective, despite the fact that the rate of approximation $|\omega - \tilde \omega| \thicksim \tilde T^{-\delta}$, $0 < \delta < 1/2$ is slow on the scale of the period $\tilde T$ of the Hill operator.
\end{abstract}

\date{\today}

\maketitle

\tableofcontents

\section{Introduction and Statement of the Main Results}\label{sec.1}

Let $U(\theta)$ be a real function on the torus $\mathbb{T}^\nu$,
\begin{equation}\label{eq:PA17-2}
U(\theta) = \sum_{n \in \zv} c(n) e^{2 \pi i n\theta}\ , \quad \theta \in \mathbb{T}^\nu.
\end{equation}

Let $\omega = (\omega_1, \dots, \omega_\nu) \in \mathbb{R}^\nu$. Assume that the following Diophantine condition holds,
\begin{equation}\label{eq:1PAI7-5-85a}
|n \omega| \ge a_0 |n|^{-b_0}, \quad n \in \mathbb{Z}^\nu \setminus \{ 0 \}
\end{equation}
for some
\begin{equation}\label{eq:PAIombasicTcondition5a}
0 < a_0 < 1,\quad \nu < b_0 < \infty.
\end{equation}
Here, and everywhere else in this work, $|n| = |(n_1, \dots, n_\nu)| = \sum_j |n_j|$ for $n \in \zv$.

Let $V(x) = U(x \omega)$. Consider the Schr\"odinger operator
\begin{equation} \label{eq:SLPAI1-1}
[H_V y](x) = - y''(x) +  V(x) y(x), \quad x \in \IR.
\end{equation}

Assume that $U$ is real-analytic, that is, the Fourier coefficients $c(n)$ obey
\begin{equation}\label{eq:1-4}
\begin{split}
\overline{c(n)} & = c(-n), \quad n \in \zv \setminus \{ 0 \}, \\
|c(n)| & \le  \ve\exp(-\kappa_0|n|), \quad n \in \zv \setminus \{ 0 \},
\end{split}
\end{equation}
$\ve > 0$, $0 < \kappa_0 \le 1$. We also assume that $c(0) = 0$ since the constant term can be subsumed in the energy.

In \cite{DG} the eigenfunctions and the spectrum of this operator were studied via a multi-scale analysis method. The spectrum turns out to be of the form
\begin{equation}\label{eq:1spectrum}
\cS = [E(0) , \infty) \setminus \bigcup_{m \in \zv \setminus \{0\} : E^-( k_m) < E^+( k_m)} (E^-(k_m), E^+(k_m)),
\end{equation}
where $E(k)$ are the Floquet eigenvalues parametrized against the quasi-impulses $k \in \mathbb{R}$, and $k_m = -m \omega/2$ are the values at which $E(k)$ may be discontinuous. The intervals $(E^-( k_m), E^+( k_m))$ are called the gaps.

One of the main results in \cite{DG} establishes \textbf{two-sided estimates relating the gaps in the spectrum and the Fourier coefficients $c(n)$}. This enabled the authors to prove the existence and uniqueness of global solutions for the Korteweg-de Vries equation with small quasi-periodic initial data; see \cite{DG2}. However, the multi-scale analysis method of \cite{DG} does not allow one to identify the manifolds of isospectral quasi-periodic potentials and to give a solution of the inverse spectral problem for these potentials. The development of such a theory and particularly finding the ``canonical'' foliation of the space of quasi-periodic potentials into isospectral tori, similarly to the periodic case (see \cite{FlMcL, McKvM, McKTr, KaPo, PoTr}), will be instrumental for the complete integrability of the Korteweg-de Vries equation with quasi-periodic initial data. We would like to mention that the study of quasi-periodic Korteweg-de Vries equations addresses Problem~1 of Deift's list of problems in random matrix theory and the theory of integrable systems, see \cite{De}.

It is the purpose of this paper to develop such a theory. It will be central to our approach to approximate the frequency vector $\omega$ with vectors having rational components, and hence to approximate the quasi-periodic potential $V$ with periodic potentials.

In the paper \cite{DGL}, a companion to the present paper, a multi-scale analysis was developed for the spectral problem dual to the periodic approximations. This analysis is the main technological tool which enables us to develop in the current work an exact analogue of the above-mentioned \textbf{two-sided estimates relating the gaps in the spectrum for the rational approximations of the vector $\omega$ and the Fourier coefficients $c(n)$}. Moreover, these estimates are actually \textbf{uniform} for the approximating rational vectors. These results set up an effective periodic approximation theory for the quasi-periodic inverse spectral problem.

We invoke some fundamental objects from the theory of Hill's equation, namely the trace formula of McKean-van Moerbeke-Trubowitz and the differential equation for the translation dynamics $V \mapsto V(\cdot+t)$ derived by Dubrovin \cite{Du} and Trubowitz \cite{Tr}. It is very important for our method that versions of these objects can be developed in the case of a general reflectionless potential; see the remarkable work by Craig \cite{Cr}. Due to these two-sided estimates we control these objects in the process of the rational approximation of the vector $\omega$. We want to mention here that there are well known two-sided estimates relating the gaps in the spectrum of Hill's equation and the Fourier coefficients $c(n)$ of the potential, see \cite{Ho, KaMi, P, Tk, Tr}. However, these estimates are \textbf{insufficient} to derive the convergence of the inverse problem solutions in the rational approximation process. We explain this issue in Section~\ref{sec.10}; see Remark~\ref{rem:9fortheorI}.

Let us now present the main results of the current paper. In addition to the Diophantine condition \eqref{eq:1PAI7-5-85a} we assume also that each component $\omega_j$ obeys
\begin{equation}
\label{eq:diophant}
\|n\omega_j\| \geq \frac{c}{|n|^{\beta}} \mbox{\ \ \ for all $n \in \mathbb{Z}\setminus \{0\}$},
\end{equation}
where $0 < c < 1 < \beta$ and $\| \cdot \|$ denotes $\dist (\cdot, \Z)$. Denote by $\mathcal{P}(\omega,\ve,\kappa_0)$ the set of the potentials
\begin{equation}\label{eq:1potentials}
V(x) = \sum_{n \in \zv} c(n) e^{2 \pi i x n\omega}\ , \quad x \in \mathbb{R}
\end{equation}
with $c(n)$ as in \eqref{eq:1-4}. For $j = 0,1$, let $V^{(j)} \in \mathcal{P}(\omega, \ve, \kappa_0)$, that is,
\begin{equation}\label{eq:1potentialsA}
V^{(j)}(x) = \sum_{n \in \zv} c^{(j)}(n) e^{2 \pi i x n\omega}\ , \quad x \in \mathbb{R},
\end{equation}
and define the distance
\begin{equation}\label{eq:1potentialsA1}
\rho (V^{(0)}, V^\one) := \sum_{n \in \zv} |c^{(0)}(n) - c^\one(n)|.
\end{equation}

Let $V \in \mathcal{P}(\omega,\ve,\kappa_0)$. Denote by $\mathcal{S} = \mathcal{S}_V$ the spectrum of the Schr\"odinger operator \eqref{eq:SLPAI1-1} and by $(E^-( k_m), E^+( k_m))$ the gaps; compare \eqref{eq:1spectrum}. Set $r_m = |m|^{-\nu-1}$, $m\in\zv$ and let $C_m$ be a circle of radius $r_m$. Consider the torus
\begin{equation}\label{eq:1torus}
\cT_V = \prod_{m \in \zv \setminus \{ 0 \} : E^-( k_m) < E^+( k_m)} C_m.
\end{equation}
Define on $\mathcal{T}_V$ the distance
\begin{equation}\label{eq:1torusdist}
d((\theta^\zero_m)_m, (\theta^\one_m)_m)=\sum_m |\theta^\zero_m - \theta^\one_m|.
\end{equation}
Let $Q(x)$ be a real bounded continuous function, $x \in \mathbb{R}$. Consider the Schr\"odinger operator
\begin{equation} \label{eq:SLPAI1-1Q}
[H_Q y](x) := -y''(x) +  Q(x) y(x) = E y(x), \quad x \in \IR.
\end{equation}
One says that $Q$ is isospectral with $V$ if the spectrum of $H_Q$ is the same as the spectrum of $H_V$, that is, $\mathcal{S}_Q = \mathcal{S}_V$. We denote by $\mathcal{ISO}(V)$ the set of all $Q$ isospectral with $V$.

\begin{remark}\label{rem:reflectionles}
$(1)$ Let $W$ be a continuous real function on the torus $\mathbb{T}^\nu$ and let $\omega' = (\omega'_1, \dots, \omega'_\nu) \in \mathbb{R}^\nu$ be a vector with rationally independent components. Consider the potential $Q(x)=W(x\omega')$. Assume that $Q$ is isospectral with $V$ and that $H_Q$ has purely absolutely continuous spectrum. What can one say about $W$ and $\omega'$? Theorem~$I$ below in particular gives a complete answer to this question.

$(2)$ As a matter of fact, the statement in Theorem~$I$ is stronger. It applies to any reflectionless potential.
Recall the definition, see ~\cite{Cr}. Consider the Green function $G(x, y; E + i \eta) = (H_Q - E - i \eta)^{-1}(x,y)$, $\eta > 0$. One says that $Q$ is reflectionless if
$$
\lim_{\eta \downarrow 0} \Re G(x, x; E + i \eta) = 0 \quad \text{for all $x$ and Lebesgue almost every } E \in \mathcal{S}_Q.
$$
Some work related to this notion can be found, for example, in \cite{Cr, K97, K08, K14, PR, R07, R11, R13, SY}. In particular, the following result is known. Assume that $Q$ is recurrent in the sense that it belongs to its own $\omega$-limit set and let $H_Q$ be as in \eqref{eq:SLPAI1-1Q} (this assumption clearly holds for the potentials discussed in part (1) of this remark). If the spectrum of $H_Q$ is purely absolutely continuous, then $Q$ is reflectionless \cite{R07}.
\end{remark}

\begin{thmi}
There exists $\ve^\one = \ve^\one(a_0,b_0,\kappa_0) > 0$ such that for $0 < \ve \le \ve^\one$, the following
statements hold for every $V \in \mathcal{P}(\omega, \ve, \kappa_0)$.

$(1)$ There is a homeomorphism $\Phi$ from $(\mathcal{T}_V,d)$ onto $(\mathcal{ISO}(V) \cap \mathcal{P}(\omega, \sqrt{2 \ve}, \kappa_0/2), \rho)$.

$(2)$ Let $Q(x)$ be a bounded continuous function, $x\in\mathbb{R}$. Assume $Q$ is reflectionless and isospectral with $V$. Then,
\begin{equation}\label{eq:1potentialsQ}
Q(x) = \sum_{n \in \zv} d(n) e^{2 \pi i x n\omega}\ , \quad x \in \mathbb{R},
\end{equation}
with
\begin{equation}\label{eq:1-4Q}
|d(n)| \le \sqrt{2 \ve} \exp \Big( -\frac{\kappa_0}{2} |n| \Big), \quad n \in \zv \setminus \{ 0 \}.
\end{equation}
In particular, let $W$ be a continuous real function on the torus $\mathbb{T}^\nu$, and
let $\omega'  \in \mathbb{R}^\nu$ be a vector with rationally independent components. Consider the potential $Q(x) = W(x \omega')$. Assume that $Q$ is isospectral with $V$. Assume also that $H_Q$ has purely absolutely continuous spectrum. Then, $Q$ can be represented in the form \eqref{eq:1potentialsQ}
 with \eqref{eq:1-4Q}.
\end{thmi}

\begin{remark}\label{rem:SY}
(a) Sodin and Yuditskii \cite{SY} studied the set of all reflectionless potentials for which the spectrum coincides with a given set $\mathcal{E}\subset \mathbb{R}$. A closed set
$$
\mathcal{E}=[\underline E, \infty)\setminus \bigcup_n (E^-_n,E^+_n)
$$
is called homogeneous if there is $\tau > 0$ such that for any $E \in \mathcal{E}$ and any $\sigma > 0$, we have
$|(E - \sigma, E + \sigma) \cap \mathcal{E}| > \tau \sigma$; see \cite{Ca}. Denote by $\mathcal{Q}_\mathcal{E}$ the set of all reflectionless potentials $Q$ with $\mathcal{S}_Q = \mathcal{E}$. Assume that the set $\mathcal{E}$ is homogeneous and
$$
\sum_n (E^+_n-E^-_n) < \infty.
$$
The following results were obtained in \cite{SY}. The set $\mathcal{Q}_\mathcal{E}$ is homeomorphic to the torus
\begin{equation}\label{eq:1torusSY}
\cT = \prod_{n: E^-_n < E^+_n} C_n.
\end{equation}
Furthermore, each $Q \in \mathcal{Q}_\mathcal{E}$ is almost-periodic. Gesztesy and Yuditskii proved in \cite{GY} that for $Q \in \mathcal{Q}_\mathcal{E}$, the operator $H_Q$ has purely absolutely continuous spectrum. They proved also that $\mathcal{Q}_\mathcal{E} \subset L^\infty(\mathbb{R})$.

One can prove using Theorems~$A$ and $B$ from \cite{DG} that the spectrum $\mathcal{S}_V$ in Theorem $I$ is homogeneous; see \cite{DGL14}. Thus, \cite{SY} and \cite{GY} apply to the set $\mathcal{S}_V$ in Theorem $I$. The results of Theorem $I$ cover all the mentioned results from \cite{SY} and \cite{GY} for $\mathcal{E} = \mathcal{S}_V$. For instance, the absolute continuity of the spectrum is due to the fact that this holds for any $Q \in \mathcal{P}(\omega, \ve, \kappa_0)$ with small $\ve$; see \cite{El}. On the other hand, the methods in \cite{SY} and \cite{GY} do not allow one to show that each potential  $Q \in \mathcal{Q}_\mathcal{E}$ has a quasi-periodic representation \eqref{eq:1potentialsQ} with the same vector of frequency $\omega$ and with Fourier coefficients which obey \eqref{eq:1-4Q}.
\\[1mm]
(b) Given the definition of $\mathcal{Q}_\mathcal{E}$ above, we can formulate part (1) of Theorem~$I$ as follows: There is a bijection $\Phi$ between $\mathcal{T}_V$ and $\mathcal{Q}_{\mathcal{S}_V}$, and the latter set is contained in $\mathcal{P}(\omega, \sqrt{2 \ve}, \kappa_0/2)$, and indeed coincides with $\mathcal{ISO}(V) \cap \mathcal{P}(\omega, \sqrt{2 \ve}, \kappa_0/2)$. Thus, $\mathcal{Q}_{\mathcal{S}_V}$ can be equipped with the metric $\rho$, and the map $\Phi$ is a homeomorphism with respect to the metrics $d$ and $\rho$.
\\[1mm]
(c) Theorem $I$ plays a key role in the description and uniqueness of the solutions of KdV equation with quasi-periodic analytic initial data like in the theorem; see the recent paper \cite{BDGL} by Binder et al.
\\[1mm]
(d) There is a very natural question about a sharper identification of the class of quasi-periodic isospectral potentials in question, in terms of the exponential decay rate of the coefficients.
Theorem $\tilde B$, stated below, and its counterpart Theorem B from \cite{DG} identify the rates within a margin comparable with the rate itself. Such an estimate is sufficient for applications
like existence and uniqueness of the solutions of the KdV equation, but one could ask whether the range of decay rates of the coefficients of two isospectral potentials can be shown to be smaller.
Of course the most intriguing question is whether sometimes two isospectral potentials
actually have different exponential rates of decay?
The approach we develop in this paper (and its predecessors) does not allow us to sharpen the margin considerably. It seems that some new ideas are needed to accomplish that goal. One possibility is to investigate some ``isospectral functionals'' and to develop a way to estimate the exponential convergence rate in terms of these functionals.
\end{remark}

As was mentioned above, the proof of Theorem~$I$ relies in a crucial way on periodic approximation. Next, we state the main results for the periodic approximants. Let $\tilde \omega = (\tilde \omega_1, \dots, \tilde \omega_\nu) \neq 0$ be a vector with rational components $\tilde \omega_j = \ell_j/t_j$, $\ell_j , t_j \in \mathbb{Z}$. Consider the function $\tilde V(x) = U(x \tilde \omega)$, $x \in \mathbb{R}$. The function $\tilde V(x)$ is obviously periodic. Consider the Hill operator
\begin{equation} \label{eq:PAI1-1}
[H_{\tilde V} y](x) = -y''(x) + \tilde V(x) y(x), \quad x \in \IR.
\end{equation}
We assume that the following ``Diophantine condition in the box'' holds, see Section~\ref{sec.5},
\begin{equation}\label{eq:PAI7-5-8}
|n \tilde \omega| \ge a_0 |n|^{-b_0}, \quad 0<|n|\le \bar R_0,
\end{equation}
for some
\begin{equation}\label{eq:PAIombasicTcondition}
0 < a_0 < 1,\quad \nu < b_0 < \infty,\quad (\bar R_0)^{b_0} > \prod t_j.
\end{equation}

\begin{remark}\label{rem:PAsectionplan1}
Clearly, for any $\tilde \omega = (\tilde \omega_1, \dots, \tilde \omega_\nu) \neq 0$, one can set up \eqref{eq:PAI7-5-8}, \eqref{eq:PAIombasicTcondition} with some $a_0 < 1, \nu < b_0 < \infty$ and $(\bar R_0)^{b_0} > \prod t_j$. In this paper we assume that $a_0, b_0$ are fixed while the denominators $t_j$ can be arbitrarily large. We use the same exponent $b_0$ in the Diophantine condition \eqref{eq:PAI7-5-8} and for $(\bar R_0)^{b_0} > \prod t_j$, just to save one piece of notation.
\end{remark}

To state the main results related to rational approximations we need the following:

\begin{defi}\label{def:omegalattice1}
Let $\mathfrak{N}(\tilde\omega) := \{ n \in \zv : n\tilde \omega = 0 \}$ and consider the quotient group $\mathfrak{Z}(\tilde \omega) := \mathbb{Z}^\nu/\mathfrak{N}(\tilde\omega)$. We call $\mathfrak{Z}(\tilde \omega)$ the $\tilde \omega$-lattice. We use the notation $[n]_{\tilde\omega} = [n]$ for the coset $n + \mathfrak{N} (\tilde\omega)$, $n \in \mathbb{Z}^\nu$. Given a set $\La \subset \zv$, we denote by $[\La]_{\tilde \omega} = [\La]$ the image of $\La$ under the map $n \rightarrow [n]_{\tilde \omega}$. We introduce the quotient distance in the standard way, that is, via $|\mathfrak{n}| = |\mathfrak{n}|_{\tilde \omega} := \min \{ |n| : n \in \mathfrak{n} \}$, $\mathfrak{n} \in \mathfrak{Z}(\tilde \omega)$. Given $\mathfrak{n} \in \mathfrak{Z}(\tilde \omega)$, we set $\mathfrak{n} \tilde \omega := n \tilde \omega$, where $n \in \mathfrak{n}$ is arbitrary. Obviously, this is a well-defined real function on $\mathfrak{Z}(\tilde \omega)$, which we denote by $\xi(\mathfrak{n})$.
\end{defi}

Here we assume that the Fourier coefficients $c(n)$ decay sub-exponentially:
\begin{equation}\label{eq:17-4-2}
\begin{split}
\overline{c(n)} & = c(-n), \quad n \in \zv \setminus \{ 0 \}, \\
|c(n)| & \le \ve \exp(-\kappa_0|n|^{\alpha_0}), \quad n \in \zv \setminus \{ 0 \},
\end{split}
\end{equation}
$0 < \kappa_0, \alpha_0 \le 1$. In all applications in this work we are interested only in the case when $U$ is analytic, that is, $\alpha_0 = 1$ in \eqref{eq:17-4-2}. We include the case $\alpha_0 < 1$ for purely technical reasons, which we explain in Section~\ref{sec.4}.

Set
\begin{equation}\label{eq:1K.1}
\begin{split}
k_n & = -\xi(n)/2, \quad n \in \mathfrak{Z}(\tilde\omega) \setminus \{0\}, \quad \mathcal{K}(\xi) = \{ k_n : n \in \mathfrak{Z}(\tilde\omega) \setminus \{0\} \}, \\
\mathfrak{J}_n & = ( k_n - \delta(n), k_n + \delta(n) ), \quad \delta(n) = a_0 (1 + |n|)^{-b_0-3}, \quad n \in \mathfrak{Z}(\tilde\omega) \setminus \{0\}, \\
\mathfrak{R}(k) & = \{ n \in \mathfrak{Z}(\tilde\omega) \setminus \{0\} : k \in \mathfrak{J}_n \}, \quad \mathfrak{G} = \{ k : |\mathfrak{R}(k)| < \infty \},
\end{split}
\end{equation}
where $a_0,b_0$ are as in the Diophantine condition \eqref{eq:PAI7-5-8}. Let $k \in \mathfrak{G}$ be such that $|\mathfrak{R}(k)| > 0$. Due to the Diophantine condition, one can enumerate the points of $\mathfrak{R}(k)$ as $n^{(\ell)}(k)$, $\ell = 0, \dots, \ell(k)$, $1 + \ell(k) = |\mathfrak{R}(k)|$, so that $|n^{(\ell)}(k)| < |n^{(\ell+1)}(k)|$. Set
\begin{equation}\label{eq:1mjdefi}
\begin{split}
T_{m}(n) & = m - n ,\quad m, n \in \mathfrak{Z}(\tilde\omega), \\
\mathfrak{m}^{(0)}(k) & = \{ 0, n^{(0)}(k) \}, \\
\mathfrak{m}^{(\ell)}(k) & = \mathfrak{m}^{(\ell-1)}(k) \cup T_{n^{(\ell)}(k)}(\mathfrak{m}^{(\ell-1)}(k)), \quad \ell = 1, \dots, \ell(k).
\end{split}
\end{equation}

\begin{thmta}
There exists $\ve_0 = \ve_0(\ka_0, a_0, b_0) > 0$ such that for $1/2 \le \alpha_0 \le 1$, $0 < \ve \le \ve_0$, and $k \in \mathfrak{G} \setminus \{ \frac{\xi(m)}{2} : m \in \mathfrak{Z}(\tilde\omega)\}$, there exist $E(k) \in \mathbb{R}$ and $\vp(k) := (\vp(n;k))_{n \in\mathfrak{Z}(\tilde\omega)}$ so that the following conditions hold:

$(a)$ $\vp(0 ;k) = 1$,
\begin{equation} \label{eq:1-17evdecay1A}
\begin{split}
|\vp(n; k)| & \le \ve^{1/2} \sum_{m \in \mathfrak{m}^{(\ell)}} \exp \Big( -\frac{7}{8} \kappa_0 |n-m|^{\alpha_0} \Big), \quad \text{ $n \notin \mathfrak{m}^{(\ell(k))}(k)$}, \\
|\vp(m; k)| & \le 2, \quad \text{for any $m \in \mathfrak{m}^{(\ell(k))}(k)$.}
\end{split}
\end{equation}

$(b)$ The function
$$
\psi(k, x) = \sum_{n \in \mathfrak{Z}(\tilde\omega)} \vp(n; k) e^{2 \pi i x (n \omega + k)}
$$
is well-defined and obeys equation \eqref{eq:PAI1-1} with $E = E(k)$, that is,
\begin{equation}\label{eq:1.sco}
H_{\tilde V} \psi(k,x) \equiv  - \psi''(k,x) + \tilde V(x) \psi(k,x) = E(k) \psi(k,x).
\end{equation}

$(c)$
$$
E(k) = E(-k), \quad \varphi(n;-k) = \overline{\varphi(-n; k)}, \quad \psi(-k, x) = \overline{\psi(k, x)},
$$
\begin{equation}\label{eq:1Ekk1EGT11}
\begin{split}
(k^0)^2 (k - k_1)^2  < E(k) - E(k_1) < 2k (k - k_1) + 2 \ve \sum_{k_1 < k_{n} < k} \delta(n), \quad 0 < k - k_1 < 1/4, \; k_1 > 0,
\end{split}
\end{equation}
where $k^\zero := \min(\ve_0, k/1024)$.

$(d)$ The limits
\begin{align}
\label{eq:1Ekm} E^\pm(k_m) & = \lim_{k \to k_m \pm 0, \; k \in \mathfrak{G} \setminus \{ \frac{\xi(m)}{2} : m \in \mathfrak{Z}(\tilde \omega)\}} E(k) \quad \text{ for $ k_m > 0$,} \\
\label{eq:1Ek0} E(0) & = \lim_{k \to 0 , \; k \in \mathfrak{G} \setminus \{\frac{\xi(m)}{2}:m\in \mathfrak{Z}(\tilde \omega)\}} E(k), \\
\label{eq:1phikm} \vp^\pm(n ;k_m) & = \lim_{k \to k_m \pm 0, \; k \in \mathfrak{G} \setminus \{\frac{\xi(m)}{2}:m\in \mathfrak{Z}(\tilde \omega)\}} \vp(n ;k) \quad \text{ for $k_m > 0$,} \\
\label{eq:1phik0} \vp(n ;0) & = \lim_{k \to 0, \; k \in \mathfrak{G} \setminus \{ \frac{\xi(m)}{2} : m \in \mathfrak{Z}(\tilde \omega)\}} \vp(n; k)
\end{align}
exist, $\vp^\pm(0; k_m) = 1$, $\vp(0 ;0) = 1$,
\begin{equation} \label{eq:1-17evdecay1Akmk0}
\begin{split}
|\vp^\pm(n ;k_m)| & \le \ve^{1/2} \sum_{m \in \mathfrak{m}^{(\ell)}} \exp \Big( -\frac{7}{8} \kappa_0 |n-m|^{\alpha_0} \Big), \quad \text{ $n \notin \mathfrak{m}^{(\ell(k_m))}(k_m)$}, \\
|\vp^\pm(n; k_m)| & \le 2 \quad \text{for any $n \in \mathfrak{m}^{(\ell(k_m))}(k_m)$,} \\
|\vp(n; 0)| & \le \ve^{1/2} \exp \Big( -\frac{7}{8} \kappa_0 |n|^{\alpha_0} \Big), \quad n \neq 0.
\end{split}
\end{equation}
The functions
\begin{align*}
\psi^\pm (k_m, x) & = \sum_{n \in \mathfrak{T}} \vp^\pm(n; k_m) e^{2 \pi i x (n \omega + k_m)} \\
\psi (0, x) & = \sum_{n \in \mathfrak{T}} \vp(n;0) e^{2 \pi i x n \omega}
\end{align*}
are well-defined and obey
\begin{equation}\label{eq:1.scokmk0}
\begin{split}
- \partial^2_{xx} \psi^\pm(k_m,x) + \tilde V(x) \psi^\pm(k_m,x) & = E^\pm(k_m) \psi^\pm(k_m,x),\\
- \partial^2_{xx} \psi(0,x) + \tilde V(x) \psi(0,x) & = E(0) \psi(0,x).
\end{split}
\end{equation}

$(e)$ The spectrum $\mathcal{S}_{\tilde V}$ of $H_{\tilde V}$ consists of the following set,
$$
\mathcal{S}_{\tilde V} = [E(0) , \infty) \setminus \bigcup_{m \in \mathfrak{Z}(\tilde \omega) \setminus \{ 0 \} : k_m > 0, \; E^-( k_m) < E^+( k_m)} (E^-( k_m), E^+( k_m)).
$$
\end{thmta}

\begin{thmtb}
$(1)$ The gaps $(E^-(k_m), E^+( k_m))$ in Theorem~ $\tilde A$ obey $E^+( k_m) - E^-(k_m) \le 2 \ve \exp(-\frac{\kappa_0}{2} |m|^{\alpha_0})$.

$(2)$ Using the notation from Theorem~$\tilde A$, there exists $\ve^\zero = \ve^\zero (a_0, b_0, \kappa_0) > 0$ such that if
the gaps $(E^-(k_m,\tilde\omega), E^+(k_m,\tilde\omega))$ obey $E^+(k_m,\tilde\omega) - E^-( k_m,\tilde\omega) \le \ve \exp(-\kappa'_0 |m|^{\alpha'_0})$ with $0 < \ve < \ve^\zero$, $\kappa'_0 \ge 4 \kappa_0$, $\alpha'_0 \ge \alpha_0$, then, in fact, $\tilde V(x)$ has a representation
\begin{equation}\label{eq:1potentialsVtil}
\tilde V(x) = \sum_{n \in \zv} d(n) e^{2 \pi i x n \tilde \omega}\ , \quad x \in \mathbb{R}
\end{equation}
with $|d(m)| \le \sqrt{2\ve} \exp(-\frac{\kappa'_0}{2} |m|^{\alpha'_0})$.
\end{thmtb}

We define the torus $\mathcal{T}_V$, the set $\mathcal{P}(\tilde\omega, \ve, \kappa)$ and the distance $\rho$ just as before.

\begin{thmti}
There exists $\ve^\one = \ve^\one(a_0,b_0,\kappa_0)$ such that for $0 < \ve < \ve^\one$ and $V \in \mathcal{P}(\tilde\omega, \ve, \kappa_0)$, there is a homeomorphism $\Phi$ from $(\mathcal{T}_V,d)$ onto $(\mathcal{ISO}(V) \cap \mathcal{P}(\tilde\omega, \sqrt{2\ve}, \kappa_0/2), \rho)$.
\end{thmti}

The structure of the paper is as follows. We will prove Theorem~$\tilde A$ in Section~\ref{sec.2}, Theorem~$\tilde B$ in Section~\ref{sec.PA4B}, and Theorem~$\tilde I$ in Section~\ref{sec.4}. After some preparatory work on periodic approximation in Sections~\ref{sec.5}--\ref{sec:9}, we prove Theorem~$I$ in Section~\ref{sec.10}.

\setcounter{section}{1}

\section{Proof of Theorem~$\tilde A$}\label{sec.2}

Let $\tilde \omega = (\tilde \omega_1, \dots, \tilde \omega_\nu) \neq 0$ be a vector with rational components, $\tilde \omega_j = \ell_j/t_j$, $\ell_j, t_j \in \mathbb{Z}$.

\begin{lemma}\label{lem:PA10omegalattice1}
$(1)$ The set $\mathfrak{F} (\tilde \omega) := \{ m \tilde \omega : m \in \mathbb{Z}^\nu \}$ is a discrete infinite subgroup of $\mathbb{R}$. \\
$(2)$ We have $\mathfrak{F} (\tilde \omega) = \{ a \tau_0 : a \in \mathbb{Z}\}$ for a suitable $\tau_0 \ge \bar T^{-1}$, $\bar T = \prod_j t_j$. \\
$(3)$ The set $\mathfrak{N} (\tilde \omega) := \{ m \in \mathbb{Z}^\nu : m \omega = 0 \}$ is a subgroup of $\mathbb{Z}^\nu$.
\end{lemma}

\begin{proof}
Clearly, $\mathfrak{F} (\tilde \omega)$ is an infinite subgroup of $\mathbb{R}$. Let $\bar T = \prod_j t_j$. Let $\xi = m \tilde \omega \neq 0$ be arbitrary. Then $\xi = \ell/\bar T$, where $\ell \in \mathbb{Z}$, $\ell \neq 0$. Hence, $|\xi| \ge 1/ \bar T$. Thus, $\tau_0 := \inf_{m \tilde \omega \neq 0} |m \tilde \omega| \ge \bar T^{-1}$. In particular, $\mathfrak{F} (\tilde \omega)$ is a discrete subgroup of $\mathbb{R}$. This validates part~(1). Part~(2) follows from part (1) and its proof.

Part~$(3)$ follows just from the definitions.
\end{proof}

Consider the quotient group $\mathfrak{Z} (\tilde \omega) := \mathbb{Z}^\nu/\mathfrak{N} (\tilde \omega)$. We call $\mathfrak{Z} (\tilde \omega)$ the $\tilde \omega$-lattice. We use the notation $[n]_{\tilde \omega} = [n]$ for the coset $n + \mathfrak{N}(\tilde\omega)$, $n \in \mathbb{Z}^\nu$. Given a set $\La \subset \zv$, we denote by $[\La]_{\tilde \omega} = [\La]$ the image of $\La$ under the map $n \rightarrow [n]_{\tilde \omega}$. We introduce the quotient distance in the standard way, that is, via $|\mathfrak{n}| = |\mathfrak{n}|_{\tilde \omega} := \min \{|n| : n \in \mathfrak{n} \}$, $\mathfrak{n} \in \mathfrak{Z}(\tilde \omega)$. Given $\mathfrak{n} \in \mathfrak{Z}(\tilde \omega)$, we set $\mathfrak{n} \tilde \omega := n \tilde \omega$, where $n \in \mathfrak{n}$ is arbitrary. Obviously, this is a well-defined real function on $\mathfrak{Z}(\tilde \omega)$.

\begin{lemma}\label{lem:PA10omegalattice1a}
$(1)$ $\mathfrak{F} (\tilde \omega) = \{ \mathfrak{m} \omega : \mathfrak{m} \in \mathfrak{Z}(\tilde \omega)\}$. \\
$(2)$ The map $\iota : \mathfrak{m} \mapsto \mathfrak{m} \tilde \omega \in \mathfrak{F} (\tilde \omega)$, $\mathfrak{m} \in \mathfrak{Z} (\tilde \omega)$ is a group isomorphism.
\end{lemma}

\begin{proof}
Part $(1)$ follows from the definition of the sets $\mathfrak{F} (\tilde \omega), \mathfrak{Z}(\tilde \omega)$. The map $\iota$ is clearly a homomorphism of Abelian groups. If $\iota(\mathfrak{m}) = 0$ then $m \tilde \omega = 0$ for $m \in \mathfrak{m}$, that is, $\mathfrak{m} = 0$. So, $\iota$ is injective. Due to $(1)$ it is also surjective.
\end{proof}

\begin{remark}\label{rem:PAtwogroups}
Although the groups $\mathfrak{F} (\tilde \omega), \mathfrak{Z} (\tilde \omega)$ are isomorphic, it is convenient to use both notions in what follows in this work. This is completely similar to the setup in the work \cite{DG}, the lattice  $\mathfrak{Z} (\tilde \omega)$ replaces $\zv$ in the setup of the current work. The multi-scale analysis developed for the case $\zv$ in \cite{DG} may be generalized to the group $\mathfrak{Z}(\tilde \omega)$, as shown in \cite{DGL}. We use the lattice $\mathfrak{Z}(\tilde \omega)$ and the quotient-distance on it to label and to estimate the Fourier coefficients. It proves to be crucial that the decay of the Fourier coefficients, which we establish in the study of the inverse spectral problem, is controlled via the quotient distance. To simplify the notation we suppress $\tilde \omega$ from $\mathfrak{Z}(\tilde \omega)$.
\end{remark}

Let $c(n) \in \mathbb{C}$, $n \in \zv$ obey
\begin{equation}\label{eq:17-4}
\begin{split}
\overline{c(n)} & = c(-n), \quad n \in \zv , \quad c(0) = 0, \\
|c(n)| & \le \ve \exp(-\kappa_0|n|^{\alpha_0}), \quad n \in \zv ,
\end{split}
\end{equation}
$0 < \kappa_0, \alpha_0 \le 1$. Consider the function
\begin{equation}\label{eq:2Vtilde}
\tilde V(x) = \sum_{n \in \zv} c(n) e^{2 \pi i xn\tilde \omega}\ , \quad x \in \mathbb{R}
\end{equation}
and the Hill operator
\begin{equation} \label{eq:17-1}
[\hat Hy](x) = -y''(x) + \tilde V(x) y(x) = E y(x), \quad x,\ve \in \IR.
\end{equation}

\begin{lemma}\label{lem:PA10omegalattice1tildeV}
We have $\tilde V(x + T) = \tilde V(x)$, $x \in \mathbb{R}$, with $T := T(\tilde \omega) := \tau_0^{-1}$.
\end{lemma}

\begin{proof}
For any $x$, we have
\begin{equation}\label{eq:PA17-4}
\tilde V(x + T) = \sum_{n \in \zv \setminus \{ 0 \}} c(n) e^{2 \pi i n\tilde \omega (x + \tau_0^{-1})} = \sum_{n \in \zv \setminus \{ 0 \}} c(n) e^{2 \pi i n\tilde \omega x} = \tilde V(x).
\end{equation}
\end{proof}

To label the Fourier coefficients via $\mathfrak{n} \in \mathfrak{Z}(\tilde \omega)$ we just add up all the coefficients labeled via $n \in \mathfrak{n}$. In Lemma~\ref{lem:PAL10omegV1} below we state the estimates for the corresponding sums. The derivation is elementary and standard and we omit it. Given $\mathfrak{n} \in \mathfrak{Z}$, set $\beta(\mathfrak{n}) = \beta_{\tilde \omega}(\mathfrak{n}) := \sum_{n \in \mathfrak{n}} |c(n)|$.

\begin{lemma}\label{lem:PAL10omegV1}
$(1)$ For any $1/2 \le \alpha_0 \le 1$, $0 < \kappa_0 \le 1$ and $R \ge 0$, we have
\begin{equation}\label{eq:PAexpsum-1}
\sum_{n \in \zv : |n| \ge R} \exp(-\kappa_0 |n|^{\alpha_0}) \le C(\kappa_0,\nu) \exp(-\kappa_0 R^{\alpha_0}/2).
\end{equation}

$(2)$ For any $\mathfrak{n}$ and any $b \ge 0$ and $1 \ge \eta > 0$, we have
$$
|\mathfrak{n}|^b \beta (\mathfrak{n})^{\eta} \le C(\kappa_0, b, \eta, \nu) \exp(-\eta \kappa_0 |\mathfrak{n}|^{\alpha_0}/2).
$$
In particular, $\beta(\mathfrak{n}) \le C(\kappa_0, \nu) \exp(-\kappa_0 |\mathfrak{n}|^{\alpha_0}/2)$.

$(3)$ For any $b > 1$, $R > 0$, we have
\begin{equation}\label{eq:PAexpsumomega1}
\sum_{|\mathfrak{n}| \ge R} |\mathfrak{n}|^b \beta (\mathfrak{n}) \le C(\kappa_0, b, \nu) \exp(-\kappa_0 R^{\alpha_0}/2).
\end{equation}
\end{lemma}

The estimate in the next lemma is used in subsequent estimates. In particular, it is required for the purposes of the multi-scale analysis in \cite{DGL}.

\begin{lemma}\label{lem:PA2Fvolume}
For any $R \ge 1$, we have
\begin{equation}\label{eq:PAexpsum-2}
| \{ \mathfrak{m} : |\mathfrak{m}| \le R \} | \le 4^\nu R^\nu.
\end{equation}
\end{lemma}

\begin{proof}
For each $\mathfrak{m}$, there exists $m_\mathfrak{m} \in \mathfrak{m}$ with $|m_\mathfrak{m}| = |\mathfrak{m}|$. Since $\mathfrak{m} \cap \mathfrak{m}' = \emptyset$ unless $\mathfrak{m} = \mathfrak{m}'$, we have
$$
| \{ \mathfrak{m} : |\mathfrak{m}| \le R \} | \le | \{ m \in \mathbb{Z}^\nu : |m| \le R \} | \le 2^\nu (R+1)^\nu \le 4^\nu R^\nu.
$$
\end{proof}

We rewrite the function $\tilde V$ as follows,
\begin{equation}\label{eq:PALVmodomega}
\begin{split}
\tilde V(x) & = \sum_{\mathfrak{n} \in \mathfrak{Z} \setminus \{ 0 \}} c(\mathfrak{n}) e^{2 \pi i \mathfrak{n} \tilde \omega x} + c([0]), \\
c(\mathfrak{n}) & := c_{\tilde \omega}(\mathfrak{n}) := \sum_{n \in \mathfrak{n}} c(n).
\end{split}
\end{equation}

Note that due to the last lemma,
\begin{equation}\label{eq:PAexpsumomega1P2}
\begin{split}
|c(\mathfrak{n})| & \le C(\kappa_0, \nu) \exp(-\kappa_0 |\mathfrak{n}|^{\alpha_0}/2), \\
\sum_{|\mathfrak{n}| \ge R} |\mathfrak{n}|^b |c(\mathfrak{n})| & \le C(\kappa_0, b, \nu) \exp(-\kappa_0 R^{\alpha_0}/2).
\end{split}
\end{equation}
In particular, the series \eqref{eq:PALVmodomega} converges.

Since the constant $c([0])$ can be incorporated in the spectral parameter we assume for convenience that $c([0])=0$.
Set
\begin{equation} \label{eq:2-7}
\begin{split}
\tilde h(\mathfrak{m}, \mathfrak{n}; k) & = (2\pi)^2 (\xi(\mathfrak{m}) + k)^2 \quad \text{if } \mathfrak{m} = \mathfrak{n}, \\
\tilde h(\mathfrak{m}, \mathfrak{n}; k) & = \ve \tilde c(\mathfrak{n}-\mathfrak{m}) \quad \text{if } \mathfrak{m} \neq \mathfrak{n}.
\end{split}
\end{equation}
We call the operators $\tilde H_{ k} = \tilde H_{k,\ve} := \bigl( \tilde h(\mathfrak{m}, \mathfrak{n}; k) \bigr)_{\mathfrak{m}, \mathfrak{n} \in \mathfrak{Z}}$ the operators dual to the Hill operator $H_{\tilde \omega}$; see \eqref{eq:17-1FtransH}. Now we are going to invoke the main results of \cite{DGL} related to the operators $\tilde H_{k}$ and use them to prove Theorem~$\tilde A$. We assume that the following ``Diophantine condition in the box'' holds,
\begin{equation}\label{eq:2PAI7-5-8}
|n \tilde \omega| \ge a_0 |n|^{-b_0}, \quad 0 < |n| \le \bar R_0,
\end{equation}
for some
\begin{equation}\label{eq:2PAIombasicTcondition}
0 < a_0 < 1,\quad \nu < b_0 < \infty, \quad (\bar R_0)^{b_0} > \prod t_j.
\end{equation}

\begin{lemma}\label{lem:PAL10omdiophsubst}
For any $\mathfrak{n} \neq 0$, we have
\begin{equation}\label{eq:7-5-8lattices}
|\mathfrak{n} \tilde \omega| \ge a_0 |\mathfrak{n}|^{-b_0}.
\end{equation}
\end{lemma}

\begin{proof}
Let $0 < |\mathfrak{n}| \le \bar R_0$ be arbitrary. There exists $n \in \mathfrak{n}$ such that $|n| = |\mathfrak{n}|$. Since $n \tilde \omega = \mathfrak{n} \tilde \omega$, \eqref{eq:2PAI7-5-8} implies
\begin{equation}\label{eq:7-5-8lattice}
|\mathfrak{n} \tilde \omega| \ge a_0 |\mathfrak{n}|^{-b_0} \quad \text{ if } 0 < |\mathfrak{n}| \le \bar R_0.
\end{equation}
Let $\tau_0 := \inf_{m \tilde \omega \neq 0} |m \tilde \omega| = \inf_{\mathfrak{m} \neq 0} |\mathfrak{m} \tilde \omega|$ as in Lemma~\ref{lem:PA10omegalattice1}. Due to Lemma~\ref{lem:PA10omegalattice1} and condition \eqref{eq:2PAIombasicTcondition}, we have $\tau_0 \ge \bar R_0^{-b_0}$. This implies
\begin{equation}\label{eq:PAombasictaulat}
|\mathfrak{n} \tilde \omega| \ge a_0 |\mathfrak{n}|^{-b_0} \quad \text { if } |\mathfrak{n}| > \bar R_0.
\end{equation}
By \eqref{eq:7-5-8lattice} and \eqref{eq:PAombasictaulat}, \eqref{eq:7-5-8lattices} holds for any $\mathfrak{n} \neq 0$.
\end{proof}

Thus, all conditions needed for Theorem~$\tilde C$ in \cite{DGL} hold with $\mathfrak{T} = \mathfrak{Z} (\tilde \omega)$:

\begin{thmtc}
There exists $\ve_0 = \ve_0(\ka_0, a_0, b_0) > 0$ such that for $0 < \ve \le \ve_0$ and any $k \in \mathfrak{G} \setminus \{ \frac{\xi(m)}{2} : m \in \mathfrak{Z}(\tilde\omega) \}$, there exist $E(k) \in \mathbb{R}$ and $\vp(k) := (\vp(n;k))_{n \in \zv}$ such that the following conditions hold:

$(1)$ $\vp(0 ;k) = 1$,
\begin{equation} \label{eq:1-17evdecay1}
\begin{split}
|\vp(n;k)| & \le \ve^{1/2} \sum_{m \in \mathfrak{m}^{(\ell)}} \exp \Big( -\frac{7}{8} \kappa_0 |n-m|^{\alpha_0} \Big), \quad \text{ $n \notin \mathfrak{m}^{(\ell(k))}(k)$}, \\
|\vp(m;k)| & \le 2, \quad \text{for any $m \in \mathfrak{m}^{(\ell(k))}(k)$,}
\end{split}
\end{equation}
\begin{equation} \label{eq:1philim}
\tilde H_{k} \vp(k) = E(k) \vp(k).
\end{equation}

$(2)$
\begin{equation}\label{eq:1EsymmetryT}
E(k) = E(-k), \quad \vp(n ;-k) = \overline{\vp(-n ;k)},
\end{equation}
\begin{equation}\label{eq:1Ekk1EGT}
(k^\zero)^2 (k - k_1)^2  < E(k) - E(k_1) < 2k (k - k_1) + 2 \ve \sum_{k_1 < k_{n} < k}(\delta(n))^{1/8} , \quad \quad 0 < k - k_1 < 1/4, \; k_1 > 0,
\end{equation}
where $k^\zero := \min(\ve_0, k/1024)$.

$(3)$ The limits
\begin{align}
\label{eq:1Ekm-2}
E^\pm(k_m) & = \lim_{k \to k_m \pm 0, \; k \in \mathfrak{G} \setminus \{\frac{\xi(m)}{2}:m\in \mathfrak{T}\}} E(k), \quad \text{ for $k_m>0$,} \\
\label{eq:1Ek0-2} E(0) & = \lim_{k \to 0 , \; k \in \mathfrak{G} \setminus \{\frac{\xi(m)}{2}:m\in \mathfrak{T}\}} E(k), \\
\label{eq:1phikm-2} \vp^\pm(n ;k_m) & = \lim_{k \to k_m \pm 0, \; k \in \mathfrak{G} \setminus \{\frac{\xi(m)}{2}:m\in \mathfrak{T}\}} \vp(n ;k), \quad \text{ for $k_m>0$,} \\
\label{eq:1phik0-2} \vp(n ;0) & = \lim_{k \to 0, \; k \in \mathfrak{G} \setminus \{\frac{\xi(m)}{2}:m\in \mathfrak{T}\}} \vp(n ;k)
\end{align}
exist,
$\vp^\pm(0 ;k) = 1$, $\vp(0 ;0) = 1$,
\begin{equation} \label{eq:1-17evdecay1Akmk0-2}
\begin{split}
|\vp^\pm(n; k_m)| & \le \ve^{1/2} \sum_{m \in \mathfrak{m}^{(\ell)}} \exp \Big( -\frac{7}{8} \kappa_0 |n-m|^{\alpha_0} \Big), \quad \text{ $n \notin \mathfrak{m}^{(\ell(k_m))}(k_m)$}, \\
|\vp^\pm(m; k_m)| & \le 2, \quad \text{for any $m \in \mathfrak{m}^{(\ell(k_m))}(k_m)$,} \\
|\vp(n ;0)| & \le \ve^{1/2} \exp \Big( -\frac{7}{8} \kappa_0 |n-m|^{\alpha_0} \Big), \quad n \neq 0,
\end{split}
\end{equation}
\begin{equation} \label{eq:1philimk0km-2}
\begin{split}
\tilde H_{k_m} \vp^\pm(k_m) & = E^\pm(k_m) \vp^\pm(k_m),\\
\tilde H_{0} \vp(0) & = E(0) \vp(0).
\end{split}
\end{equation}

$(4)$ Assume $E^-(k_{n^\zero}) < E^+(k_{n^\zero})$. Let $E \in (E^-(k_{n^\zero}) + \delta, E^+(k_{n^\zero}) - \delta)$ with $\delta > 0$ arbitrary. Then for every $k$, we have
\begin{equation}\label{eq:11Hinvestimatestatement1PQreprep2D}
|[(E - \tilde H_{k})^{-1}](m,n)| \le \begin{cases} \exp(-\frac{1}{8} \kappa_0 |m-n|^{\alpha_0}) & \text{if $|m-n| > [16 \log \delta^{-1}]^{1/\alpha_0}$}, \\ \delta^{-1} & \text{for any $m,n$.} \end{cases}
\end{equation}

$(5)$ Let $E < E(0) - \delta$, $\delta > 0$. For every $k$, we have
\begin{equation}\label{eq:11Hinvestimatestatementkzero}
|[(E - \tilde H_{k})^{-1}](m,n)| \le \begin{cases} \exp(-\frac{1}{8} \kappa_0 |m-n|^{\alpha_0}) & \text{if $|m-n| > [16 \log \delta^{-1}]^{1/\alpha_0}$}, \\ \delta^{-1} & \text{for any $m,n$.} \end{cases}
\end{equation}
\end{thmtc}

The relation between the Hill operator $\tilde H$ and the dual operators $\tilde H_k$ is via the Fourier transform. We denote by $\widehat{f}(k)$ the Fourier transform of a function $f(x)$,
\begin{equation}\label{eq:11Fourier}
\widehat{f}(k) := \int_\mathbb{R} e^{-2 \pi i k x} f(x) \, dx,
\end{equation}
$x, k \in \IR$. Let $\cS(\IR)$ be the space of Schwartz functions $f(x)$, $x \in \IR $. Let $g(k)$ be a measurable function which decays faster than $|k|^{-a}$ as $|k| \rightarrow \infty$ for any $a > 0$. Let $\psi = \check{g}$ be its inverse Fourier transform. Then $\psi$ belongs to the domain of $\tilde H$ and the following identity holds:
\begin{equation} \label{eq:17-1FtransH}
\widehat{\tilde H\psi}(k) = (2 \pi)^2k^2 \widehat{\psi}(k) + \sum_{\mathfrak{m} \in \mathfrak{Z} (\tilde \omega) \setminus \{ 0 \}} c(-\mathfrak{m}) \widehat{\psi} (k + \mathfrak{m} \omega).
\end{equation}
In particular, this identity holds for any $f \in \cS(\IR)$. We rewrite \eqref{eq:17-1FtransH} as follows,
\begin{equation} \label{eq:17-1FtransHdual}
\widehat{\tilde H \psi}(k + \mathfrak{n} \tilde \omega) =  (2 \pi)^2 k^2 \widehat{\psi}(k + \mathfrak{n} \tilde \omega) + \sum_{\mathfrak{m} \in \mathfrak{Z} (\tilde \omega) \setminus \{ n \}} c(\mathfrak{n} - \mathfrak{m}) \widehat{\psi} (k + \mathfrak{m} \omega) = [\tilde H_k \widehat{\psi} (k + \cdot \omega)](\mathfrak{n});
\end{equation}
see \eqref{eq:2-7}. One has for any $\mathfrak{m}, \mathfrak{n}, \mathfrak{l} \in \mathfrak{Z} = \mathfrak{Z}(\tilde \omega)$,
\begin{equation} \label{eq:12cocyclic}
\tilde H_{k + \mathfrak{l} \tilde \omega}(\mathfrak{m}, \mathfrak{n}) = \tilde H_{k}(\mathfrak{m} + \mathfrak{l}, \mathfrak{n} + \mathfrak{l}).
\end{equation}
Let $\mathfrak{l} \in \mathfrak{Z}$ be arbitrary. Define $U_\mathfrak{l} : \ell^2(\mathfrak{Z}) \to \ell^2(\mathfrak{Z})$ by setting $[U_\mathfrak{l} f](\mathfrak{n}) := f(\mathfrak{l} + \mathfrak{n})$, $f \in \ell^2(\mathfrak{Z})$.

\begin{lemma}\label{lem:13.1UINV}
$(1)$ $U_\mathfrak{l}$ is a unitary operator. \\
$(2)$ For any linear operator $T$ whose domain $\mathcal{D}_T$ contains the standard basis vectors $\delta_{\mathfrak{m}}$, we have
\begin{equation} \label{eq:12cocyclic11}
[U_\mathfrak{l} T U^{-1}_\mathfrak{l}](\mathfrak{m}, \mathfrak{n}) = T(\mathfrak{m} + \mathfrak{l}, \mathfrak{n} + \mathfrak{l}).
\end{equation}

$(3)$
\begin{equation} \label{eq:12cocyclic12}
[U_\mathfrak{l} H_{k} U^{-1}_\mathfrak{l}] (\mathfrak{m}, \mathfrak{n}) = H_{k + \mathfrak{l} \tilde \omega} (\mathfrak{m},\mathfrak{n}).
\end{equation}
In particular, if for some $E \in \IC$, the inverse $(E - H_{k})^{-1}$ exists, then for any $\mathfrak{l} \in \mathfrak{Z}$, the inverse $(E - H_{k + \mathfrak{l} \tilde \omega})^{-1}$ also exists and
\begin{equation} \label{eq:12cocyclic13}
(E - H_{k + \mathfrak{l} \tilde \omega})^{-1} (\mathfrak{m}, \mathfrak{n}) = (E - H_{k})^{-1} (\mathfrak{m} + \mathfrak{l}, \mathfrak{n} + \mathfrak{l}).
\end{equation}
\end{lemma}

\begin{proof}
Parts $(1)$ and $(2)$ follow from the definition of $U_\mathfrak{l}$. Part $(3)$ follows from part $(2)$ due to \eqref{eq:12cocyclic}.
\end{proof}

To prove Theorem~$\tilde A$, we need the following:

\begin{lemma}\label{lem:13.1}
Assume that for some $E \in \IR$, there exist $\gamma(E) > 0$, $B(E) < \infty$ such that for any $k$, $(E - H_{k})$ is invertible, and for any $\mathfrak{m}, \mathfrak{n} \in \mathfrak{Z}$, we have
\begin{equation} \label{eq:13-7}
|(E - \tilde H_{k})^{-1}(\mathfrak{m},\mathfrak{n})| \le B(E) \exp(-\gamma(E) |\mathfrak{m}-\mathfrak{n}|^{\alpha_0}),
\end{equation}
where $\alpha_0 > 0$. Then, $(E - \tilde H)$ is invertible.
\end{lemma}

\begin{proof}
For any $k$, we have
\begin{equation} \label{eq:13-a}
\begin{split}
\sum_{\mathfrak{l} \in \mathfrak{Z}} |(E - \tilde H_{k})(\mathfrak{m},\mathfrak{l})| |(E - \tilde H_{k})^{-1}(\mathfrak{l},\mathfrak{n})| & \le (k + \mathfrak{m}\tilde \omega)^2 B(E)\exp(-\gamma(E) |\mathfrak{m}-\mathfrak{n}|^{\alpha_0}) \\
& \quad + \ve_0 B(E) \exp(-\gamma_1(E) |\mathfrak{m} - \mathfrak{n}|^{\alpha_0}) \left[\sum_{\mathfrak{r} \in \mathfrak{Z}} \exp(-\gamma_1(E) |\mathfrak{r}|^{\alpha_0}) \right]^2 \\
& \le [(k + \mathfrak{m} \tilde \omega)^2 B(E) + C(\gamma(E), \alpha_0, \nu)] \exp(-\gamma(E) |\mathfrak{m} - \mathfrak{n}|^{\alpha_0}/2), \\
\sum_{\mathfrak{l} \in \mathfrak{Z}} (E - \tilde H_{k})(\mathfrak{m}, \mathfrak{l}) (E - \tilde H_{k})^{-1} (\mathfrak{l}, \mathfrak{n}) & = \delta_{\mathfrak{m}, \mathfrak{n}},
\end{split}
\end{equation}
where $\delta_{\mathfrak{m}, \mathfrak{n}}$ is the Kronecker symbol. In particular, for any $k$ and for any bounded function $\psi : \mathfrak{Z} \to \IC$, we have
\begin{equation} \label{eq:13-b}
\sum_{\mathfrak{l}, \mathfrak{n} \in \mathfrak{Z}} (E - \tilde H_{k})(\mathfrak{m}, \mathfrak{l}) (E - \tilde H_{k})^{-1}(\mathfrak{l}, \mathfrak{n}) \psi(\mathfrak{n}) = \psi(\mathfrak{m}), \quad \mathfrak{m} \in \mathfrak{Z},
\end{equation}
and the series converges absolutely.

Let $f \in \cS(\IR^1)$ be arbitrary. Set
\begin{equation} \label{eq:13-8}
g(k) = \sum_{\mathfrak{n}\in \mathfrak{Z}} (E - \tilde H_{k})^{-1}(0,\mathfrak{n}) \widehat{f} (k + \mathfrak{n} \tilde \omega).
\end{equation}
Note that since $\widehat{f} \in \cS(\IR)$, one obtains using condition \eqref{eq:13-7}, $\lim_{|k| \to \infty} |k|^a |g(k)| = 0$ for any $a > 0$. Therefore, \eqref{eq:17-1FtransH} holds, $\psi := \check{g} \in L^2(\IR)$. Using \eqref{eq:12cocyclic13} from Lemma~\ref{lem:13.1UINV}, one obtains for any $k$ and any $\mathfrak{m} \in \mathfrak{Z}$,
\begin{equation} \label{eq:13-8b}
\begin{split}
g(k + \mathfrak{m}\tilde \omega) & = \sum_{\mathfrak{n} \in \mathfrak{Z}} (E - \tilde H_{k + \mathfrak{m}\tilde \omega})^{-1}(0,\mathfrak{n}) \widehat{f} (k + \mathfrak{m}\tilde \omega + \mathfrak{n}\tilde \omega) \\
& = \sum_{\mathfrak{n} \in \mathfrak{Z}} (E - \tilde H_{ k})^{-1} (\mathfrak{m}, \mathfrak{n} + \mathfrak{m}) \widehat{f} (k + \mathfrak{m} \tilde \omega + \mathfrak{n} \tilde \omega) \\
& = \sum_{\mathfrak{n}\in \mathfrak{Z}} (E -\tilde H_{k})^{-1} (\mathfrak{m}, \mathfrak{n}) \widehat{f} (k + \mathfrak{n} \tilde \omega).
\end{split}
\end{equation}
Combining this with \eqref{eq:13-b}, one obtains for any $k$,
\begin{equation} \label{eq:13-8b-2}
\sum_{\mathfrak{m} \in \mathfrak{Z}} (E - \tilde H_{k})(0,\mathfrak{m}) g(k + \mathfrak{m}\tilde \omega) = \widehat{f}(k).
\end{equation}
Using the definition \eqref{eq:13-8} and \eqref{eq:13-a}, we get
\begin{equation} \label{eq:13-9}
\begin{split}
\Big| \int_\IR g(k) h(k) \, dk \Big| & \le \| \widehat{f} \|_{2} \| h \|_2\sum_{\mathfrak{n}\in \mathfrak{Z}} |(E - \tilde H_{k})^{-1}(0,\mathfrak{n})| \\
& \le \|\widehat{f} \|_{2} \| h \|_2 B(E) C(\gamma(E),\alpha_0,\nu)
\end{split}
\end{equation}
for any $h \in L^2(\IR)$. Hence, $\| g \|_2 \le B(E) C(\gamma(E),\alpha_0,\nu) \| \widehat{f} \|_{2}$. In particular, the inverse Fourier transform $\psi = \check{g} \in L^2(\IR)$ obeys $\| \psi \|_2 = \| g \|_2 \le M(E,\alpha_0,\nu) \|f\|_2$. Combining \eqref{eq:17-1FtransH} with \eqref{eq:13-8b-2}, we find
\begin{equation} \label{eq:17-1FtransHa}
\begin{split}
\widehat{[(E - \tilde H) \psi]}(k) & = E- (2 \pi)^2k^2 \widehat{\psi}(k) -\sum_{\mathfrak{m} \in  \mathfrak{Z}} c(-\mathfrak{m}) \widehat{\psi} (k+\mathfrak{m}\omega) \\
& = E - (2 \pi)^2k^2 g(k) - \sum_{\mathfrak{m} \in  \mathfrak{Z}} c(-\mathfrak{m}) g(k + \mathfrak{m} \omega) \\
& = \sum_{\mathfrak{m} \in \mathfrak{Z}} (E - \tilde H_{k})(0,\mathfrak{m}) g(k + \mathfrak{m} \tilde \omega) \\
& = \widehat{f}(k).
\end{split}
\end{equation}
So,
\begin{equation} \label{eq:17-1FtransHaIn}
(E - \tilde H) \psi = f, \quad \| \psi \|_2 \le M(E,\alpha_0,\nu) \| f \|_2.
\end{equation}
Since $f \in \cS(\IR)$ is arbitrary, $(E - \tilde H)$ is invertible.
\end{proof}

\begin{proof}[Proof of Theorem~$\tilde A$]
Given $k \in \IR$ and $\vp(\mathfrak{n}) : \mathfrak{Z} \rightarrow \mathbb{C}$ such that $|\varphi(\mathfrak{n})| \le C_\varphi |\mathfrak{n}|^{-\nu-1}$, where $C_\varphi$ is a constant, set
\begin{equation} \label{eq:11-5}
y_{\vp, k}(x) = \sum_{\mathfrak{n} \in \mathfrak{Z}}\, \vp(\mathfrak{n}) e^{2 \pi i (\mathfrak{n}\tilde \omega + k)x}.
\end{equation}
The function $y_{\vp, k}(x)$ satisfies the Hill equation \eqref{eq:17-1} if and only if
\begin{equation} \label{eq:111-6}
(2 \pi)^2 (\mathfrak{n}\tilde \omega + k)^2 \vp(n) + \sum_{\mathfrak{m} \in \mathfrak{Z}}\, c(\mathfrak{n} - \mathfrak{m}) \vp(\mathfrak{m}) = E \vp(\mathfrak{n})
\end{equation}
for any $\mathfrak{n} \in \mathfrak{Z}$. Let $E(k)$ and $(\vp(\mathfrak{n}; k))_{\mathfrak{n} \in \mathfrak{Z}}$ be as in Theorem~$\tilde C$. Then,
$$
\psi(k, x) = \sum_{\mathfrak{n} \in \mathfrak{Z}} \vp(\mathfrak{n}; k) e^{2 \pi i (\mathfrak{n}\tilde \omega + k)x}
$$
obeys equation \eqref{eq:17-1} with $E = E(k)$, that is,
\begin{equation}\label{eq:11.sco}
\tilde H \psi \equiv - \psi''(k,x) + \tilde V(x) \psi(k,x) = E(k) \psi(k, x).
\end{equation}
Due to $(1)$ and $(2)$ from Theorem~$\tilde C$, conditions $(a)$--$(d)$ in Theorem~$\tilde A$ hold. Combining $(4)$, $(5)$ from Theorem~$\tilde C$ with Lemma~\ref{lem:13.1}, one concludes that
$$
\spec H \subset  [E(0), \infty) \setminus \bigcup_{\mathfrak{m} \in \mathfrak{Z} \setminus \{0\} : E^-(k_\mathfrak{m}) < E^+(k_\mathfrak{m})} (E^-(k_\mathfrak{m}), E^+( k_\mathfrak{m})).
$$
Recall the following general basic fact from the spectral theory of Schr\"odinger equations: if for some $E \in \mathbb{R}$, there exists a bounded smooth function $y\neq 0$ which obeys equation \eqref{eq:17-1}, that is,
\begin{equation}\label{eq:11.sco-2}
- y'' + \tilde V(x) y(x) = E y(x),
\end{equation}
then $E \in \spec \tilde H$. For any $k \in \mathcal{G} \setminus \mathcal{K} (\omega)$, the function $\psi(k,x)$ is bounded. Hence, $E(k) \in \spec H$. It follows from \eqref{eq:1Ekk1EGT11} that $E(k)$ is continuous at each point $k \in \mathcal{G} \setminus \mathcal{K} (\omega)$. It follows also from \eqref{eq:1Ekk1EGT11} that $E(k)$ is monotone for $k \in \mathcal{G} \setminus \mathcal{K} (\omega)$, $k > 0$. Recall also that $E(-k) = E(k)$. Hence,
$$
\overline{\{ E(k) : k \in \mathcal{G} \setminus \mathcal{K}(\omega) \}} =[E(0), \infty) \setminus \bigcup_{\mathfrak{m} \in \mathfrak{Z} \setminus \{0\} : E^-(k_\mathfrak{m}) < E^+(k_\mathfrak{m})} (E^-(k_\mathfrak{m}), E^+( k_\mathfrak{m})),
$$
which implies
$$
\spec H =[E(0), \infty) \setminus \bigcup_{\mathfrak{m} \in \mathfrak{Z} \setminus \{0\} : E^-(k_\mathfrak{m}) < E^+(k_\mathfrak{m})} (E^-(k_\mathfrak{m}), E^+( k_\mathfrak{m})).
$$
This finishes the proof of Theorem~$\tilde A$.
\end{proof}

\section{Proof of Theorem~$\tilde B$}\label{sec.PA4B}

To prove Theorem~$\tilde B$ we need to invoke Theorem~$\tilde D$ from \cite{DGL} which contains a very detailed description of the results for the multi-scale analysis of the operators $\tilde H_k$. First of all we need to invoke the definition of the weight functions, see Section~2 in \cite{DGL}. These definitions are rather lengthy. We still reproduce them here because we use the structure of the weight functions in our proof of Theorem~$\tilde B$.

As always in the present paper, the role of the group $\mathfrak{T}$ in \cite{DGL} is played by $\mathfrak{Z}(\tilde\omega)$ here. Since we quote some parts from \cite{DGL}, let us nevertheless write $\mathfrak{T}$ for the group in this section. Moreover, to simplify the notation (and also to conform with \cite{DGL}) we use the Latin letters $m, n, \ldots$ to denote elements of the group $\mathfrak{T}$, rather than $\mathfrak{m}, \mathfrak{n}, \ldots$ as before.

\begin{defi}\label{def:aux1}
\begin{itemize}

\item[(1)] For each $m \in \mathfrak{T}$, let $\gamma(m) := (m)$ be the sequence which consists of one point $m$. Set $\Ga(m,m; 1) := \{ \gamma(m) \}$, $\Ga(m,n;1) := \emptyset$ for $n \not= m$. Let $\La \subset \mathfrak{T}$. Set
\begin{equation}\label{eq:auxtraject}
\begin{split}
\Ga(k,\La) & = \{ \gamma = (n_1,\dots ,n_k) : n_j \in \La, \quad n_{j+1} \neq n_j \}, \; k \ge 2, \\
\Ga(m,n;k,\La) & = \{ \gamma \in \Ga(k, \La), \quad n_1 = m, n_k = n \}, \; m, n \in \La, \; k \ge 2, \\
\Ga_1(m,n;\La) & = \bigcup_{k \ge 1} \Ga(m, n; k, \La), \quad \Ga_1(\La) = \bigcup_{m, n \in \La} \Ga_1(m, n; \La).
\end{split}
\end{equation}
We call the sequences $\gamma$ in this definition trajectories.

\item[(2)] Let $\La \subset \mathfrak{T}$, $w(m,n)$, $D(m)$ be real functions, $m, n \in \La$, obeying $w(m,n) \ge 0$, $D(m) \ge 1$, $w(m,m) = 1$, and
\begin{equation}\label{eq:2.weighdecaycond}
w(m,n) \le \exp(-\kappa_0 |m-n|^{\alpha_0}),
\end{equation}
$m, n \in \La$, where $0 < \kappa_0 < 1$, $0 < \alpha_0 \le 1$. For $\gamma = (n_1, \dots, n_k)$, set
\begin{equation}\label{eq:auxtrajectweightO}
\begin{split}
w_{D, \kappa_0} (\gamma) & = \Big[ \prod_{1 \le j \le k-1} w(n_j,n_{j+1}) \Big] \exp \Big( \sum_{1 \le j \le k} D(n_j) \Big), \\
\|\gamma\| & = \sum_{1 \le i \le k-1} |n_i - n_{i+1}|^{\alpha_0}, \quad \bar D(\gamma) := \max_j D(n_j), \\
W_{D,\kappa_0} (\gamma) & = \exp \Big( -\kappa_0 \|\gamma\| + \sum_{1 \le j \le k} D(n_j) \Big).
\end{split}
\end{equation}
Here, $\|\gamma\| = 0$ if $k = 1$. Obviously, $w_{D,\kappa_0} (\gamma)\le W_{D,\kappa_0} (\gamma)$.

\item[(3)] Let $T \gg 1$. We say that $\gamma = (n_1, \dots, n_{k})$, $n_j \in \La$, $k \ge 1$ belongs to $\Ga_{D, T, \kappa_0} (n_1, n_{k}; k, \La)$ if the following condition holds:
\begin{equation}\label{eq:auxtrajectweight5}
\min (D(n_{i}), D(n_{j})) \le T \| (n_{i}, \dots, n_{j}) \|^{\alpha_0/5} \quad \text{for any $i < j$ such that $\min (D(n_{i}), D(n_{j})) \ge 4 T \kappa_0^{-1}$}.
\end{equation}
Note that $\Ga_{D, T, \kappa_0} (n_1, n_{1}; 1, \La) = \{ (n_1) \}$. Set $\Ga_{D, T, \kappa_0} (m, n; \La) = \bigcup_k \Ga_{D, T, \kappa_0} (m, n; k, \La)$, $\Ga_{D, T, \kappa_0} (\La) = \bigcup_{m,n} \Ga_{D, T, \kappa_0} (m, n; \La)$.

\item[(4)] Set
\begin{equation}\label{eq:auxtrajectweight1}
\begin{split}
s_{D, T, \kappa_0; k, \La} (m, n) & = \sum_{\gamma \in \Ga_{D, T, \kappa_0} (m, n; k, \La)} w_{D, \kappa_0} (\gamma),\\
S_{D, T, \kappa_0; k, \La} (m, n) & = \sum_{\gamma \in \Ga_{D, T, \kappa_0} (m, n; k, \La)} W_{D, \kappa_0} (\gamma).
\end{split}
\end{equation}
Note that $s_{D, T, \kappa_0; 1, \La} (m, m) = S_{D, T, \kappa_0; 1, \La} (m, m)=\exp(D(m))$.

\item[(5)] Set $\mu_{\La}(m) := \dist (m,\mathfrak{T} \setminus \La)$. We say that the function $D(m)$, $m \in \La$ belongs to $\mathcal{G}_{\La, T, \kappa_0}$ if the following condition holds:
\begin{equation}\label{eq:auxDcond}
D(m) \le T \mu_{\La}(m)^{\alpha_0/5} \quad \text{for any $m$ such that $D(m) \ge 4 T \kappa_0^{-1}$}.
\end{equation}

\item[(6)] Let $D \in \mathcal{G}_{\La, T, \kappa_0}$. We say that $\gamma = (n_1, \dots, n_{k})$, $n_j \in \La$, $k \ge 1$ belongs to $\Ga_{D, T, \kappa_0} (n_1, n_{k}; k, \La, \mathfrak{R})$ if the following conditions hold:
\begin{equation}\label{eq:auxtrajectweight5NNNNN}
\begin{split}
& \min (D(n_{i}), D(n_{j})) \le T \| (n_{i}, \dots, n_{j}) \|^{\alpha_0/5} \\
& \text{for any $i < j$ such that $\min (D(n_{i}), D(n_{j})) \ge 4 T \kappa_0^{-1}$}, \quad \text{unless $j = i + 1$}.
\end{split}
\end{equation}
Moreover,
\begin{equation}\label{eq:auxtrajectweight5NNNNN1}
\begin{split}
& \text{if $\min (D(n_{i}), D(n_{i+1})) > T |(n_{i} - n_{i+1})|^{\alpha_0/5}$} \quad \text{for some $i$, then } \\
& \min (D(n_{j'}), D(n_{i})) \le T \| (n_{j'}, \dots, n_{i}) \|^{\alpha_0/5}, \quad \min (D(n_{i}), D(n_{j''})) \le T \| (n_{i}, \dots, n_{j''}) \|^{\alpha_0/5}, \\
& \min (D(n_{j'}), D(n_{i+1})) \le T \|(n_{j'}, \dots, n_{i+1}) \|^{\alpha_0/5}, \quad \min (D(n_{i+1}), D(n_{j''})) \le T \| (n_{i+1}, \dots, n_{j''}) \|^{\alpha_0/5},\\
& \text{for any $j' < i < i+1 < j''$.}
\end{split}
\end{equation}

Set $\Ga_{D, T, \kappa_0} (m, n; \La, \mathfrak{R}) = \bigcup_{k} \Ga_{D, T, \kappa_0} (m, n; k, \La, \mathfrak{R})$, $\Ga_{D, T, \kappa_0} (\La, \mathfrak{R}) = \bigcup_{m,n} \Ga_{D, T, \kappa_0} (m, n; \La, \mathfrak{R})$.

Set
\begin{equation}\label{eq:auxtrajectweight111111}
\begin{split}
s_{D, T, \kappa_0; k, \La, \mathfrak{R}} (m,n) & = \sum_{\gamma \in \Ga_{D, T, \kappa_0} (m, n; k, \La, \mathfrak{R})} w_{D,\kappa_0} (\gamma),\\
S_{D, T, \kappa_0; k, \La, \mathfrak{R}} (m,n) & = \sum_{\gamma \in \Ga_{D, T, \kappa_0} (m, n; k, \La, \mathfrak{R})} W_{D,\kappa_0} (\gamma).
\end{split}
\end{equation}
The letter $\mathfrak{R}$ in the above definitions does not signify an additional parameter; it is used as an abbreviation for the word ``resonant", for the purpose of distinguishing between these objects and similar objects introduced earlier in the definition.
\end{itemize}
\end{defi}

It is convenient to use the notation $\xi(n) := n \tilde \omega$, $n \in \mathfrak{T}$. Clearly, $\xi(n)$ is a real additive function on $\mathfrak{T}$. It is easy to see that the case $|\tilde \omega| > 1$ follows from the case $|\tilde \omega|\le 1$. Thus we assume that $|\tilde \omega|\le 1$, so that
\begin{equation}\label{eq:2xisbounded}
|\xi(n)| \le |n|, \quad n \in \mathfrak{T}.
\end{equation}
Recall that due to Lemma~\ref{lem:PAL10omdiophsubst} we have the following Diophantine condition:
\begin{equation}\label{eq:7-5-8latticeR}
|\xi(n)| \ge a_0 |n|^{-b_0}, \quad \text{ for any $|n| > 0$},
\end{equation}
where $a_0,b_0$ are constants.

Next we need to invoke the scales setup from \cite{DGL}.

\begin{defi}\label{defi:7generalcase}
Let $a_0$, $b_0$ be as in \eqref{eq:7-5-8latticeR}. Set $b_1 = 32 b_0$, $\beta_1 = b_1^{-1} = (32 b_0)^{-1}$. Fix an arbitrary $R_1$ with $\log R_1 \ge  \alpha_0^{-1}\max (\log (100 a_0^{-1}) , 2^{34}\beta_1^{-1} \log \kappa_0^{-1})$. Set
\begin{equation}\label{eq:A.1A}
R^\one = R_1, \quad \delta_0^4 := \delta^\zero_0 = (R^\one)^{-\frac{1}{\beta_1}}, \quad \delta_0^{(u-1)} = \exp \bigl( - (\log R^{(u-1)})^2 \bigr), \quad u = 2, \dots, \quad R^{(u)} := \bigl( \delta_0^{(u-1)} \bigr)^{-\beta_1}.
\end{equation}
Note that $\log \delta_0^{-1} > D(\kappa_0, \alpha_0, a_0, b_0) := 2^{32} \alpha_0^{-1}\beta_1^{-1} \log \kappa_0^{-1}$.
It is important to mention that
\begin{equation}\label{eq:diphnores}
\begin{split}
|\xi(n)| \ge a_0  |m|^{-b_0} \ge a_0 ( 48 R^{(u)})^{-b_0} & > (R^{(u)})^{-2b_0} = (\delta_0^{(u-1)})^{1/16} \; \text{ if $0 < |m| \le 48 R^{(u)}$}, \\
\log R^{(u)} & = \beta_1 (\log R^{(u-1)})^2, \\
\exp(-\kappa_0 (R^{(u-1)})^{\alpha_0}) & < (\delta_0^{(u)})^{16}.
\end{split}
\end{equation}
Set
\begin{equation}\label{eq:10K.1}
\begin{split}
k_n & = -\xi(n)/2, \quad \mathcal{I}_n = ( k_n - (\delta^\es_0)^{3/4}, k_n + (\delta^\es_0)^{3/4} ) \quad \text{if $12R^\esone < |n| \le 12 R^\es$}, \\
\mathcal{R}(k) & = \{ n \in \mathfrak{T} \setminus \{0\} : k \in \mathcal{I}_n\}, \quad \mathcal{G} = \{ k : |\mathcal{R}(k)| < \infty \}.
\end{split}
\end{equation}
\end{defi}

The following simple observations are due to the Diophantine property of the function $\xi(m)$.

\begin{lemma}\label{lem:10resetdiscr11}
Assume $m_1 \in \mathcal{R}(k)$. Let $12 R^{(s_1-1)} < |m_1| \le 12 R^{(s_1)}$. Then,

$(1)$ $|\xi(m_1)| > |m_1|^{-2b_0} \ge (12R^{(s_1)})^{-2b_0}$, $|k| > |m_1|^{-2b_0} \ge (12R^{(s_1)})^{-2b_0}$.

$(2)$ $\sgn(k) = -\sgn (\xi(m_1))$.

$(3)$ If $m_2 \in \mathcal{R}(k)$, $m_1 \neq m_2$, $|m_1| \le |m_2|$, then, in fact,
\begin{equation}\label{eq:10mjdeficompmm}
|m_2| > \frac{1}{2} R^{(s_1+1)} = \frac{1}{2} \exp(\beta_1 (\log  R^{(s_1)})^2),
\end{equation}
with $\beta_1$ as in \eqref{eq:A.1A}; in particular, $\beta_1 < 1/2$, $\log R^{(1)} > 2^4 \beta^{-2}_1$.

$(4)$ Let $k = k_{m_2}$, so that $m_2 \in \mathcal{R}(k)$. Assume  $m_1 \neq m_2$. Then \eqref{eq:10mjdeficompmm} holds.
\end{lemma}

\begin{defi}\label{defi:7slm}
Assume that $0 < |\mathcal{R}(k)| < \infty$. We enumerate the points of $\mathcal{R}(k)$ as $n^{(\ell)}(k)$, $\ell = 0, \dots, \ell(k)$, $1 + \ell(k) = |\mathcal{R}(k)|$, so that $|n^{(\ell)}(k)| < |n^{(\ell+1)}(k)|$. Let $s^{(\ell)}(k)$ be defined so that $12 R^{(s^{(\ell)}(k)-1)} < n^{(\ell)}(k) \le 12 R^{(s^{(\ell)}(k))}$, $\ell = 0, \dots, \ell(k)$, $s^{(\ell(k))} := s^{(\ell(k))}(k)$, $n^{(\ell(k))} := n^{(\ell(k))}(k)$. For $s^{(\ell)}(k) \le s < s^{(\ell+1)}(k)$, set $\mathcal{P}^\es(k) = \{ 0, n^{(\ell(k))} \}$. For $s < s^{(0)}(k)$, set $\mathcal{P}^\es(k) = \{ 0 \}$. Furthermore, set
\begin{equation}\label{eq:10mjdefi}
\begin{split}
T_{m}(n) & = m - n, \quad m,n \in \mathfrak{T}, \\
\mathfrak{m}^{(0)}(k) & = \{ 0, n^{(0)}(k)\}, \quad \mathfrak{m}^{(\ell)}(k) = \mathfrak{m}^{(\ell-1)}(k) \cup T_{n^{(\ell)}(k)} (\mathfrak{m}^{(\ell-1)}(k)), \quad \ell = 1,\dots,\ell(k).
\end{split}
\end{equation}
Let $n^\zero \in \mathfrak{T} \setminus \{ 0 \}$, $k = k_{n^\zero}$. Clearly, $n^\zero \in \mathcal{R}(k)$. It follows from $(4)$ in Lemma~\ref{lem:10resetdiscr11} that $\mathfrak{m}^{(\ell(k))}(k) = n^\zero$. Finally, we set $\ell(k) = 0$, $s^{(0)}(k) = 0$ if $\mathcal{R}(k) = \emptyset$. In particular, for $k = 0$, $\mathcal{R}(k) = \emptyset$, $\ell(k) = 0$, $s^{(0)}(k) = 0$. Note also that $n^{(\ell)}(k_{n^\zero}) = n^\zero$ for any $n^\zero \neq 0$.
\end{defi}

\begin{remark}\label{rem:2regardingtheoremD}
$(1)$ We now invoke Theorem~$\tilde D$ from \cite{DGL}. We invoke here only those statements from~Theorem $\tilde D$
that are needed to prove Theorem~$\tilde B$. We still use the enumeration the way it is in \cite{DGL}. For this reason there are obvious discrepancies in the enumeration below.

$(2)$ In one of the statements we have to invoke the classes of matrix-functions $GSR^{[\mathfrak{s}^{(\ell(k))}(k),s]} \bigl( \mathfrak{m}^{(\ell(k))}(k), m^+(k), m^-(k), \La^{(s)}_k(0); \delta_0,\mathfrak{t}^{(\ell(k))}(k)\bigr)$ from \cite{DGL}. The description of these classes is very involved and we do not reproduce it here. The reason why we need these classes of matrix functions is that in the proof of Theorem~$\tilde B$, we will invoke the theory of continued-fraction-functions from \cite{DG, DGL}. The latter plays a central role in the most important estimates developed in \cite{DG} and \cite{DGL}. It was established in those papers that the characteristic polynomials of the resonant part of the matrices belonging to the above-mentioned classes can be represented as continued-fraction-functions. The resonant eigenvalues of the matrices are implicit functions of the variables $\ve, k$ defined via zeros of the continued-fraction-function. The theory of continued-fraction-functions itself is completely independent of the rest of the material. One can find a comprehensive discussion of this theory in Section~5 of \cite{DG}.
\end{remark}

\begin{thmtd}
There exists $\ve_0 = \ve_0(a_0,b_0,\kappa_0) > 0$ such that for $1/2 \le \alpha_0 \le 1$,  $0 < \ve < \ve_0$, and any $k$, one can define sets $\La^{(s)}_{k}$, $s=1,2,\dots$ such that the following statements hold.

$(I)$ Assume that $\mathcal{R}(k, \tilde \omega) \neq \emptyset$. Let $n^{(\ell)}(k)$, $s^{(\ell)}(k)$, $\ell(k)$, $\mathfrak{m}^{(\ell)}(k)$ be as in \eqref{eq:10mjdefi}. Set $s^{(\ell(k)+1)}(k) := \infty$. Let $s^{(\ell)}(k) \le s < s^{(\ell+1)}(k)$ be arbitrary.

$(0)$ $H_{\La^{(s^{(\ell(k))}(k)+q)}_k(0), \ve, k} \in GSR^{[\mathfrak{s}^{(\ell(k))}(k),s]} \bigl( \mathfrak{m}^{(\ell(k))}(k), m^+(k), m^-(k), \La^{(s)}_k(0); \delta_0, \mathfrak{t}^{(\ell(k))}(k) \bigr)$, $\mathfrak{t}^{(\ell)}(k) = (\tau^\zero(k), \dots, \tau^{(\ell)}(k))$, $\tau^{(r)}(k) = |k_{n_r}| ||k| - |k_{n_r}||$, $m^+(k) = 0$, $m^-(k) = n^{(\ell(k))}(k)$ if $|k| > |k_{n^{(\ell)}(k)}|$, $m^-(k) = 0$, $m^+(k) = n^{(\ell(k))}(k)$ if $|k| < |k_{n^{(\ell)}(k)}|$. In particular,
\begin{equation}\label{eq:boxescondition2}
B(0, R^\es) \cup B(n^{(s^{(\ell)}(k))}, R^\es) \subset \La^\es_{k} \subset B(0, 16 R^{(s)}).
\end{equation}

$(ii)$ Let $|k_1 - k| < (\delta^\esone_0)^{1/4}$, $k_1 \notin \{ \xi(m) : m \in \mathfrak{F} \}$. There exist real-analytic functions $E(0, \Lambda^\es_{k}; \ve, k_1)$, $E(n^{(\ell)}(k), \Lambda^\es_{k}; \ve, k_1)$ of $\ve$, $E(0, \Lambda_{k}; \ve, k_1) \neq E(\mathfrak{m}^{(\ell)}(k, \tilde \omega), \Lambda^\es_{k}; \ve, k_1)$ for any $\ve$, such that $E(\cdot, \La_{k}; 0, k_1) = v(\cdot,k_1)$ and
\begin{equation} \label{eq:5specHEEAAA}
\spec  H_{\La^\es_{k}, \ve, k_1} \cap \{ \min_{\cdot} |E - E(\cdot, \Lambda^\es_{k}; \ve, k_1)| < (\delta^{(s-1)}_0)^{1/8} \} = \{ E(0, \Lambda^\es_{k}; \ve, k_1), E(n^{(\ell)}(k), \Lambda^\es_{k}; \ve, k_1) \}.
\end{equation}
Furthermore,
\begin{equation} \label{eq:7specplitell1}
|E(0, \Lambda^\es_{k}; \ve, k_1) - E(n^{(\ell)}(k), \Lambda^\es_{k}; \ve, k_1)| < \exp \Big( -\frac{\kappa_0}{2} |n^{(\ell)}(k)|^{\alpha_0} \Big).
\end{equation}

$(iii)$ Let $\min_{\cdot} | E - E(\cdot, \Lambda^\es_{k}; \ve, k_1)| < (\delta^{(s-1)}_0)^{1/8}$. There exists a function $D(\cdot) := D(\cdot, \Lambda^\es_{k}) \in \mathcal{G}_{\Lambda^\es_{k}, T, \kappa_0}$, $T = 4 \kappa_0 \log \delta_0^{-1}$, see $(5)$ in Definition~\ref{def:aux1}, such that
\begin{equation}\label{eq:3Hinvestimatestatement1PQ}
|[(E - H_{\La^\es_{k} \setminus \mathcal{P}^\es(k), \ve, k_1})^{-1}] (x,y)| \le s_{D(\cdot), T, \kappa_0, |\ve|; \Lambda^\es_{k} \setminus \{0,n^{(\ell)}\}, \mathfrak{R}}(x,y).
\end{equation}

$(iv)$
\begin{equation}\label{eq:7Ederiss1TD-1}
|E(0, \La^{(s)}_k; \ve, k) - E(0, \La^{(s+1)}_k; \ve, k)| \le |\ve| (\delta^\es_0)^6.
\end{equation}

$(III)$ Let $n^\zero \in \mathfrak{T} \setminus \{ 0 \}$ be arbitrary. Let $s \ge s^{(\ell( k_{n^\zero}))}$. Assume, for instance, $k_{n^\zero} > 0$.

$(1)$ The limits
\begin{equation}\label{eq:10kk1comp1lim}
E^\pm(\La^\es_{k_{n^\zero}}; \ve, k_{n^\zero}) := \lim_{k_1 \rightarrow k_{n^\zero} \pm 0} E(0, \La^\es_{k_{n^\zero}}; \ve, k_1)
\end{equation}
exist, and
\begin{equation}\label{eq:10kk1comp1gapsize}
0 \le E^+(\La^\es_{k_{n^\zero}}; \ve, k_{n^\zero}) - E^-(\La^\es_{k_{n^\zero}}; \ve, k_{n^\zero}) \le 2 |\ve| \exp \Big( -\frac{\kappa_0}{2} |n^\zero|^{\alpha_0} \Big).
\end{equation}
Furthermore,
\begin{equation}\label{eq:10specHEEAAA}
\begin{split}
\spec H_{\La^\es_{k_{n^\zero}}(0), \ve,k_{n^\zero}} & \cap \{ E : \min_{\pm} |E - E^\pm(\La^\es_{k_{n^\zero}}; \ve, k_{n^\zero})| < 8 (\delta^{(\bar s_0-1)}_0)^{1/8}\} \\
& = \{ E^-(\La^\es_{k_{n^\zero}}; \ve, k_{n^\zero}), E^+(\La^\es_{k_{n^\zero}}; \ve, k_{n^\zero}) \}.
\end{split}
\end{equation}

Let $\min_{\pm} |E - E^\pm(\La^\es_{k_{n^\zero}, \tilde \omega}; \ve, k_{n^\zero}) \bigr)| < (\delta_0^{(s-1)})^{1/8}$. There exists a function $D(\cdot) := D(\cdot, \Lambda^\es_{k_{n^\zero}}) \in \mathcal{G}_{\Lambda^\es_{k_{n^\zero}}, T, \kappa_0}$, $T = 4 \kappa_0 \log \delta_0^{-1}$, see $(5)$ in Definition~\ref{def:aux1}, such that
\begin{equation}\label{eq:10Hinvestimatestatement1PQ}
|[(E - H_{\Lambda^\es_{k_{n^\zero}} \setminus \mathcal{P}^\es(k_{n_0}), \ve, k_{n^\zero}})^{-1}](m,n)| \le s_{D(\cdot), T, \kappa_0, |\ve|; \Lambda^\es_{k} \setminus \{0,n^\zero\}, \mathfrak{R}}(m,n).
\end{equation}

$(2)$ \begin{equation}\label{eq:7Ederiss1TD-2}
|E^\pm \bigl(\La^{(s)}_{k_{n^\zero}}; \ve, k_{n^\zero} \bigr) - E^\pm \bigl( \La^{(s+1)}_{k_{n^\zero}}; \ve, k_{n^\zero} \bigr)| \le |\ve| (\delta^\es_0)^6.
\end{equation}

$(3)$ $E = E^\pm(\La^\es_{k_{n^\zero}}; \ve, k_{n^\zero})$ obeys the following equation,
\begin{equation}\label{eq:10Eequation0}
E - v(0, k_{n^\zero}) - Q(0, \La^\es_{k_{n^\zero}}; \ve, E) \mp \big| G(0, n^\zero, \La^\es_{k_{n^\zero}}; \ve, E) \big| = 0,
\end{equation}
where
\begin{equation} \label{eq:10-10acbasicfunctions}
\begin{split}
& Q(0,\La^\es_{k_{n^\zero}}; \ve, E) \\
& = \sum_{m', n' \in \La^\es_{k_{n^\zero}} \setminus \mathcal{P}^\es(k_{n_0})} h(0, m'; \ve, k_{n^\zero}) [(E - H_{\Lambda^\es_{ k_{n^\zero}} \setminus \{ 0, n^\zero \}, \ve, k_{n^\zero}})^{-1}](m',n') h(n', 0; \ve, k_{n_0}), \\
& G(0, n^\zero, \La^\es_{k_{n^\zero}}; \ve, E) = h(0, n^\zero; \ve, k_{n^\zero} ) \\
& + \sum_{m', n' \in \La^\es_{k_{n^\zero}} \setminus \mathcal{P}^\es(k_{n_0})} h(0, m'; \ve, k_{n^\zero}) [(E - H_{\Lambda^\es_{ k_{n^\zero}} \setminus \{ 0, n^\zero \}, \ve, k_{n^\zero}})^{-1}](m',n') h(n', n^\zero; \ve, k_{n^\zero}).
\end{split}
\end{equation}
\end{thmtd}

We need to invoke also the following fact from the proof of Theorem~$\tilde C$: the function $E(k)$ in Theorem~$\tilde A$ obeys
\begin{equation}\label{eq:7Ederiss1TD-3}
|E(k) - E(0, \La^{(s+1)}_k; \ve, k)| \le 2 |\ve| (\delta^\es_0)^6;
\end{equation}
see \eqref{eq:7Ederiss1TD-1} in Theorem~$\tilde D$.

\begin{proof}[Proof of part $(1)$ of Theorem~B]
Let $n^\zero \in \mathfrak{T} \setminus \{ 0 \}$ be arbitrary. We assume that $k_{n^\zero} = - \frac{n^\zero \omega}{2} > 0$. The case $k_{n^\zero} = - \frac{n^\zero \omega}{2} < 0$ is similar. Due to part $(III)$ of Theorem~$\tilde D$, we have
\begin{equation}\label{eq:11kk1comp1gapsize}
0 \le E^+(\La^\es; k_{n^\zero}) - E^-(\La^\es; k_{n^\zero}) \le 2 \ve \exp \Big(-\frac{\kappa_0}{2} |n^\zero| \Big),
\end{equation}
where $\La = \La^{(s)}_{k_{n^\zero}}(0)$, $s = s^{(\ell(k_{n^\zero}))}(k_{n^\zero})$, $\ell = \ell(k_{n^\zero})$. Combining this with \eqref{eq:7Ederiss1TD-3}, we obtain
\begin{equation}\label{eq:11kk1comp1gapsizelim}
0 \le E^+(k_{n^\zero}) - E^-(k_{n^\zero}) \le 2 \ve \exp \Big(-\frac{\kappa_0}{2} |n^\zero|^{\alpha_0} \Big),
\end{equation}
as claimed in part $(1)$ of Theorem~$\tilde B$.
\end{proof}

\begin{remark}\label{rem3proofs}
$(1)$ The proof of part $(2)$ of Theorem~$\tilde B$ is practically the same as the proof of part $(2)$ of Theorem~$B$ in
\cite{DG}. The fact that we assume $|c(n)| \le \ve \exp(-\kappa_0 |n|^{\alpha_0})$, rather than $|c(n)| \le \ve \exp(-\kappa_0 |n|)$, does not introduce many changes. This is because the estimation of the weights in Definition~\ref{def:aux1} works for any $\alpha_0$. This fact was verified in \cite{DGL}. Here we just utilize this fact. First we discuss the main contraction principle for the Fourier coefficients $c(m)$, $m \in \mathfrak{T}$. The contraction follows from the eigenvalue equation \eqref{eq:10Eequation0} in Theorem~$\tilde D$. The derivation is done in Lemma~\ref{lem:11.gapinequalities} below. Once the contraction is established, the proof of part $(2)$ in Theorem~$\tilde B$ relies on the elementary theory of the weight functions $s_{D(\cdot);...}$, see Definition~\ref{def:aux1}. We reproduce here the lemmas from the proof of part $(2)$ of Theorem~$B$ to demonstrate explicitly how this weight function theory finishes the proof of Theorem~$\tilde B$.

$(2)$ To derive Lemma~\ref{lem:11.gapinequalities} we introduce in the current presentation the additional Lemmas~\ref{lem:3.symmetry} and \ref{lem:3.smalldenominators} to show the connection between the statements in Section~4 of \cite{DG} and the estimates needed in the proof of Theorem~$\tilde B$. We use equation \eqref{eq:10Eequation0} from part $(III)$ of Theorem~$\tilde D$:
\begin{equation} \label{eq:11Eequation0}
[E - v(0, k_{n^\zero}) - Q^{(s)}(0,\La; \ve, k_{n^\zero}, E)  \mp \big| G^{(s)}(0,n^\zero,\La; \ve, k_{n^\zero}, E) \big|]|_{E = E^\pm(\La; \ve, k_{n^\zero}} = 0.
\end{equation}
Set
\begin{equation} \label{eq:11-10acbasicfunctions-2}
\begin{split}
a_1 & = v(0, k_{n^\zero} + \theta) + Q^{(s)}(0, \La; k_{n^\zero} + \theta, E), \\
a_2 & = v(0, k_{n^\zero} - \theta) + Q^{(s)}(0, \La; k_{n^\zero} - \theta, E), \\
b & = G^{(s)}(0, n_0, \La; k_{n^\zero} + \theta, E),
\end{split}
\end{equation}
$f_i := E - a_i$, $i = 1,2$, $f = f_1 - |b|^2 f_2^{-1}$.  Let for instance $k_{n^\zero} > 0$. Due to \cite[Proposition~6.9]{DG}, $f$ and $f_i$ are continued-fraction-functions: $f \in \mathfrak{F}^{(\ell)}_{\mathfrak{g}^{(\ell - 1)},1} (f_1,f_2,b^2)$, for $\theta > 0$, see \cite[Definition~4.9]{DG}. These functions are rational in terms of the variables $\ve, k, E$. One can get rid of the corresponding denominators to turn them into polynomials. This is done inductively in \cite[Definition~4.9]{DG} by setting
\begin{equation}\label{eq:4a-bfunctions1a}
\begin{split}
\mu^{(f)} & = \begin{cases} \mu^{(f_1)}\mu^{(f_2)}f_2 & \text {if} \quad f = f_1 - \frac{b^2}{f_2}, \\ \mu^{(f_1)} \mu^{(f_2)} f_1 & \text {if} \quad f = f_2 - \frac{b^2}{f_1}, \end{cases} \\
\chi^{(f)} & = \mu^{(f)} f, \\
\tau^{(f)} & = (\chi^{(f_2)}-\chi^{(f_1)}) \tau^{(f_1)}\tau^{(f_2)},
\end{split}
\end{equation}
starting with $\chi^{(f_i)} := f_i$, $\mu^{(f_i)} := 1$, $\tau^{(f_i)} := 1$, $\sigma(f_i) := 1$, $i = 1, 2$ when $f \in \mathfrak{F}^{(1)}_{\mathfrak{g}^\one} (a_1,a_2,b^2)$. Due to the theory of continued-fraction-functions,
\begin{equation}\label{eq:4muchismoothness}
|\partial^\alpha \mu^{(f_j)}|, |\partial^\alpha \chi^{(f_j)}| \le 1, \quad 0 \le |\alpha| \le 2,
\end{equation}
see part $(2)$ of \cite[Lemma~4.11]{DG}. The estimates \eqref{eq:4muchismoothness} allow us to get rid of the small denominators of continued-fraction-functions.

$(3)$ We prove part $(2)$ of Theorem~$\tilde B$ for $\alpha'_0 = \alpha_0$. We do this just to avoid a repetition of the same arguments. One can see that the proof works for any $\alpha_0' \ge \alpha_0$.
\end{remark}

\begin{lemma}\label{lem:3.symmetry}
$(i)$ $f_2(\theta) = f_1(-\theta)$, $\chi^{(f_2)}(\theta) = \chi^{(f_1)}(-\theta)$. \\
$(ii)$ $|\chi^{(f_2)}(\theta) - \chi^{(f_1)}(\theta)| > (\tau^{(f_{1})})^8 |\theta|$.
\end{lemma}

\begin{proof}
$(i)$ is due to the definition \eqref{eq:11-10acbasicfunctions-2}. $(ii)$ follows from $(i)$ due to parts $(8)$, $(9)$ of \cite[Lemma~4.11]{DG}.
\end{proof}

\begin{lemma}\label{lem:3.smalldenominators}
The following estimates hold,
\begin{equation} \label{eq:3smalldenominators}
\begin{split}
|\tau^{(f_i)}| & > \exp(-C(b_0) (\log [4 + |n^{(0)}|])^3), \\
|\mu^{(f_i)}| & > \exp(-C(b_0) (\log [4 + |n^{(0)}|])^3).
\end{split}
\end{equation}
\end{lemma}

\begin{proof}
Let $n^\zero \neq 0$ and let $n^{(\ell')} := n^{(\ell')}(k_{n^\zero})$, $\ell' = 0, \dots, \ell = \ell(k_{n^\zero})$
be as in Definition~\ref{defi:7slm}. Recall first of all that $n^{(\ell(k_{n^\zero}))} = n^\zero$, see Definition~\ref{defi:7slm} and Lemma~\ref{lem:10resetdiscr11}. Assume for instance that $\bar \ell := \ell(k_{n^\zero})>1$, so that $n^{(\bar \ell - 1)}$ exists. Then due to \cite[Definition~4.9]{DG}, we have $f_i \in \mathfrak{F}^{(\bar \ell - 1)}_{\mathfrak{g}^{(\ell - 1)}} (f_{i,1}, f_{i,2}, b_i^2)$, $i = 1,2$. Here
$f_{i,1}, f_{i,2}, b_i^2$ are defined similarly to $f_1, f_2, b$ in \eqref{eq:11-10acbasicfunctions-2} with $k_{n^{(\bar \ell - 1)}}$ in the role of $k_{n^\zero}$ and $k_{n^{(\bar\ell-1)}} - k_{n^\zero}$ in the role of $\theta$. Due to Lemma~\ref{lem:3.symmetry} we obtain $|\chi^{(f_{i,2})} - \chi^{(f_{i,1})}| > (\tau^{(f_{i,1})})^8 |k_{n^{(\bar \ell - 1)}} - k_{n^\zero}|$. Due to Lemma~\ref{lem:10resetdiscr11}, $|k_{n^{(\bar \ell - 1)}} - k_{n^\zero}| > (2 |n^\zero|)^{-2b_0}$. Similar estimates hold for $f_{i,j}$ in the role $f_i$, etc. Using induction we conclude that
\begin{equation} \label{eq:3smallden1}
|\tau^{(f_i)}| > \prod_{\ell \le \bar \ell} |k_{n^{(\ell - 1)}} - k_{n^{(\ell)}}|^{16} > \prod_{\ell \le \bar \ell} |n^{(\ell)}|^{-32 b_0}.
\end{equation}
Once again, due to Lemma~\ref{lem:10resetdiscr11}, $|n^{(\ell - 1)}| < \exp ((\log |n^{(\ell)}|)^{2/3})$, $\bar \ell < \log (4 + |n^\zero|)$. Substituting this in \eqref{eq:3smallden1}, we obtain the first estimate in \eqref{eq:3smalldenominators}. Due to part $(4)$ of \cite[Lemma~4.11]{DG}, we have $|\mu^{(f_i)}| \ge 2^{-2^{\ell - 1} + 1} \tau^{(f_{i})}$. Therefore the second estimate in \eqref{eq:3smalldenominators} follows from the first one.
\end{proof}

\begin{lemma}\label{lem:11.gapinequalities}
Using the notation from the proof of part $(1)$ of Theorem~$\tilde B$, for any $n^\zero$, the Fourier coefficient $c(n^\zero)$ obeys the following estimate,
\begin{equation} \label{eq:11gapsineqstatement}
\begin{split}
|c(n^\zero)| & \le \exp(C(b_0)(\log [4+|n^{(0)}|])^3) (E^+(\La;k_{n^\zero}) - E^-(\La;k_{n^\zero})) \\
& \quad + \sum_{m', n' \in \La \setminus \{ 0,n^\zero \}} |c(m')| s_{D(\cdot;\La \setminus \{ 0, n^\zero \}), T, \kappa_0, \ve; \La \setminus \{0,n^\zero\}, \mathfrak{R}}(m',n') |c(n'- n^\zero)|.
\end{split}
\end{equation}
\end{lemma}

\begin{proof}
Equation \eqref{eq:11Eequation0} implies the following inequality,
\begin{equation} \label{eq.11Eestimates1APFINALB1}
|b(k_{n^\zero}, E)| |_{E = E^+(\La;k_{n^\zero})} \le |f_1 (k_{n^\zero}, E) |_{E = E^+(\La;k_{n^\zero})} - f_1 (k_{n^\zero}, E) |_{E = E^-(\La;k_{n^\zero})}|.
\end{equation}
Now we get rid of small denominators in the continued-fraction-functions $f_i$:
\begin{equation} \label{eq.11Eestimatesnodenominators}
\begin{split}
|b(k_{n^\zero}, E)| & |_{E = E^+(\La;k_{n^\zero})}\prod_\pm |\mu^{(f_1)}(k_{n^\zero}, E) |_{E = E^\pm(\La;k_{n^\zero})}| \\
& \le \Big|\mu^{(f_1)}(k_{n^\zero}, E) |_{E = E^-(\La;k_{n^\zero})}\chi^{(f_1)}(k_{n^\zero}, E) |_{E = E^+(\La;k_{n^\zero})} \\
& \quad - \mu^{(f_1)}(k_{n^\zero}, E) |_{E = E^+(\La;k_{n^\zero})}\chi^{(f_1)}(k_{n^\zero}, E) |_{E = E^-(\La;k_{n^\zero})}\Big|.
\end{split}
\end{equation}
Next we invoke the estimates \eqref{eq:4muchismoothness},
\begin{equation} \label{eq.11Eestimates1APFINALB4-1}
|b(k_{n^\zero}, E)| |_{E = E^+(\La;k_{n^\zero})}\prod_\pm |\mu^{(f_1)}(k_{n^\zero}, E) |_{E = E^\pm(\La;k_{n^\zero})}| \le 2 (E^+(\La;k_{n^\zero}) - E^-(\La; k_{n^\zero})).
\end{equation}
Due to Lemma~\ref{lem:3.smalldenominators} this implies
\begin{equation} \label{eq.11Eestimates1APFINALB4-2}
|b(k_{n^\zero}, E)| |_{E = E^+(\La;k_{n^\zero})} \le \exp(C(b_0) (\log [4 + |n^{(0)}|])^3) (E^+(\La;k_{n^\zero}) - E^-(\La; k_{n^\zero})).
\end{equation}
Due to the definitions in \eqref{eq:10-10acbasicfunctions}, \eqref{eq:11-10acbasicfunctions-2},
\begin{equation} \label{eq:11-10acbasicfunctionsBineq}
b(k_{n^\zero},E) = c(n^\zero) + \sum_{m', n' \in \La \setminus \{0,n^\zero\}} c(m') [(E - H_{\La \setminus \{ 0, n^\zero \}, k_{n^\zero}})^{-1}] (m',n') c(n'- n^\zero),
\end{equation}
and due to \eqref{eq:10Hinvestimatestatement1PQ},
\begin{equation}\label{eq:11Hinvestimatestatement1PQ}
|[(E - H_{\La \setminus \{0,n^\zero\},k_{n^\zero}})^{-1}](m,n)| \le s_{D(\cdot; \La \setminus \{0,n^\zero\}), T, \kappa_0, \ve; \La \setminus \{0,n^\zero\}, \mathfrak{R}}(m,n).
\end{equation}
Therefore, \eqref{eq:11gapsineqstatement} follows from \eqref{eq.11Eestimates1APFINALB4-2}.
\end{proof}

It is convenient to introduce the following notation: $\La := \La^\es$, $\La'(n^\zero) = \La \setminus \{ 0, n^\zero \}$, $\La(n^\zero) = \mathfrak{T} \setminus \{0,n^\zero\}$,
\begin{equation} \label{eq:11-10snotation}
s(n^\zero;m',n') = \begin{cases} s_{D(\cdot; \La(n^\zero)), T, \kappa_0, \ve; \La(n^\zero), \mathfrak{R}}(m',n') & \text{if $m', n' \in \La'(n^\zero)$}, \\ 0 & \text{if $m' \in \La(n^\zero) \setminus \La'(n^\zero)$, or $n' \in \La(n^\zero) \setminus \La'(n^\zero)$, or both.}
\end{cases}
\end{equation}

Let us recall the main properties of the sum $s(n^\zero; m', n')$ from \cite[Section~2]{DG}.

\begin{lemma}\label{lem:11.sumsproperties}
Let $s(n^\zero; m', n')$ be as in \eqref{eq:11-10snotation}.

$(1)$
\begin{equation}\label{eq:11auxtrajectweightO}
\begin{split}
s(n^\zero;m,n) & \le \sum_{\gamma \in \Ga_{n^\zero}(m, n)} w_{n^\zero} (\gamma), \\
w_{n^\zero} (\gamma) & = \Big[ \prod w(n_j,n_{j+1}) \Big] \exp \Big( \sum_{1 \le j \le k} D_{n^\zero}(n_j) \Big).
\end{split}
\end{equation}
Here $w(m,n) := |c(n-m)|$, $\Ga_{n^\zero}(m, n)$ stands for a set of trajectories $\gamma = (n_1,\dots,n_k)$, $k := k(\gamma) \ge 1$, $n_j \in \La(n^\zero)$, $n_1 = m$, $n_k = n$, $n_{j+1} \neq n_j$, $D_{n^\zero}(x) > 0$, $x \in \zv \setminus \{0,n^\zero\}$. Moreover, the following conditions hold:

$(i)$ $D_{n^\zero}(x) \le T \mu_{n^\zero}(x)^{1/5}$ for any $x$ such that $D_{n^\zero}(x) \ge 4 T \kappa_0^{-1}$, where $\mu_{n^\zero}(x) = \min (|x|,|x-n^\zero|)$, $T = 4 \kappa_0 \log \delta_0^{-1}$,

$(ii)$
\begin{equation}\label{eq:11auxtrajectweight5NNNNN}
\begin{split}
& \min (D_{n^\zero}(n_{i}), D_{n^\zero}(n_{j})) \le T \| (n_{i}, \dots, n_{j}) \|^{1/5} \\
& \text{for any $i < j$ such that $\min (D_{n^\zero}(n_{i}), D_{n^\zero}(n_{j})) \ge 4 T \kappa_0^{-1}$}, \quad \text{unless $j = i + 1$}.
\end{split}
\end{equation}
Moreover,
\begin{equation}\label{eq:11auxtrajectweight5NNNNN1}
\begin{split}
& \text{if $\min (D_{n^\zero}(n_{i}), D_{n^\zero}(n_{i+1})) > T |(n_{i} - n_{i+1})|^{1/5}$ for some $i$, then} \\
& \min (D_{n^\zero}(n_{j'}), D_{n^\zero}(n_{i})) \le T \| (n_{j'}, \dots, n_{i}) \|^{1/5}, \\
& \min (D_{n^\zero}(n_{i}), D_{n^\zero}(n_{j''})) \le T \| (n_{i}, \dots, n_{j''}) \|^{1/5}, \\
& \min (D_{n^\zero}(n_{j'}), D_{n^\zero}(n_{i+1})) \le T \|(n_{j'}, \dots, n_{i+1}) \|^{1/5}, \\
& \min (D_{n^\zero}(n_{i+1}), D_{n^\zero}(n_{j''})) \le T \| (n_{i+1}, \dots, n_{j''}) \|^{1/5}, \\
& \text{for any $j' < i < i+1 < j''$.}
\end{split}
\end{equation}

$(2)$ Assume that for all $n \in \mathfrak{T}$, we have $|c(n)| \le \tilde \ve \exp (- \tilde \kappa |n|)$ with $\tilde \ve \le \ve_0$, $\tilde \kappa \ge \kappa_0$. Let $\gamma = (n_1, \dots, n_{k}) \in \Ga_{n^\zero} := \bigcup_{m,n} \Ga_{n^\zero}(m, n)$. Set $M = 4 T \kappa_0^{-1}$, $\bar D(\gamma) = \max_j D(n_j)$, $t_D(\gamma) := \frac{\log \bar D(\gamma)}{\log M}$, $\vartheta_t = \sum_{0 < s \le t} 2^{-5s}$. Then,
\begin{equation}\label{eq:11auxtrajectweight2}
w_{n^\zero}(\gamma) \le \begin{cases}  \tilde \ve^{k(\gamma)-1} \exp (- \tilde \kappa \| \gamma \| + k (\gamma) M^5) & \text{if $t_D(\gamma) \le 5$}, \\
\tilde \ve^{k(\gamma)-1} \exp (- \tilde \kappa (1 - \vartheta_{t_D(\gamma)+1}) \| \gamma \| + 2 \bar D(\gamma)) & \text{if $t_D(\gamma) > 5$}. \end{cases}
\end{equation}
Furthermore, $\bar D(\gamma) \le 2 T [\min (|n_1|,|n^\zero-n_k|)^{1/5} + \| \gamma \|^{1/5}]$.
\end{lemma}

In the next lemma we establish an estimate similar to \eqref{eq:11auxtrajectweight2} under a slightly weaker condition on $|c(n)|$, and also an estimate for the sum of such terms.

\begin{lemma}\label{lem:11scaleddecayestimate}
Let $0 < \tilde \ve \le \ve_0$, $\tilde \kappa \ge 2 \kappa_0$, $R_1 \ge 2^{30} (\kappa_0^{-1}T)^2$. Set $R_t = 5 R_{t-1}/4$, $\rho_{t-1} = 2^{-10} t^{-2}$, $t = 2,\dots$, $\sigma_t = \sum_{1 \le \ell \le t} \rho_\ell$. Assume
\begin{equation} \label{eq:11scaledconditionaldecay}
|c(p)| \le \begin{cases} \tilde \ve \exp ( - \tilde \kappa |p|) & \text{if $ 0 < |p| \le R_2$}, \\ \tilde \ve \exp ( - \frac{15}{16} (1 - \sigma_{3t}) \tilde \kappa |p|) & \text{if $R_{t-1} < |p| \le R_t$, $3 \le t \le t_0$.}\end{cases}
\end{equation}
For $t \ge 1$, let $\Gamma^{(t)}_{n^\zero}$ be the set of trajectories $\gamma = (n_1, \dots, n_{k}) \in \Ga_{n^\zero}$ with $\| \gamma \| \le 2 R_t$ and $\max_j |n_{j+1} - n_j| \le R_{t+1}$. Then, for any $\gamma \in \Gamma^{(t)}_{n^\zero}$ with $t \le t_0-1$, we have
\begin{equation} \label{eq:11gamscaledconditionaldecay}
w_{n^\zero}(\gamma) \le \tilde \ve^{k(\gamma)-1} \exp \Big( -\frac{15}{16} (1 - \sigma_{3t+4}) \tilde \kappa \| \gamma \| + 2 \bar D(\gamma) + k(\gamma)M \Big).
\end{equation}
Furthermore,
\begin{equation} \label{eq:11gamscaledconditionaldecaysum1statement}
\sum_{\gamma \in \Ga_{n^\zero}(m, n) : k(\gamma) \ge 2, \| \gamma \| \le R_{t_0}} w_{n^\zero}(\gamma) \le \exp(- 2 T (\min (|m|,|n^\zero-n|)^{1/5}) \exp \Big( -\frac{15}{16} (1 - \sigma_{3t_0+2}) \tilde \kappa |n-m| \Big).
\end{equation}
\end{lemma}

\begin{proof}
The proof of \eqref{eq:11gamscaledconditionaldecay} goes by induction in $t = 1, 2, \dots$. Let $\gamma = (n_1,\dots,n_{k}) \in \Gamma^{(1)}_{n^\zero}$. Then, in particular, $\max_j |n_{j+1}-n_j| \le R_{2}$. Due to \eqref{eq:11scaledconditionaldecay}, one has $w(n_j,n_{j+1}) \le \tilde \ve \exp(-\tilde \kappa |n_j-n_{j+1}|)$. Hence \eqref{eq:11auxtrajectweight2} applies. Note that $1 - \vartheta_{t_D(\gamma)+1} > 15/16$. This implies \eqref{eq:11gamscaledconditionaldecay} for $t = 1$ in both cases in \eqref{eq:11auxtrajectweight2}.

Let $\Gamma^{(t)}_{n^\zero,0}$ be the set of trajectories $\gamma = (n_1,\dots,n_{k}) \in \Gamma^{(t)}_{n^\zero}$ with $\max_j |n_{j+1}-n_j| \le R_t$, $\Gamma^{(t)}_{n^\zero,1} = \Gamma^{(t)}_{n^\zero} \setminus \Gamma^{(t)}_{n^\zero,0}$.

Let $\gamma = (n_1,\dots,n_{k}) \in \Gamma^{(t)}_{n^\zero,1}$. Then there exists $j_0$ such that $|n_{j_0+1} - n_{j_0}| > R_t$. Note that $|n_{j+1} - n_{j}| < R_t$ for any $j \neq j_0$, since $\| \gamma \| \le 2 R_t$. Let $\gamma_1 = (n_1,\dots,n_{j_0})$, $\gamma_2 = (n_{j_0+1},\dots,n_k)$. Note that $\| \gamma_1 \| + \| \gamma_2 \| < R_t < 2 R_{t-1}$ since $\| \gamma \| \le 2 R_t$ and $|n_{j_0+1} - n_{j_0}| > R_t$. Therefore $\gamma_1,\gamma_2 \in \Gamma^{(t-1)}_{n^\zero}$. Hence, the inductive assumption applies,
\begin{equation} \label{eq:11gamscaledconditionaldecay2}
\begin{split}
w_{n^\zero}(\gamma_i) & \le \tilde \ve^{k(\gamma_i)-1} \ve \exp \Big(-\frac{15}{16} (1 - \sigma_{3t+1}) \tilde \kappa \| \gamma_i \| + 2 \bar D(\gamma_i) + k(\gamma_i)M \Big), \quad i=1,2, \\
w_{n^\zero}(\gamma) & = w_{n^\zero}(\gamma_1) |c(n_{j_0+1} - n_{j_0})| w_{n^\zero}(\gamma_2) \\
& \le \tilde \ve^{k(\gamma)-1} \exp \Big( -\frac{15}{16} (1 - \sigma_{3t+1}) \tilde \kappa (\|\gamma_1\| + \|\gamma_2\|) \\
& \qquad -\frac{15}{16} (1 - \sigma_{3t+3}) \tilde \kappa |n_{j_0+1} - n_{j_0}| + 2 \bar D(\gamma_1) + 2 \bar D(\gamma_2) + k(\gamma)M \Big).
\end{split}
\end{equation}
Let $D(n_{j_i}) = \bar D(\gamma_i)$, $i=1,2$. We have the following cases:

$(a)$ Assume $j_1 < j_0$. In this case due to \eqref{eq:11auxtrajectweight5NNNNN}, we have
\begin{equation}\label{eq:11auxtrajectweight5NUPP}
2 \min (\bar D(\gamma_1), \bar D(\gamma_2)) \le 2 T \|\gamma \|^{1/5} < \frac{\rho_{3t+4}}{4} \tilde \kappa (\|\gamma_1\| + \|\gamma_2\| + |n_{j_0+1} - n_{j_0}|)
\end{equation}
since $\|\gamma\| > R_t \ge 2^{30} (\kappa_0^{-1}T)^2 (5/4)^{t-1}$, $t \ge 2$. Combining \eqref{eq:11gamscaledconditionaldecay2} with \eqref{eq:11auxtrajectweight5NUPP}, we obtain \eqref{eq:11gamscaledconditionaldecay}.

$(b)$ Assume $j_0 + 1 < j_2$. Similarly to case $(a)$, one verifies \eqref{eq:11gamscaledconditionaldecay}.

$(c)$ Assume $j_1 = j_0$, $j_0+1 = j_2$. Let $\gamma'_1 = (n_1,\dots,n_{j_0-1})$, $\gamma'_2 = (n_{j_0+2},\dots,n_k)$. Once again, applying the inductive assumption, we obtain
\begin{equation} \label{eq:11gamscaledconditionaldecay7}
\begin{split}
w_{n^\zero}(\gamma'_i) & \le \tilde \ve^{k(\gamma_i)-1} \exp \Big( -\frac{15}{16} (1 - \sigma_{3t+1}) \tilde \kappa \|\gamma'_i\| + 2 \bar D(\gamma'_i) + k(\gamma_i)M \Big), \quad i=1,2, \\
w_{n^\zero}(\gamma) & = w_{n^\zero}(\gamma'_1) \exp(D_{n^\zero}(n_{j_0})) |c(n_{j_0+1} - n_{j_0})| \exp(D_{n^\zero} (n_{j_0+1})) w_{n^\zero}(\gamma'_2) \\
& \le \tilde \ve^{k(\gamma)-1} \exp \Big(-\frac{15}{16} (1 - \sigma_{3t+1}) \tilde \kappa (\|\gamma_1\| + \|\gamma_2\|) - \frac{15}{16} (1 - \sigma_{3t+3}) \tilde \kappa |n_{j_0+1} - n_{j_0}| \Big) \times \\
& \quad \exp(2 \bar D(\gamma_1) + 2 \bar D(\gamma_2) + D_{n^\zero}(n_{j_0}) + D_{n^\zero}(n_{j_0+1}) + k(\gamma)M).
\end{split}
\end{equation}
We have $2 D_{n^\zero}(\gamma'_1) = 2 \min (D_{n^\zero}(\gamma'_1), D_{n^\zero}(n_{j_0})) < 2 T \|\gamma\|^{1/5} < \rho_{3t+4} \tilde \kappa (\|\gamma_1\| + \|\gamma_2\| + |n_{j_0+1} - n_{j_0}| )/4$. Similarly, $2 D_{n^\zero}(\gamma'_2) < \rho_{3t+4} \tilde \kappa (\|\gamma_1\| + \|\gamma_2\| + |n_{j_0+1} - n_{j_0}|)/4$. Therefore, \eqref{eq:11gamscaledconditionaldecay} follows from \eqref{eq:11gamscaledconditionaldecay7}.

Thus, \eqref{eq:11gamscaledconditionaldecay} holds for $\gamma \in \Gamma^{(t)}_{n^\zero,1}$ in any event. Let $\gamma = (n_1,\dots,n_{k}) \in \Gamma^{(t)}_{n^\zero,0}$. Assume $\|(n_1,\dots,n_{k})\| \le 2 R_{t-1}$. Recall that $\max_j |n_{j+1} - n_j| \le R_t$ since $\gamma = (n_1,\dots,n_{k}) \in \Gamma^{(t)}_{n^\zero,0}$. Hence, in this case the inductive assumption applies and even a stronger estimate than \eqref{eq:11gamscaledconditionaldecay} holds. Assume $\|(n_1,\dots,n_{k})\| > 2 R_{t-1}$. Then there exists $j_0$ such that $\|(n_1,\dots,n_{j_0})\| \le 2 R_{t-1}$, $\|(n_1,\dots,n_{j_0+1})\| > 2 R_{t-1}$. Let $\gamma_1 = (n_1,\dots,n_{j_0})$, $\gamma_2 = (n_{j_0+1},\dots,n_k)$. Note that $\|\gamma_2\| = \|\gamma\| - \|(n_1,\dots,n_{j_0+1})\| < 2 R_t - 2 R_{t-1} < 2 R_{t-1}$ since $R_t < 2 R_{t-1}$. Therefore $\gamma_1,\gamma_2 \in \Gamma^{(t-1)}_{n^\zero}$. Hence, the inductive assumption applies,
\begin{equation} \label{eq:11gamscaledconditionaldecay10}
\begin{split}
w_{n^\zero}(\gamma_i) & \le \tilde \ve^{k(\gamma_i)-1} \exp \Big( -\frac{15}{16} (1 - \sigma_{3t+1}) \tilde \kappa \|\gamma_i\| + 2 \bar D(\gamma_i) + k(\gamma_i) M \Big), \quad i=1,2, \\
w_{n^\zero}(\gamma) & = w_{n^\zero}(\gamma_1) |c(n_{j_0+1} - n_{j_0})| w_{n^\zero}(\gamma_2) \\
& \le \tilde \ve^{k(\gamma)-1} \exp \Big( -\frac{15}{16} (1 - \sigma_{3t+1}) \tilde \kappa (\|\gamma_1\| + \|\gamma_2\|) - \frac{15}{16} (1 - \sigma_{3t+3}) \tilde \kappa |n_{j_0+1} - n_{j_0}| \\
& \qquad + 2 \bar D(\gamma_1) + 2 \bar D(\gamma_2) + k(\gamma)M \Big).
\end{split}
\end{equation}
Let $D(n_{j_i}) = \bar D(\gamma_i)$, $i=1,2$. We have the following cases:

$(\alpha)$ Assume $j_1 < j_0$. In this case, due to \eqref{eq:11auxtrajectweight5NNNNN}, we have
\begin{equation}\label{eq:11auxtrajectweight5NUPPal}
2 \min (\bar D(\gamma_1), \bar D(\gamma_2)) \le 2 T \| \gamma \|^{1/5} < \frac{\rho_{3t+4}}{4} \tilde \kappa (\|\gamma_1\| + \|\gamma_2\| + |n_{j_0+1} - n_{j_0}|).
\end{equation}
Combining \eqref{eq:11gamscaledconditionaldecay10} with \eqref{eq:11auxtrajectweight5NUPPal}, we obtain \eqref{eq:11gamscaledconditionaldecay}.

$(\beta)$ Assume $j_0 + 1 < j_2$. Similarly to case $(\alpha)$, one verifies \eqref{eq:11gamscaledconditionaldecay}.

$(\gamma)$ Assume $j_1 = j_0$, $j_0 + 1 = j_2$. Let $\gamma'_1 = (n_1, \dots, n_{j_0-1})$, $\gamma'_2 = (n_{j_0 + 2}, \dots, n_k)$. Once again, applying the inductive assumption, we obtain
\begin{equation} \label{eq:11gamscaledconditionaldecay7be}
\begin{split}
w_{n^\zero}(\gamma'_i) & \le \tilde \ve^{k(\gamma'_i)-1} \exp \Big(-\frac{15}{16} (1 - \sigma_{3t+1}) \tilde \kappa \|\gamma'_i\| + 2 \bar D(\gamma'_i) + k(\gamma_i') M \Big), \quad i=1,2, \\
w_{n^\zero}(\gamma) & = w_{n^\zero}(\gamma'_1) |c(n_{j_0-1} - n_{j_0})|  \exp(D_{n^\zero}(n_{j_0})) |c(n_{j_0+1} - n_{j_0})| \exp(D_{n^\zero} (n_{j_0+1})) \\
|c(n_{j_0+2} - n_{j_0+1})| w_{n^\zero}(\gamma'_2) & \le \tilde \ve^{k(\gamma)-1} \exp(-\frac{15}{16} (1 - \sigma_{3t+1}) \tilde \kappa (\|\gamma'_1\| + \|\gamma'_2\|)) \\
& \quad \times \exp(- \frac{15}{16} (1 - \sigma_{3t+3}) \tilde \kappa (|n_{j_0} - n_{j_0-1}| + |n_{j_0+1} - n_{j_0}| + |n_{j_0+2} - n_{j_0+1}|)) \\
& \quad \times \exp(2 \bar D(\gamma'_1) + 2 \bar D(\gamma'_2) +  D_{n^\zero}(n_{j_0}) +  D_{n^\zero}(n_{j_0+1}) + k(\gamma)M).
\end{split}
\end{equation}
We have $2 \bar D_{n^\zero}(\gamma'_1) = 2 \min (\bar D_{n^\zero}(\gamma'_1), D_{n^\zero}(n_{j_0})) < 2 T \|\gamma\|^{1/5} < \rho_{3t+4} \tilde \kappa (\|\gamma'_1\| + \|\gamma'_2\| + |n_{j_0+2} - n_{j_0-1}| )/4$. Similarly, $2 \bar D_{n^\zero}(\gamma'_2) < \rho_{3t+4} \tilde \kappa (\|\gamma_1\| + \|\gamma_2\| + |n_{j_0+1} - n_{j_0}| )/4$. Therefore, \eqref{eq:11gamscaledconditionaldecay} follows from \eqref{eq:11gamscaledconditionaldecay7be}.

Thus \eqref{eq:11gamscaledconditionaldecay} holds in any event. Recall that $\bar D(\gamma) \le 2 T [(\min (|m|, |n^\zero - n|)^{1/5} + \|\gamma\|^{1/5}]$, $\gamma \in \Ga_{n^\zero}(m, n)$; see Lemma~\ref{lem:11.sumsproperties}. Set $w'_{n^\zero}(\gamma) = \exp(- 2 T (\min (|m|,|n^\zero-n|)^{1/5}) w_{n^\zero}(\gamma)$. Note that $2 T \|\gamma\|^{1/5} < \rho_{3t+5} \tilde \kappa \|\gamma\|/4$ if $\|\gamma\| \ge R_t$. Due to the elementary estimate \eqref{eq:PAexpsumomega1} in Lemma~\ref{lem:PAL10omegV1} one has for any $\alpha, k > 0$,
\begin{equation}\label{eq:11auxtrajectweight21a}
\sum_{\gamma \in \Ga(m, n; k, \mathfrak{T})} \exp(-\alpha \|\gamma\|) < (8 \alpha^{-1})^{(k-1)\nu}.
\end{equation}
Set $\Ga^{(t)}_{n^\zero}(m, n) = \Ga_{n^\zero}(m, n) \cap \Ga^{(t)}_{n^\zero}$. Note that if $\gamma \in \Ga_{n^\zero}$ and $\|\gamma\| \le R_{t+1}$, then $\gamma \in \Ga^{(t)}_{n^\zero}$ and \eqref{eq:11gamscaledconditionaldecay} applies. Finally, $\|\gamma\| \ge |n-m|$ for any $\gamma \in \Ga_{n^\zero}(m, n)$. Taking all that into account, we obtain
\begin{equation} \label{eq:11gamscaledconditionaldecaysum1}
\begin{split}
& \sum_{\gamma \in \Ga_{n^\zero}(m, n) : k(\gamma) \ge 2, \|\gamma\| \le R_{t_0}} w'_{n^\zero}(\gamma)  \le \sum_{t \le t_0-1} \sum_{\gamma \in \Ga_{n^\zero}(m, n), k(\gamma) \ge 2, R_t \le \|\gamma\| \le R_{t+1}} w'_{n^\zero}(\gamma) \\
& \le e^M \sum_{t \le t_0-1} \sum_{\gamma \in \Ga_{n^\zero}(m, n),k(\gamma) \ge 2, R_t \le \|\gamma\| \le R_{t+1}} (e^M \tilde \ve)^{k(\gamma)-1} \exp(-\frac{15}{16} (1 - \sigma_{3t+4}) \tilde \kappa \|\gamma\| + \rho_{3t+5} \tilde \kappa \|\gamma\|/4) \\
& \le e^M \sum_{k \ge 1} (e^M \tilde \ve)^{k-1} \exp(-\frac{15}{16} (1 - \sigma_{3t_0+2}) \tilde \kappa |n-m|) \sum_{k(\gamma) = k} \exp(-\rho_{3t_0+2} \tilde \kappa \|\gamma\|/4) \\
& \le e^M \exp(-\frac{15}{16} (1 - \sigma_{3t_0+2}) \tilde \kappa |n-m|) \sum_{k \ge 1} (e^M \tilde \ve)^{k-1}(8 \alpha^{-1})^{(k-1)\nu} \le \exp(-\frac{15}{16} (1 - \sigma_{3t_0+2}) \tilde \kappa |n-m|).
\end{split}
\end{equation}
Here, in the last step, $\alpha = \rho_{3t+5}\tilde \kappa/4$, and we have used $\tilde \ve < \ve_0$.
\end{proof}

\begin{proof}[Proof of part $(2)$ of Theorem~$\tilde B$]
Set $R_1 = 2^{30}(\kappa_0^{-1}T)^2$, $R_t = 5R_{t-1}/4$, $\rho_{t-1} = 2^{-10} t^{-2}$, $t = 2,\dots$, $\sigma_t = \sum_{1 \le \ell \le t} \rho_\ell$ as in Lemma~\ref{lem:11scaleddecayestimate}. Set also $\ve^\zero_0 := \exp(-2 R_2)$. One can see that $\ve^\zero_0 < \ve_0^4$. Assume that
\begin{equation} \label{eq:11theoremBgapcondition}
E^+(\La;k_{n^\zero}) - E^-(\La;k_{n^\zero}) \le \ve^\zero \exp(-\kappa^\zero |n^\zero|), \quad \text{ for all $n^\zero \in \mathfrak{T} \setminus \{0\}$},
\end{equation}
where $\ve^\zero < \ve^\zero_0$, $\kappa^\zero > 4 \kappa_0$. We can replace \eqref{eq:11gapsineqstatement} by the following estimate,
\begin{equation} \label{eq:11-10acbasicfunctionsBineq2}
|c(n^\zero)| \le (\ve^\zero)^{3/4} \exp(-3 \kappa^\zero |n^\zero|/4) + \sum_{m', n' \in \La(n^\zero)} |c(m')| s(n^\zero;m',n') |c(n'- n^\zero)|.
\end{equation}

Assume that with some $(\ve^\zero)^{1/2} < \hat \ve \le \ve^\zero_0$ and $\kappa_0 \le \hat \kappa \le \kappa^\zero/2$, we have $|c(p)|\le \hat \ve \exp(-\hat \kappa |p|)$ for $|p| > 0$. Set $\tilde \ve = \hat \ve/2$, $\tilde \kappa = 7 \hat \kappa/6$. We claim that in this case, in fact,
\begin{equation} \label{eq:11scaledconditionaldecayY}
|c(p)| \le \begin{cases} \tilde \ve \exp(-\tilde \kappa |p|) & \text {if $0 < |p| \le R_2$}, \\ \tilde \ve \exp(-\frac{15}{16} (1 - \sigma_{3t}) \tilde \kappa |p|) & \text {if $R_{t-1} < |p| \le R_t$, $t \ge3$}.
\end{cases}
\end{equation}
It is important to note here that $\frac{15}{16} (1 - \sigma_{3t}) \tilde \kappa > (\frac{15}{16})^2 \frac{7}{6} \hat \kappa := L \hat \kappa$, with $L > 1$. This allows one to iterate the argument and Theorem~$\tilde B$ follows.

The verification of the claim goes by induction in $t$, starting with the first line in \eqref{eq:11scaledconditionaldecayY}, and then with the help of Lemma~\ref{lem:11scaleddecayestimate}. The idea is to run $n^\zero$ in \eqref{eq:11-10acbasicfunctionsBineq2} and to combine the inequalities which we have for different $n^\zero$'s. To this end it is convenient to replace $n^\zero$ in the notation. To verify the first line in \eqref{eq:11scaledconditionaldecayY}, we invoke \eqref{eq:11gamscaledconditionaldecaysum1statement} from Lemma~\ref{lem:11scaleddecayestimate} with $\hat \ve$ in the role of $\tilde \ve$ and $\hat \kappa$ in the role of $\tilde \kappa$. Note that condition \eqref{eq:11scaledconditionaldecay} of Lemma~\ref{lem:11scaleddecayestimate} holds for any $t$ for trivial reasons. So,
\begin{equation}\label{eq:11auxtrajectweightO5B}
\begin{split}
& \sum_{m,n \in \La_{p}} |c(m)| s(p;m,n) |c(p-n)| \\
& \quad \le \hat \ve^2 \sum_{m,n \in \La_{p}} \exp(-\hat \kappa |m|) \exp \Big( -\frac{15}{16} (1 - \sigma_{3t_0+2}) \hat \kappa |n-m| + 2 T (\min(|m|,|n-p|))^{1/5} \Big) \exp(-\hat \kappa |p-n|).
\end{split}
\end{equation}
Elementary estimations  in \eqref{eq:11auxtrajectweightO5B} yield
\begin{equation}\label{eq:11auxtrajectweightO5B1}
\sum_{m,n \in \La_{p}} |c(m)| s(p;m,n) |c(p-n)| \le \hat \ve^{3/2}/4 \le \tilde \ve (\ve^\zero_0)^{1/2}/2 < (\tilde \ve /2) \exp(-R_2).
\end{equation}
It follows from \eqref{eq:11-10acbasicfunctionsBineq2} combined with \eqref{eq:11auxtrajectweightO5B1} that for any $|p| > 0$, we have
\begin{equation}\label{eq:11auxtrajectweightO5B2}
|c(p)| < (\ve^\zero)^{3/4} \exp(-3 \kappa^\zero |p|/4) + (\tilde \ve/2) \exp(-R_2) < \tilde \ve \exp(-\tilde \kappa R_2).
\end{equation}
This verifies the first line in \eqref{eq:11scaledconditionaldecay}.

Assume now that for some $\ell \ge 2$, \eqref{eq:11scaledconditionaldecay} holds for any $0 < |p| \le R_{t}$ and any $t \le \ell$. Let $|q| > R_\ell$ be arbitrary. For $t \ge 1$, let $\Gamma^{(t)}_{q}$ be the set of trajectories $\gamma = (n_1,\dots,n_{k}) \in \Ga_{q}$ with $\|\gamma\| \le 2 R_t$ and $\max_j |n_{j+1} - n_j| \le R_{t+1}$. Let $\Ga^{(t)}_q(m,n) = \Ga^{(t)}_{q} \cap \Ga_{q}(m,n)$. We have
\begin{equation}\label{eq:11auxtrajectweightO5}
\begin{split}
\sum_{m,n \in \La_{q}} |c(m)| s(q;m,n) |c(q-n)| & \le \sum_{m,n} |c(m)| |c(q-n)| \sum_{\gamma \in \Ga_{q}(m,n)} w_{q} (\gamma) \\
& \le \Sigma_1 + \Sigma_2 + \Sigma_3 \\
& := \sum_{m,n : |m|,|n-q| \le R_\ell} |c(m)| |c(q-n)| \sum_{\gamma \in \Ga^{(\ell-1)}_{q}(m,n)} w_{q} (\gamma) \\
& \quad + \sum_{m,n} |c(m)| |c(q-n)| \sum_{\gamma \in \Ga_{q}(m,n), \quad \|\gamma\| > 2 R_{\ell-1}} w_{q} (\gamma) + \Sigma_3.
\end{split}
\end{equation}
Here the sum $\Sigma_3$ is over the cases when $\|\gamma\| \le 2 R_{\ell-1}$ and either $\max (|m|,|q-n|) > R_\ell$ or $\gamma = (n_1,\dots,n_{k})$ obeys $\max_j |n_{j+1}-n_j| > R_{\ell}$, or both.

Using \eqref{eq:11gamscaledconditionaldecaysum1statement} from Lemma~\ref{lem:11scaleddecayestimate} with $\ell$ in the role of $t_0$ and the inductive assumption, one obtains
\begin{equation}\label{eq:11auxtrajectweightO6}
\begin{split}
\Sigma_1 \le \tilde \ve^2 \sum_{m,n} \exp \Big( -\frac{15}{16} (1 - \sigma_{3 \ell}) \tilde \kappa |m| \Big) \exp \Big( -\frac{15}{16} (1 - \sigma_{3 \ell + 2}) \tilde \kappa |m-n| + 2 T (\min(|m|,|q-n|))^{1/5} \Big) \times \\
\exp \Big( -\frac{15}{16} (1 - \sigma_{3\ell}) \tilde \kappa|q-n| \Big).
\end{split}
\end{equation}
Note that $2 T (\min(|m|,|q-n|))^{1/5}) < \frac{1}{4} \rho_{3 \ell + 3} (|m| + |m-n| + |q-n|)$ since $|q| > R_\ell$. Estimating the sum in \eqref{eq:11auxtrajectweightO6}, one obtains
\begin{equation}\label{eq:11auxtrajectweightO6a}
\Sigma_1 \le \tilde \ve^{3/2} \exp \Big( -\frac{15}{16} (1 - \sigma_{3\ell+2} - \frac{1}{4} \rho_{3\ell+3}) \tilde \kappa |q| \Big).
\end{equation}

To estimate the sum $\Sigma_2$, we use Lemma~\ref{lem:11scaleddecayestimate} with $\hat \ve$ in the role of $\tilde \ve$ and $\hat \kappa$ in the role of $\tilde\kappa$:
\begin{equation}\label{eq:11auxtrajectweightO15}
\begin{split}
\Sigma_2 & \le \sum_{m,n} |c(m)| |c(q-n)| \sum_{t \ge \ell-1} \sum_{\gamma \in \Ga_{q}(m, n), k(\gamma) \ge 2, R_t \le \|\gamma\| \le R_{t+1}} w_{q}(\gamma) \\
& \le \sum_{m,n} |c(m)| |c(q-n)| \exp(2 T (\min(|m|,|q-n|))^{1/5}) e^M \times \\
& \quad \big[ \sum_{\gamma \in \Ga_{q}(m, n),k(\gamma) \ge 2, 2 R_{\ell-1} \le \|\gamma\| \le R_{\ell}} + \sum_{t \ge \ell} \sum_{\gamma \in \Ga_{q}(m,n),k(\gamma) \ge 2, R_t \le \|\gamma\| \le R_{t+1}} \big] \\
& \quad (e^M \hat \ve)^{k(\gamma) - 1} \exp \Big( -\frac{15}{16} (1 - \sigma_{3t+4}) \hat \kappa \|\gamma\| + 2 T \|\gamma\|^{1/5} \Big) \\
& \le \hat \ve^2 \sum_{m,n} \exp(-\hat \kappa (|m| + |q-n|) + 2 T (\min(|m|,|q-n|))^{1/5}) \times \\
& \quad \exp \Big( -\frac{15}{16} (1 - \sigma_{3\ell+2}) \hat \kappa \times (2 R_{\ell-1}) \Big) \\
& \le \hat \ve^{3/2} \exp \Big( -\frac{15}{16} (1 - \sigma_{3\ell+2} - \frac{1}{4} \rho_{3\ell+3}) \hat \kappa \times (2R_{\ell-1}) \Big) \\
& < \hat\ve^{3/2} \exp \Big( -\frac{15}{16} (1 - \sigma_{3\ell+2} - \frac{1}{4} \rho_{3\ell+3}) \tilde \kappa R_{\ell+1} \Big).
\end{split}
\end{equation}

Let us now estimate $\Sigma_3$. Given $r,s \in \La_q$ with $|s-r| > R_\ell$, denote by $\Ga_{q;r,s}$ the set of trajectories $\gamma = (n_1,\dots,n_{k}) \in \bar \Ga_{q}$ with $\|\gamma\| \le 2 R_{\ell-1}$ and such that
\begin{equation}\label{eq:11gammarsplit}
\gamma = \gamma' \cup \gamma'',
\end{equation}
where $r$ is the endpoint of $\gamma'$ and $s$ is the starting point of $\gamma''$. Note that since $\|\gamma\| \le 2 R_{\ell-1}$, one has $|n_{j+1} - n_j| \le R_\ell$ for all $j$ with one exception when $n_{j+1} = s$, $n_j = r$. In particular, the inductive assumption applies to $\gamma'$, $\gamma''$. Denote by $\Sigma'_3$ the part of sum $\Sigma_3$ with $\gamma \in \Ga_{q;r,s}$ and with $|m|, |q-n| \le R_\ell$. Then just as in the above derivations, one obtains
\begin{equation}\label{eq:11auxtrajectweightO16}
\Sigma'_3 \le \tilde \ve^{3/2} \sum_{|r-s| > R_\ell} \exp \Big( -\frac{15}{16} (1 - \sigma_{3\ell+2}) \tilde \kappa |r| \Big) |c(r-s)| \exp \Big( -\frac{15}{16} (1 - \sigma_{3 \ell + 2}) \tilde \kappa|q-s| \Big).
\end{equation}
The estimation of the rest of the sum $\Sigma_3$ is similar. One has
\begin{equation}\label{eq:11auxtrajectweightO17}
\begin{split}
\Sigma_3 & \le \tilde \ve^{3/2} \sum_{|r-s| > R_\ell} \exp \Big( -\frac{15}{16} (1 - \sigma_{3\ell+2}) \tilde \kappa |r| \Big) |c(r-s)| \exp \Big( -\frac{15}{16} (1 - \sigma_{3 \ell + 2}) \tilde \kappa |q-s| \Big) \\
& \quad + 2 \tilde \ve^{3/2} \sum_{|r| > R_\ell} |c(r)| \exp \Big( -\frac{15}{16} (1 - \sigma_{3\ell+2}) \tilde \kappa |q-r| \Big) \\
& \le \tilde \ve^{3/2} \sum_{R_\ell < |r-s| \le R_{\ell+1}} \exp \Big( -\frac{15}{16} (1 - \sigma_{3\ell+2}) \tilde \kappa|r| \Big) |c(r-s)| \exp \Big( -\frac{15}{16} (1 - \sigma_{3 \ell + 2}) \tilde \kappa |q-s| \Big) \\
& \quad + 2 \tilde \ve^{3/2} \sum_{R_\ell < |r| \le R_{\ell+1}} |c(r)| \exp \Big( -\frac{15}{16} (1 - \sigma_{3 \ell + 2}) \tilde \kappa |q-r| \Big) \\
& \quad \hat \ve^{3/2} \exp \Big( -\frac{15}{16} (1 - \sigma_{3 \ell + 2} - \frac{1}{4} \rho_{3 \ell + 3}) \tilde \kappa R_{\ell + 1} \Big).
\end{split}
\end{equation}

Now we invoke \eqref{eq:11-10acbasicfunctionsBineq2}. For $|q| > R_\ell$, one obtains
\begin{equation} \label{eq:11-10acbasicfunctionsBineq2F1}
\begin{split}
|c(q)| & \le (\ve^\zero)^{3/4} \exp(- \kappa^\zero |q|/2) + \sum_{1 \le i \le 3} \Sigma_i \le (\ve^\zero)^{3/4} \exp(-\tilde\kappa|q|) \\
& \quad + \tilde \ve^{3/2} \exp \Big( -\frac{15}{16} (1 - \sigma_{3 \ell + 2} - \frac{1}{4} \rho_{3 \ell + 3}) \tilde \kappa |q| \Big) + 2 \hat \ve^{3/2} \exp \Big( -\frac{15}{16} (1 - \sigma_{3 \ell + 2} - \frac{1}{4} \rho_{3 \ell + 3}) \tilde \kappa R_{\ell + 1} \Big) \\
& \quad + \tilde \ve^{3/2} \sum_{R_\ell < |r-s| \le R_{\ell + 1}} \exp \Big( -\frac{15}{16} (1 - \sigma_{3 \ell + 2}) \tilde \kappa |r| \Big) |c(r-s)| \exp \Big( -\frac{15}{16} (1 - \sigma_{3 \ell + 2}) \tilde \kappa |q-s| \Big) \\
& \quad + 2 \tilde \ve^{3/2} \sum_{R_\ell < |r| \le R_{\ell + 1}} |c(r)| \exp \Big( -\frac{15}{16} (1 - \sigma_{3 \ell + 2}) \tilde \kappa |q-r| \Big).
\end{split}
\end{equation}

Here we have replaced $\kappa^\zero/2$ by $\tilde \kappa < \kappa^\zero/2$. Now we consider $R_\ell < |q| \le R_{\ell+1}$. We replace $R_{\ell+1}$ in the exponent by a smaller quantity $|q|$ and we obtain a self-contained system of inequalities for $|c(q)|$ with $R_\ell < |q| \le R_{\ell+1}$. This allows us to iterate  \eqref{eq:11-10acbasicfunctionsBineq2F1}. It is convenient to replace the multiple sums via summation over trajectories $\gamma = (n_0,\dots,n_k) \in \Ga(0,q)$. Set
\begin{equation} \label{eq:11-10acbasicfunctionsBineq2F2}
\begin{split}
\ve' = \tilde \ve^{1/4}, \quad \kappa' & = \frac{15}{16} (1 - \sigma_{3\ell+2} - \frac{1}{4} \rho_{3\ell+3}) \tilde \kappa, \quad w'(m,n) = \ve' \exp(-\kappa' |m-n|), \\
w'((n_0,\dots,n_k) & = \prod w'(n_j,n_{j+1}).
\end{split}
\end{equation}
Iterating \eqref{eq:11-10acbasicfunctionsBineq2F1} $N$ times, one obtains
\begin{equation} \label{eq:11-10acbasicfunctionsBineq2F3}
\begin{split}
|c(q)| & \le \tilde \ve^{3/2} (\sum_{0 \le k \le N} 4^k \ve'^k) \exp \Big( -\frac{15}{16} (1 - \sigma_{3 \ell + 2} - \frac{1}{4} \rho_{3 \ell + 3}) \tilde \kappa R_{\ell + 1} \Big) \\
& \quad + \tilde \ve \sum_{1 \le k \ge 3N} 4^{k} \ve'^k \sum_{\gamma \in \Ga(0,q) : k(\gamma) = k} w'(\gamma) + 4^N \ve'^N.
\end{split}
\end{equation}
Taking here $N$ large enough and evaluating the sums over $\gamma$ as before, one obtains \eqref{eq:11scaledconditionaldecay} which implies the statement in part $(2)$ of Theorem~$\tilde B$.
\end{proof}

\section{Proof of Theorem~$\tilde I$}\label{sec.4}

The proof of Theorem~$\tilde I$ is a simple combination of Theorem~$\tilde B$ with fundamental facts from the theory of Hill's equation. These facts play an instrumental role also in the proof of Theorem~$I$. We refer for these facts to ~\cite{McKvM}, ~\cite{McKTr}, ~\cite{KaPo}, ~\cite{PoTr},\cite{Du}
 \cite{Tr}.

$(a)$ Let $Q(x)$ be a real $T$-periodic smooth function,
\begin{equation} \label{eq:8potential}
Q(x) = \sum_{n \in \mathbb{Z} \setminus \{ 0 \}} d(n) e^{\frac{2\pi i nx}{T}}.
\end{equation}
Consider the Hill equation
\begin{equation} \label{eq:8Hill}
[H y](x) = -y''(x) + Q(x) y(x) = E y(x), \quad x \in \IR.
\end{equation}
Let $y_1(x, \lambda)$, $y_2(x, \lambda)$ be fundamental solutions of \eqref{eq:8Hill} defined via the initial conditions $y_1(0, \lambda) = 1$, $y'_1(0, \lambda) = 0$, $y_2(0, \lambda) = 0$, $y'_2(0, \lambda) = 1$. The function
\begin{equation} \label{eq:4HillDisc}
\triangle (\lambda) = y_1(T, \lambda) + \partial_x y_2(T, \lambda)
\end{equation}
is called the Hill discriminant. The function $\triangle(\lambda)$ is analytic. Let $H_{\pm P}$ be the operator $H$ restricted to the interval $[0,T]$ with periodic, respectively, anti-periodic boundary conditions. The number $\lambda$ is an eigenvalue of $H_{\pm P}$ if and only if
\begin{equation} \label{eq:4HillDiscequation}
\triangle (\lambda) = \pm 2.
\end{equation}
The roots of $\triangle^2 = 4$ can be enumerated as
\begin{equation} \label{eq:4Hillroots}
E_0 < E^-_1 \le E^+_1 < E^-_2 \le E^+_2 < \cdots,
\end{equation}
where $E^\pm_n$ are the periodic, respectively, anti-periodic eigenvalues, depending on $n$ being even or odd, respectively. These eigenvalues obey the following asymptotics,
$$
E^\pm_n = \frac{\pi^2 n^2}{T^2} + O(T^2 n^{-1}).
$$
The spectrum of $H$ consists of the following set,
\begin{equation} \label{eq:8Hillspec}
\widetilde{\mathcal{S}} = [E_0,\infty) \setminus \bigcup_{n \in \mathcal{O}} (E_n^-,E_n^+),
\end{equation}
where $\mathcal{O} = \{ n : E_n^- < E_n^+ \}$. The intervals $G_n \eqdef (E_n^-, E_n^+)$, $n \in \mathcal{O}$ are called the gaps.

$(b)$ Consider the operator $H$ restricted to the interval $[0,T]$ with Dirichlet boundary conditions. Its spectrum consists of eigenvalues $\mu_1 < \mu_2 < \cdots$. The closure of the gap $G_n$, $n \in \mathcal{O}$ contains the eigenvalue $\mu_n$ and no other Dirichlet eigenvalues. The following trace formula holds,
\begin{equation} \label{eq:8tracefla}
Q(0) = \sum_{n \in \mathcal{O}} E_n^+ + E^-_n - 2 \mu_n.
\end{equation}
The length of the gap decays at least like $O(T^2 n^{-2})$ $($depending on the smoothness of $Q$, see more comments below$)$. Therefore the series converges nicely.

$(c)$ Consider the gap $G_n = (E_n^-,E_n^+)$. Set
\begin{equation}\label{eq:8toruspm-2}
\begin{split}
G_{n,\sigma} & = \{ (\lambda,\sigma) : \lambda \in G_n\}, \; \sigma = \pm , \\
\mathcal{C}_n & = G_{n,-} \cup G_{n,+} \cup \{E_n^-,E_n^+\}.
\end{split}
\end{equation}
The set $\mathcal{C}_n$ has a natural smooth structure which makes it diffeomorphic to the circle $\mathbb{T} = \{ e^{i \theta} : 0 \le \theta \le 2 \pi \}$. The points $E_n^\pm$ are identified with $\theta = 0, \pi$, respectively. The sign $\sigma$ is considered as a function on $\mathcal{C}_n$ with the convention that $\sigma(\theta) = +$ for $0 < \theta < \pi$, $\sigma(\theta) = -$ for $\pi < \theta < 2 \pi$, $\sigma(0) = \sigma(\pi) = 0$.

$(d)$ Given $t \in \mathbb{R}$, consider the potential $Q(\cdot + t)$. By unitary equivalence, the periodic spectrum for this potential is the same as for $Q$. Denote by $\mu_n(t)$ the Dirichlet eigenvalues of $Q(\cdot + t)$. Due to the trace formula \eqref{eq:8tracefla}, the eigenvalues $\mu_n(t)$ recover the potential:
\begin{equation} \label{eq:8traceflat}
Q(t) = \sum_{n \in \mathcal{O}} E_n^+ + E^-_n - 2\mu_n(t).
\end{equation}

$(e)$ Denote by $\mathcal{I} \mathcal{S} \mathcal{O} (Q)$ the set of $T$-periodic smooth real functions isospectral with $Q$. The map $F : W \mapsto \mu = (\mu_n)_{n \in \mathcal{O}} \in \mathcal{C} = \prod_{n \in \mathcal{O}} \mathcal{C}_n$ is a bijection from the set $\mathcal{I} \mathcal{S} \mathcal{O} (Q)$ onto the torus $\mathcal{C}$.

$(f)$ The functions $\mu_n(t)$ are smooth. For each $t$, the value $\mu_n(t)$ can be uniquely identified with a point on the torus $\mathcal{C}_n$ so that these functions obey the following system of differential equations,
\begin{equation}\label{eq:8ODEsystem}
\dot \mu_n = \Psi_n(\mu) := \sigma(\mu_n(t))\sqrt{ 4(\mu_n - \underline{E} )(E_n^+ - \mu_n)(\mu_n - E_n^-) \prod_{i \neq n} \frac{(E_i^+ - \mu_n)( E_i^- - \mu_n)}{(\mu_i - \mu_n)^2} }.
\end{equation}
When $\mu_n(t)$ reaches $E_n^-$, it does not pause, but passes through this point, crossing from $G_{n,+}$ to $G_{n,-}$. Similarly, when $\mu_n(t)$ reaches $E_n^+$, it passes through this point, crossing from $G_{n,-}$ to $G_{n,+}$. The point $\mu_n(t)$ rotates clockwise on the circle $\mathcal{C}_n$ when $t$ runs in the interval $[0,T]$.

$(g)$ For any given initial point $\mu^\zero = (\mu^\zero_n)_{n \in \mathcal{O}} \in \mathcal{C}$, there exists a unique global solution $\mu(t) = (\mu_n(t))_{n \in \mathcal{O}}$, $t \in \mathbb{R}$ with $\mu(0) = \mu^\zero$. In particular, the map $\Phi : \mu^\zero \mapsto W$, defined via the (right-hand side of the) trace formula \eqref{eq:8traceflat} is the inverse for the map $F$.

$(h)$ The solutions $\mu(t)$ of the system \eqref{eq:8ODEsystem} are stable with respect to the initial
data $\mu(0)$. Namely, set
$$
d((\mu_n)_{n \in \mathcal{O}}, (\lambda_n)_{n \in \mathcal{O}}) = \sum_{n \in \mathcal{O}} d_n(\mu_n, \lambda_n),
$$
where $d_n$ stands for the natural distance one the circle $\mathcal{C}_n$. Then,
\begin{equation}\label{eq:8ODEsystability}
d(\mu(t), \lambda(t)) \le K e^{Lt} d(\mu(0), \lambda(0)),
\end{equation}
where the constants $K,L$ depend on the spectrum $\widetilde{\mathcal{S}}$, that is, on $Q$.

$(i)$ Assume that
\begin{equation} \label{eq:8potentialanalytic}
Q(x) = \sum_{n \in \mathbb{Z} \setminus \{0\}} d(n) e^{\frac{2\pi i nx}{T}}
\end{equation}
is analytic, that is,
\begin{equation} \label{eq:8coeffexpdecay}
|d(n)| \le A \exp \left( -\eta \frac{|n|}{T} \right)
\end{equation}
with some constants $A, \eta > 0$. Then the gap lengths obey
\begin{equation} \label{eq:8Hillgapsexpdecay}
(E_n^+ - E_n^-) \le A_1 \exp \left( -\eta_1 \frac{|n|}{T} \right).
\end{equation}
Let $W$ be a continuous real $T$-periodic function. Assume $W \in \mathcal{I} \mathcal{S} \mathcal{O} (Q)$. Then,
\begin{equation} \label{eq:8potentialanalyticQ}
W(x) = \sum_{n \in \mathbb{Z} \setminus \{0\}} g(n) e^{\frac{2\pi i nx}{T}}
\end{equation}
with
\begin{equation} \label{eq:8coeffexpdecayQ}
|g(n)| \le A_1 \exp \left( -\eta_1 \frac{|n|}{T} \right).
\end{equation}

\medskip

Recall how periodic potentials arose in our global context by replacing the vector $\omega$ with a vector $\tilde \omega$ having rational components. For such periodic potentials $\tilde V$, we have the representation involving $\zv$ (resp., the group $\mathfrak{Z}(\tilde \omega)$) of the potential and the gaps inherited from this global setting, but also the representation involving $\Z$ that comes from the general periodic theory summarized above. It is natural to connect these two representations, and the following definition and lemma will be useful in this regard.

\begin{defi}\label{defi:twometrics}
Let $\mathfrak{m} \in \mathfrak{Z}(\tilde \omega)$. One has $\mathfrak{m} \tilde \omega = \frac{\ell(\mathfrak{m})}{T}$ with $\ell(\mathfrak{m}) \in \mathbb{Z}$, where $T = \tau_0^{-1}$, see Lemmas~\ref{lem:PA10omegalattice1} and \ref{lem:PA10omegalattice1a}. We define a new distance on $\mathfrak{Z}(\tilde \omega)$ by setting $|\mathfrak{m}|_\mathbb{R} = \frac{|\ell(\mathfrak{m})|}{T}$. Note that $\mathfrak{Z}(\tilde \omega) \ni \mathfrak{m} \mapsto \ell(\mathfrak{m}) \in \Z$ is an isomorphism.
\end{defi}

We need the following elementary properties of the distances $|\mathfrak{m}|$ and $|\mathfrak{m}|_\mathbb{R}$.

\begin{lemma}\label{lem:PAL10omegagropnorms}
$(0)$ $|\mathfrak{m} \tilde \omega| \le |\mathfrak{m}| |\tilde \omega|$. \\
$(1)$ $|-\mathfrak{m}| = |\mathfrak{m}|$, $|h \mathfrak{m}| \le |h| |\mathfrak{m}|$ for any $\mathfrak{m} \in \mathfrak{Z}(\tilde \omega)$ and any $h \in \mathbb{Z}$. \\
$(2)$ For any $\mathfrak{m}$, we have $c_{\tilde \omega} |\mathfrak{m}|_\mathbb{R} \le |\mathfrak{m}| \le C_{\tilde \omega} |\mathfrak{m}|_\mathbb{R}$, where $c_{\tilde \omega}, C_{\tilde \omega} \in (0,\infty)$ are constants.
\end{lemma}

\begin{proof}
$(0)$ Let $m_1 \in \mathfrak{m}$ be such that $|\mathfrak{m}| = |m_1|$. Then $|\mathfrak{m} \tilde \omega| = |m_1 \tilde \omega| \le |m_1| |\tilde \omega| = |\mathfrak{m}| |\tilde \omega|$.

$(1)$ Note that the coset $-[m]$ consists of all vectors $-m_1$ with $m_1 \in [m]$. This implies $|-\mathfrak{m}| = |\mathfrak{m}|$. Let $h \in \mathbb{Z}$, $h > 0$. Then $h \mathfrak{m} = \mathfrak{m} + \mathfrak{m} + \cdots + \mathfrak{m}$ and $|h \mathfrak{m}| \le h |\mathfrak{m}|$ follows from the subadditivity property $|\mathfrak{m}_1 + \mathfrak{m}_2| \le |\mathfrak{m}_1| + |\mathfrak{m}_2|$.

$(2)$ There exists $\mathfrak{m}_0$ such that $\ell(\mathfrak{m}_0) = 1$. Since $\ell$ is an isomorphism, $\mathfrak{m} = \ell(\mathfrak{m}) \mathfrak{m}_0$ for any $\mathfrak{m}$. This implies
$$
|\mathfrak{m}| =| \ell(\mathfrak{m}) \mathfrak{m}_0| \le |\ell(\mathfrak{m})| |\mathfrak{m}_0| = C_{\tilde \omega} |\mathfrak{m}|_\mathbb{R},
$$
with $C_{\tilde \omega} = T |\mathfrak{m}_0|$. This proves the right-hand side inequality. Due to the last definition, we have $|\mathfrak{m}|_\mathbb{R} = |\mathfrak{m} \tilde \omega| \le c_{\tilde \omega}^{-1}| \mathfrak{m}|$, with $c_{\tilde \omega} = |\tilde \omega|^{-1}$.
\end{proof}

Now we turn to the proof of Theorem~$\tilde I$.

\begin{proof}[Proof of Theorem~$\tilde I$]
Let $V \in \mathcal{P}(\tilde \omega, \ve, \kappa_0)$. Let $\ve^\zero = \ve^\zero (a_0, b_0, \kappa_0)$ be as in Theorem~$\tilde B$. Assume $0 < \ve < \ve^\zero$. Let
$$
\cS = [\underline{E} , \infty) \setminus \bigcup_{m \in \mathfrak{Z}(\tilde \omega) \setminus \{ 0 \} : k_\mathfrak{m} > 0, \; E^-_\mathfrak{m} < E^+_\mathfrak{m}} (E^-_\mathfrak{m}, E^+_\mathfrak{m}),
$$
$\underline{E} = E(0)$, $E^\pm_\mathfrak{m} = E^\pm( k_\mathfrak{m}))$ be as in Theorem~$\tilde B$. Re-enumerate the eigenvalues $E^\pm_m$ via $E^\pm_\ell = E^\pm_{\mathfrak{m}}$, with $\ell = \ell(\mathfrak{m})$ defined in Lemma~\ref{lem:PAL10omegagropnorms}. Let $\mathcal{O} = \{ \ell \in \mathbb{Z} : \ell = \ell(\mathfrak{m}), \; E^-_{\mathfrak{m}} < E^+_{\mathfrak{m}} \}$.

The function $V$ is $T$-periodic with $T = \tau_0^{-1}$,
\begin{equation} \label{eq:8potentialVinell}
V(x) = \sum_{\ell \in \mathbb{Z} \setminus \{ 0 \}} c(\ell) e^{\frac{2 \pi i \ell x}{T}},
\end{equation}
where $c(\ell) := c(\mathfrak{m})$ with $\ell = \ell(\mathfrak{m})$; see Lemma~\ref{lem:PA10omegalattice1tildeV}. Let $Q \in \mathcal{I} \mathcal{S} \mathcal{O} (V)$. One has
\begin{equation} \label{eq:8potentialQinell}
Q(x) = \sum_{\ell \in \mathbb{Z} \setminus \{ 0 \}} d(\ell) e^{\frac{2 \pi i \ell x}{T}}.
\end{equation}
We can rewrite \eqref{eq:8potentialQinell} as
\begin{equation} \label{eq:8potentialQinell1}
Q(x) = \sum_{\mathfrak{m} \in \mathfrak{Z}(\tilde \omega) \setminus \{ 0 \}} d(\mathfrak{m}) e^{2 \pi i x \mathfrak{m} \tilde \omega}
\end{equation}
with $d(\mathfrak{m}) = d(\ell(\mathfrak{m}))$, $\mathfrak{m} \in \mathfrak{Z}(\tilde \omega) \setminus \{ 0 \}$.

Consider the circles $\mathcal{C}_n$ and the torus $\mathcal{C} = \prod_{\ell \in \mathcal{O}} \mathcal{C}_\ell$. Let $\mu \in \mathcal{C}$. Let $Q = \Phi(\mu)$ be the periodic potential defined as in part $(g)$ of the fundamental facts from the theory of Hill's equation, listed in the beginning of this section. Take an arbitrary $0 < \alpha < 1$. Let $\mathcal{G}_{\alpha}$ be the collection of all $\mu \in \mathcal{C}$ such that $Q = \Phi(\mu)$ has the representation \eqref{eq:8potentialQinell1} with
\begin{equation} \label{eq:8potentialQinellest}
|d(\mathfrak{m})| \le \sqrt{2 \ve} \exp \Big( -\frac{\kappa_0}{2} |\mathfrak{m}|^{\alpha} \Big).
\end{equation}
Let us verify that, actually, $\mathcal{G}_\alpha = \mathcal{C}$. Clearly $\mathcal{G}_\alpha$ is not empty since it contains the point $\mu^\zero$ with $\Phi (\mu^\zero) = V$. The set $\mathcal{G}_\alpha$ is obviously closed. We will now check that the set $\mathcal{G}_\alpha$ is also open. Note first of all that due to $(i)$, for any $Q \in \mathcal{I} \mathcal{S} \mathcal{O} (V)$, we have
\begin{equation} \label{eq:8potentialanalyticQ11}
Q(x) = \sum_{\mathfrak{m} \in \mathfrak{Z}(\tilde \omega) \setminus \{ 0 \}} d(\mathfrak{m}) e^{2 \pi i x \mathfrak{m} \tilde \omega}
\end{equation}
with
\begin{equation} \label{eq:8coeffexpdecayQ11}
|d(\mathfrak{m})| \le A \exp(-\eta |\mathfrak{m}|_\mathbb{R})
\end{equation}
and $A, \eta > 0$ being some constants depending on $V$. Due to Lemma~\ref{lem:PAL10omegagropnorms} we have
\begin{equation} \label{eq:8coeffexpdecayQ12}
|d(\ell)| \le A \exp(-\gamma |\mathfrak{m}|), \quad \ell = \ell(\mathfrak{m}),
\end{equation}
with $\gamma = \gamma(V) > 0$. Take $M = M(V, \alpha)$ large enough so that for $\ell = \ell(\mathfrak{m})$ with
$|\mathfrak{m}| > M$, we have
\begin{equation} \label{eq:8coeffexpdecayQ13}
|d(\ell)| < \ve_0 \exp \left( -\frac{\kappa_0}{4} |\mathfrak{m}|^\alpha \right).
\end{equation}
Set
$$
\delta = K^{-1} e^{-LT} \frac{\ve_0}{2} \exp(-\frac{\kappa_0}{4} M^\alpha).
$$
Let $\mu \in \mathcal{G}_\alpha$, and consider $\lambda \in \mathcal{C}$ with $d(\lambda,\mu) < \delta$. Let $Q = \Phi(\mu)$, $W = \Phi(\lambda)$. Due to $(h)$,
\begin{equation}\label{eq:8ODEsystability1}
\max_{0 \le t \le T} d(\mu(t), \lambda(t)) \le K e^{LT} \delta.
\end{equation}
Combining this with the trace formula \eqref{eq:8traceflat}, we obtain
\begin{equation}\label{eq:8ODEsystability2}
\max_{0 \le t \le T} |Q(t) - W(t)| \le K e^{LT} \delta.
\end{equation}
Let $d(\ell)$, $g(\ell)$, $\ell \in \mathbb{Z}$ be the Fourier coefficients of $Q$ and $W$, respectively. Then,
\begin{equation}\label{eq:8ODEsystability3}
|d(\ell) - g(\ell)| = \Big| \int_0^T (Q(t) - W(t)) e^{\frac{2 \pi i \ell t}{T}} \frac{dt}{T} \Big| \le \max_{0 \le t \le T} d(\mu(t), \lambda(t)) \le K e^{LT} \delta.
\end{equation}
Since $\mu \in \mathcal{G}_\alpha$, \eqref{eq:8potentialQinellest} holds. Assume $\ve < \ve_0^2/8$. Then,
\begin{equation}\label{eq:8ODEsystability4}
|g(\ell)| \le \sqrt{2 \ve} \exp \left( -\frac{\kappa_0}{2} |\mathfrak{m}|^{\alpha} \right) + K e^{LT} \delta < \ve_0 \exp \left( -\frac{\kappa_0}{4} |\mathfrak{m}|^\alpha \right), \quad \ell = \ell(\mathfrak{m})
\end{equation}
for $\ell = \ell(\mathfrak{m})$ with $|\mathfrak{m}| \le M$. Thus, the estimate \eqref{eq:8coeffexpdecayQ13} holds for all $\ell \in \mathbb{Z}$.

Now we want to apply part $(2)$ of Theorem~$\tilde B$ with $W$ in the role of $V$, $\kappa_0/4$ in the role of $\kappa_0$, $\kappa_0$ in the role of $\kappa'_0$, and with $\alpha_0' = \alpha_0 = \alpha$. The reason we pick these values of the parameters is that due to part $(1)$ in Theorem~$\tilde B$, we have $E^+_\mathfrak{m} - E^-_\mathfrak{m} \le 2 \ve \exp(-\kappa_0 |\mathfrak{m}|)$. The coefficients $g(\mathfrak{m})$ obey $|g(\mathfrak{m})| < \ve_0 \exp(-\frac{\kappa_0}{4} |\mathfrak{m}|^\alpha)$ for all $\mathfrak{m}$. Let $\ve^\zero = \ve^\zero (a_0, b_0, \kappa_0/4)$ be defined as in part $(2)$ of Theorem~$\tilde B$. Set
$$
\ve^\one = \min \left( \ve_0^2/8, \ve^\zero/2 \right)
$$
and assume $0 < \ve < \ve^\one$. Now we apply part $(2)$ of Theorem~$\tilde B$ and conclude that the Fourier coefficients $g(\mathfrak{m})$ in fact obey $|g(\mathfrak{m})| \le \sqrt{2 \ve} \exp(-\frac{\kappa_0}{2} |\mathfrak{m}|^{\alpha})$. This means $W \in \mathcal{G}_\alpha$. Thus, we have verified that $\mathcal{G}_\alpha$ is open.

Since $\mathcal{C}$ is connected, we conclude that $\mathcal{G}_\alpha = \mathcal{C}$. This means that for any $\mu \in \mathcal{C}$, $Q = \Phi(\mu)$ has the representation \eqref{eq:8potentialQinell1} and \eqref{eq:8potentialQinellest} holds. Now we can send $\alpha \to 1$ to obtain
\begin{equation} \label{eq:8potentialQinellest11}
|d(\mathfrak{m})| \le \sqrt{2 \ve} \exp \left( -\frac{\kappa_0}{2} |\mathfrak{m}| \right),
\end{equation}
as claimed in Theorem~$\tilde I$. The rest of the statements in Theorem~$\tilde I$ are due to the general facts from the theory of Hill's equation.
\end{proof}

\section{Periodic Approximations}\label{sec.5}

Let $\omega = (\omega_1, \ldots, \omega_\nu) \in \mathbb{R}^\nu$. Assume that the following Diophantine condition holds,
\begin{equation}\label{eq:PAI7-5-85a-zz}
|n \omega| \ge a_0 |n|^{-b_0}, \quad n \in \mathbb{Z}^\nu \setminus \{ 0 \}
\end{equation}
for some
\begin{equation}\label{eq:PAIombasicTcondition5a-zz}
0 < a_0 < 1, \quad \nu < b_0 < \infty.
\end{equation}
Assume also that each component $\omega_j$ obeys
\begin{equation}
\label{eq:diophant-zz}
\| n \omega_j \| \geq \frac{c}{|n|^{\beta}} \mbox{\ \ \ for all $n \in \mathbb{Z} \setminus \{ 0 \}$},
\end{equation}
$0 < c < 1 < \beta$.

\begin{lemma}\label{lem:PAomegas}
Given an arbitrary $r = 1, 2, \ldots$, there exists a vector $\tilde \omega = \tilde \omega^{(r)} = (\tilde \omega_1^{(r)}, \ldots, \tilde \omega_\nu^{(r)})$ with rational components $\tilde \omega_j := \tilde \omega_j^{(r)} = \ell_j^{(r)}/t_j^{(r)}$, $\ell_j^{(r)}, t_j^{(r)} \in \mathbb{Z}$ such that the following conditions hold,
$$
|\omega_j - \tilde \omega_j^{(r)}| < 1/r, \quad r < t_j^{(r)} \le c^{-1} r^\beta.
$$
\end{lemma}

\begin{proof}
Let $[a_{j,1}, a_{j,2}, \ldots]$ be the continued fraction expansion of $\omega_j$. Let $\omega_j^{(s)} = [a_{j,1}, a_{j,2}, \ldots, a_{j,s}] = p_j^{(s)}/q_j^{(s)}$ be the convergents for $\omega_j$. Recall that
$$
|\omega_j - \omega_j^{(s)}| \le \frac{1}{q_j^{(s)} q_j^{(s+1)}}.
$$
Combining this inequality with condition \eqref{eq:diophant-zz}, we get
\begin{equation}\label{eq:diophantap1}
q_j^{(s+1)} \le c^{-1} (q^\es_j)^{\beta}.
\end{equation}

Given $r$, find $s_j = s_j(r)$ such that $q_j^{(s_j)} \le r < q_j^{(s_j+1)}$. Note that due to \eqref{eq:diophantap1}, we have
$$
q_j^{(s_j+1)} \le c^{-1} (q_j^{(s_j)})^{\beta} \le c^{-1} r^{\beta}.
$$
Set $\tilde \omega_j^{(r)} = \ell_j^{(r)}/t_j^{(r)} = p_j^{(s_j+1)}/q_j^{(s_j+1)}$. We have $|\omega_j - \omega_j^{(s_j+1)}| < 1/q_j^{(s_j+1)} < 1/r$ and we are done.
\end{proof}

\begin{corollary}\label{cor:5Diophantine}
Using the notation from the previous lemma, the following ``Diophantine condition in the box'' holds,
\begin{equation}\label{eq:5PAI7-5-8-zz}
|n \tilde \omega| \ge a'_0 |n|^{-b_0}, \quad 0 < |n| \le \bar R_0,
\end{equation}
with $a_0' = \frac{a_0}{2}$, $\bar R_0 := (\nu^{-1}a'_0r)^{1/(b_0+1)}$. For $r > r_0(a_0, b_0, \nu)$, we also have
\begin{equation}\label{eq:5PAIombasicTcondition-zz}
(\bar R_0)^{\bar b_0} > \prod t_j
\end{equation}
with $\bar b_0 = \bar b_0 (a_0, b_0, \beta, \nu)$.
\end{corollary}

\begin{proof}
We have
\begin{equation}\label{eq:5PAI7-5-8}
|n \tilde \omega| \ge |n \omega| - |n| |\tilde \omega - \omega| \ge a_0 |n|^{-b_0} - |n| \nu r^{-1} \ge \frac{a_0}{2} |n|^{-b_0}, \quad 0 < |n| \le \bar R_0.
\end{equation}
Since $t_j \le r$, the second statement follows.
\end{proof}

\begin{remark}\label{rem:5PA1}
$(1)$ Let $c(n) \in \mathbb{C}$, $n \in \zv$ obey
\begin{equation}\label{eq:517-4}
\begin{split}
\overline{c(n)} & = c(-n), \quad n \in \zv , \\
|c(n)| & \le \ve \exp(-\kappa_0|n|^{\alpha_0}), \quad n \in \zv \setminus \{ 0 \},
\end{split}
\end{equation}
$0 < \kappa_0, \alpha_0 \le 1$. Let $\tilde \omega = \tilde \omega^{(r)}$ be as in the last corollary. Consider the function
\begin{equation}\label{eq:52Vtilde}
\tilde V(x) = \sum_{n \in \zv} c(n) e^{2 \pi i xn\tilde \omega}\ , \quad x \in\mathbb{R}
\end{equation}
and the Hill equation
\begin{equation} \label{eq:517-1}
[\hat H y](x) = -y''(x) +\tilde V(x) y(x) = Ey(x), \quad x, \ve \in \IR.
\end{equation}
Due to the last corollary, Theorems~$\tilde A$ and $\tilde B$ apply to the potential $\tilde V(x)$. Moreover, all quantities like $\ve_0$ and all estimates are uniform in $r$.

$(2)$ Let $\tilde \omega = (\tilde \omega_1, \ldots, \tilde \omega_\nu) \neq 0$ be a vector with rational components $\tilde \omega_j = \ell_j/t_j$, $\ell_j, t_j \in \mathbb{Z}$. Let as usual $\mathfrak{N}(\tilde \omega) := \{ m \in \mathbb{Z}^\nu : m \omega = 0 \}$, $\mathfrak{Z}(\tilde \omega) = \zv/\mathfrak{N}(\tilde \omega)$, $\mathfrak{F}(\tilde \omega) := \{ m \tilde \omega : m \in \mathbb{Z}^\nu\}$. Recall that $(i)$ $\mathfrak{F}(\tilde \omega) := \{ a \tau_0 : a \in \mathbb{Z}\}$, where $\tau_0 \ge |\bar T^{-1}|$, $\bar T = \prod_j t_j$, $(ii)$ the map $\iota : \mathfrak{m} \to \mathfrak{m} \tilde \omega \in \mathfrak{F}(\tilde \omega)$, $\mathfrak{m} \in \mathfrak{Z}(\tilde \omega)$ is a group isomorphism, see Lemmas~\ref{lem:PA10omegalattice1} and \ref{lem:PA10omegalattice1a}.

We use the notation $[n]_{\tilde \omega} := [n]$ for the coset $n + \mathfrak{N}(\tilde \omega)$, $n \in \mathbb{Z}^\nu$. We denote by $|\mathfrak{n}|$ the standard quotient distance in $\mathfrak{Z}(\tilde \omega)$, $|\mathfrak{n}| := |\mathfrak{n}|_{\tilde \omega} = \min \{ |n| : n \in \mathfrak{n} \}$, $\mathfrak{n} \in \mathfrak{Z}(\tilde \omega)$. Let $\tilde V$ be as in \eqref{eq:52Vtilde}. Recall that due to Lemma~\ref{lem:PA10omegalattice1tildeV} one has $\tilde V(x+T) = \tilde V(x)$, $x \in \mathbb{R}$, with $T := T(\tilde \omega) := \tau_0^{-1}$.
\end{remark}

\section{The Distance Between the Gaps}\label{sec.6}

In this section we derive estimates relating the length of the gaps with the distance between them. We discuss in detail the setting of Theorem~$\tilde A$. To simplify the notation we assume that $\alpha_0 = 1$ in this theorem, since this is the only case we are interested in in the rest of this paper. We remark after this discussion that completely similar conclusions hold in the setting of Theorem A from \cite{DG}.

\begin{lemma}\label{lem2:homogenuityplus}
Using the notation from Theorem~$\tilde A$, the following statements hold:

$(1)$ For any $\mathfrak{m}' \neq \mathfrak{m}$ with $|\mathfrak{m}'| \ge |\mathfrak{m}|$, we have
$$
\dist ([E_\mathfrak{m}^-, E_\mathfrak{m}^+], [E_{\mathfrak{m}'}^-, E_{\mathfrak{m}'}^+])) \ge a |\mathfrak{m}'|^{-b},
$$
where $E^\pm_\mathfrak{n} := E^\pm(k_\mathfrak{n})$, $a, b > 0$ are constants depending on $a_0, b_0, \kappa_0, \nu$.

$(2)$ Let $\mathfrak{m}' \neq \mathfrak{m}$ be such that $|\mathfrak{m}| \ge |\mathfrak{m}'|$ and
\begin{equation}\label{eq:2resonance}
(E_{\mathfrak{m}'}^+ - E_{\mathfrak{m}'}^-) \ge (\dist ([E_\mathfrak{m}^-, E_\mathfrak{m}^+], [E_{\mathfrak{m}'}^-, E_{\mathfrak{m}'}^+]))^4.
\end{equation}
Then, in fact, $|\mathfrak{m}| \ge \tau \exp(\gamma |\mathfrak{m}'|)$, where $\tau, \gamma > 0$ depend on $a_0, b_0, \kappa_0, \nu$.

$(3)$ Using the notation of part $(2)$, denote by $\mathfrak{R}_\mathfrak{m}$ the set of all $\mathfrak{m}' \neq \mathfrak{m}$ with $|\mathfrak{m}'| \le |\mathfrak{m}|$ such that condition \eqref{eq:2resonance} holds. Then $\mathfrak{R}_\mathfrak{m}$ can be enumerated as follows: $n^{(r)}_\mathfrak{m}$, $r = 1, \ldots, r_0$, $|n^{(r-1)}_\mathfrak{m}| \ge |n^{(r)}_\mathfrak{m}|$,
$$
|n^{(r-2)}_\mathfrak{m}| \ge a \exp \left(\kappa |n^{(r)}_\mathfrak{m}| \right),
$$
where $a, \kappa > 0$ depend on $a_0, b_0, \kappa_0, \nu$, $r = 2, 3, \ldots$.

$(4)$ For any $\mathfrak{m} \neq 0$, we have
$$
E_\mathfrak{m}^- - \underline{E}\ge a|\mathfrak{m}|^{-b}.
$$

$(5)$
$$
E^+_\mathfrak{m} \le 2 (k_\mathfrak{m}+1)^2 \le C(|\tilde \omega| |\mathfrak{m}| + 1)^2.
$$
\end{lemma}

\begin{proof}
Recall from Lemma~\ref{lem:PAL10omdiophsubst} that
$$
|\mathfrak{m} \tilde\omega| > a_0 |\mathfrak{m}|^{-b_0} \quad \text{for any $\mathfrak{m} \neq 0$}.
$$
In what follows we denote by $a_j$ constants depending on $a_0,b_0,\kappa_0,\nu$.

$(1)$ Let $\mathfrak{m}' \neq \mathfrak{m}$, $|\mathfrak{m}| \le |\mathfrak{m}'|$ be arbitrary. Then,
$$
|k_\mathfrak{m} - k_{\mathfrak{m}'}| = |(\mathfrak{m} - \mathfrak{m}')\tilde \omega|/2 \ge a_0 (2 |\mathfrak{m}'|)^{-b_0}/2 \ge a_1 |\mathfrak{m}'|^{-b_0}.
$$
Assume for instance that $k_{\mathfrak{m}'} > k_{\mathfrak{m}} > 0$. Due to \eqref{eq:1Ekk1EGT11} in Theorem~$\tilde A$, we have
$$
E^-(k_{\mathfrak{m}'}) - E^+(k_{\mathfrak{m}}) > (k^0)^2 (k_{\mathfrak{m}'} - k_\mathfrak{m})^2 \ge (k^0)^2 a_2 |\mathfrak{m}'|^{-2b_0},
$$
where
$$
k^\zero := \min (\ve_0, k_{\mathfrak{m}} /1024) \ge \ve_0 a_3 |\mathfrak{m}'|^{-b_0}.
$$
Thus,
$$
E^-(k_{\mathfrak{m}'}) - E^+(k_{\mathfrak{m}}) > (k^0)^2 (k_{\mathfrak{m}'} - k_\mathfrak{m})^2 \ge \ve_0^2 a_4 | \mathfrak{m}'|^{-4 b_0} = a_5 |\mathfrak{m}'|^{-4 b_0},
$$
which means that
$$
\dist ([E_\mathfrak{m}^-, E_\mathfrak{m}^+], [E_{\mathfrak{m}'}^-, E_{\mathfrak{m}'}^+]) \ge a |\mathfrak{m}'|^{-b}.
$$
with $a=a_5$ and $b=4b_0$. The remaining cases are completely similar.

$(2)$ Just as in $(1)$ we obtain
$$
|\mathfrak{m}| \ge a_6 \dist ([E_\mathfrak{m}^-, E_\mathfrak{m}^+], [E_{\mathfrak{m}'}^-, E_{\mathfrak{m}'}^+]))^{-\frac{1}{4 b_0}},
$$
since this time we assume $|\mathfrak{m}| \ge |\mathfrak{m}'|$. Combining this with \eqref{eq:2resonance} we obtain
\begin{equation}\label{eq:3gapslenthsAG}
|\mathfrak{m}| \ge a_6 (E_{\mathfrak{m}'}^+ - E_{\mathfrak{m}'}^-)^{-\frac{1}{16 b_0}}.
\end{equation}
On the other hand, by Theorem~$\tilde B$,
\begin{equation}\label{eq:3gapslenths}
E^+(k_{\mathfrak{m}'}) - E^-(k_{\mathfrak{m}'}) < 2 \ve \exp \Big( -\frac{\kappa_0}{2} |\mathfrak{m}'| \Big).
\end{equation}
Combining \eqref{eq:3gapslenthsAG} with \eqref{eq:3gapslenths} we obtain the statement.

$(3)$ Let us enumerate $\mathfrak{R}_\mathfrak{m}$ as $n^{(r)}_\mathfrak{m}$, $r = 1, \ldots, r_0$,
so that
$$
|n^{(r-1)}_\mathfrak{m}| \ge |n^{(r)}_\mathfrak{m}|
$$
Let $\mathfrak{m}',\mathfrak{m}''\in\mathfrak{R}_\mathfrak{m}$, $\mathfrak{m}' \neq \mathfrak{m}''$,
$|\mathfrak{m}'| \ge |\mathfrak{m}''|$.
Using condition \eqref{eq:2resonance} and Theorem $\tilde B$ get
\begin{equation}\label{eq:2resonancerep1}
\begin{split}
\dist ([E_{\mathfrak{m}'}^-, E_{\mathfrak{m}'}^+], [E_{\mathfrak{m}''}^-, E_{\mathfrak{m}''}^+])
\le \dist ([E_{\mathfrak{m}'}^-, E_{\mathfrak{m}'}^+], [E_{\mathfrak{m}}^-, E_{\mathfrak{m}}^+])
+\dist ([E_{\mathfrak{m}}^-, E_{\mathfrak{m}}^+], [E_{\mathfrak{m}''}^-, E_{\mathfrak{m}''}^+]),\\
+(E_{\mathfrak{m}}^+- E_{\mathfrak{m}}^-)\le (E_{\mathfrak{m}'}^+ - E_{\mathfrak{m}'}^-)^{1/4}+(E_{\mathfrak{m}''}^+ - E_{\mathfrak{m}''}^-)^{1/4}+(E_{\mathfrak{m}}^+- E_{\mathfrak{m}}^-)\\
\le [2 \ve \exp \Big( -\frac{\kappa_0}{2} |\mathfrak{m}'|\Big )]^{1/4}
+[2 \ve \exp \Big( -\frac{\kappa_0}{2} |\mathfrak{m}''|\Big )]^{1/4}+
2 \ve \exp \Big( -\frac{\kappa_0}{2} |\mathfrak{m}|\Big )<\exp \Big( -\frac{\kappa_0}{8} |\mathfrak{m}''|\Big ).
\end{split}
\end{equation}
On the other hand by $(1)$
$$
|\mathfrak{m}'| \ge a_6 \dist ([E_{\mathfrak{m}'}^-, E_{\mathfrak{m}'}^+], [E_{\mathfrak{m}''}^-, E_{\mathfrak{m}''}^+]))^{-\frac{1}{4 b_0}},
$$
and the statement follows.

$(4)$ The proof is completely similar to $(1)$ and we omit it.

$(5)$ This follows from \eqref{eq:1Ekk1EGT11} in Theorem $\tilde A$.
\end{proof}

\begin{remark}\label{rem:6pot1}
Let $V$ be as in Theorem $\tilde A$. Consider the Schr\"odinger equation
\begin{equation} \label{eq:1-1zz}
[H\psi](x) := -\psi''(x) +  V(x) \psi(x) = E \psi(x), \qquad x \in \IR,
\end{equation}
Theorems $\tilde A$ and $\tilde B$ in the current work are versions of Theorems~A and B in \cite{DG}. In this work we need a version of  Lemma~\ref{lem2:homogenuityplus} in the setting of equation \eqref{eq:1-1zz}. We restate here Theorems~A and B from \cite{DG}. It is clear from the statements of Theorems~A and B that actually  Lemma~\ref{lem2:homogenuityplus} and its proof work word for word in this setting as well. We do not restate it in order to avoid redundant repetitions. Later we refer to  Lemma~\ref{lem2:homogenuityplus} for \textbf{both settings}.
\end{remark}

\bigskip

We now restate Theorems A and B from \cite{DG}. Set
\begin{equation}\label{eq:6K.1}
\begin{split}
k_n & = -n\omega/2, \quad n \in \zv \setminus \{0\}, \quad \mathcal{K}(\omega) = \{ k_n : n \in \zv \setminus \{0\} \}, \\
\mathfrak{J}_n & = ( k_n - \delta(n), k_n + \delta(n) ), \quad \delta(n) = a_0 (1 + |n|)^{-b_0-3}, \quad n \in \zv \setminus \{0\}, \\
\mathfrak{R}(k) & = \{ n \in \zv \setminus \{0\} : k \in \mathfrak{J}_n \}, \quad \mathfrak{G} = \{ k : |\mathfrak{R}(k)| < \infty \},
\end{split}
\end{equation}
where $a_0,b_0$ are as in the Diophantine condition \eqref{eq:1PAI7-5-85a}. Let $k \in \mathfrak{G}$ be such that $|\mathfrak{R}(k)| > 0$. Due to the Diophantine condition, one can enumerate the points of $\mathfrak{R}(k)$ as $n^{(\ell)}(k)$, $\ell = 0, \dots, \ell(k)$, $1 + \ell(k) = |\mathfrak{R}(k)|$, so that $|n^{(\ell)}(k)| < |n^{(\ell+1)}(k)|$. Set
\begin{equation}\label{eq:6mjdefi}
\begin{split}
T_{m}(n) & = m - n ,\quad m, n \in \mathbb{Z}^\nu, \\
\mathfrak{m}^{(0)}(k) & = \{ 0, n^{(0)}(k) \}, \quad \mathfrak{m}^{(\ell)}(k) = \mathfrak{m}^{(\ell-1)}(k) \cup T_{n^{(\ell)}(k)}(\mathfrak{m}^{(\ell-1)}(k)), \quad \ell = 1, \dots, \ell(k).
\end{split}
\end{equation}

\begin{thma}
There exists $\ve_0 = \ve_0(\ka_0, a_0, b_0) > 0$ such that for $0 < \ve < \ve_0$ and $k \in \mathfrak{G} \setminus \frac{\omega}{2}(\zv \setminus \{0\})$, there exist $E(k) \in \mathbb{R}$ and $\vp(k) := (\vp(n;k))_{n \in \zv}$ such that the following conditions hold:

$(a)$ $\vp(0; k) = 1$,
\begin{equation} \label{eq:6-17evdecay1A}
\begin{split}
|\vp(n ;k)| & \le \ve^{1/2} \sum_{m \in \mathfrak{m}^{(\ell)}} \exp \Big( -\frac{7}{8} \kappa_0 |n-m| \Big), \quad \text{ $n \notin \mathfrak{m}^{(\ell(k))}(k)$}, \\
|\vp(m;k)| & \le 2 \quad \text{for any $m \in \mathfrak{m}^{(\ell(k))}(k)$.}
\end{split}
\end{equation}

$(b)$ The function
$$
\psi(k, x) = \sum_{n \in \zv} \vp(n; k) e^{2 \pi i x (n \omega + k)}
$$
is well-defined and obeys equation \eqref{eq:1-1zz} with $E = E(k)$, that is,
\begin{equation}\label{eq:6.sco}
H \psi(k,x) \equiv  - \psi''(k,x) + V(x) \psi(k,x) = E(k) \psi(k,x).
\end{equation}

$(c)$
$$
E(k) = E(-k), \quad \varphi(n;-k) = \overline{\varphi(-n; k)}, \quad \psi(-k, x)=\overline{\psi(k, x)},
$$
\begin{equation}\label{eq:6Ekk1EGT11}
\begin{split}
(k^0)^2 (k - k_1)^2  < E(k) - E(k_1) < 2k (k - k_1) + 2 \ve \sum_{k_1 < k_{n} < k} \delta(n), \quad 0 < k - k_1 < 1/4, \; k_1 > 0,
\end{split}
\end{equation}
where $k^\zero := \min(\ve_0, k/1024)$.

$(d)$ The spectrum of $H$ consists of the following set,
$$
\cS = [E(0) , \infty) \setminus \bigcup_{m \in \zv \setminus \{0\} : E^-( k_m) < E^+( k_m)} (E^-( k_m), E^+( k_m)),
$$
where
$$
E^\pm(k_m) = \lim_{k \to k_m \pm 0, \; k \in \mathfrak{G} \setminus \mathcal{K}(\omega)} E(k), \quad \text{ for $k_m>0$.}
$$
\end{thma}

\begin{thmb}
$(1)$ The gaps $(E^-(k_m), E^+(k_m))$ in Theorem~A obey $E^+(k_m) - E^-(k_m) \le 2 \ve \exp(-\frac{\kappa_0}{2} |m|)$.

$(2)$ Using the notation from Theorem~A, there exists $\ve^\zero > 0$ such that if the gaps $(E^-(k_m), E^+(k_m))$ obey $E^+(k_m) - E^-(k_m) \le \ve \exp(-\kappa |m|)$ with $0 < \ve < \ve^\zero$, $\kappa > 4 \kappa_0$, then, in fact, the Fourier coefficients $c(m)$ obey $|c(m)| \le \ve^{1/2} \exp(-\frac{\kappa}{2} |m|)$.
\end{thmb}

\section{Convergence of the Spectra of the Rational Approximants}\label{sec.8}

For $\omega \in \mathbb{R}^\nu$, let $H_\omega = -\Delta + U(\omega x)$, with $U:\mathbb{T}^\nu \to \mathbb{R}$ as before and $\omega \in \mathbb{R}^\nu$.

We prove an upper bound on distances of points in the spectra of $H_\omega$ and $H_{\omega'}$, by adapting a result proved in \cite{AMS90, CEY90} for discrete Schr\"odinger operators.

\begin{theorem}\label{Tdistspec}
If $U\in C^1$, then there are constants $C, \tilde C>0$ which depend only on $\lVert U \rVert_\infty$ and $\lVert \nabla U \rVert_\infty$ such that for every $E \in \spec (H_\omega)$ and every $\omega'$ with $\lvert \omega - \omega'\rvert < \tilde C$,
\[
\dist(E, \spec (H_{\omega'}) ) \le C(1+\lvert E\rvert^{1/4} ) \lvert \omega - \omega' \rvert^{1/2}.
\]
\end{theorem}

To prove the theorem, we will need the following lemma.

\begin{lemma}
Let $H = - \Delta+ V$ be a Schr\"odinger operator on $\mathbb{R}$ with $V$ bounded, and let $E \in \spec(H)$. For any $\epsilon \in (0,1)$ and $L>1$,  there exist $a\in \mathbb{R}$, $\phi \in D(H)$ such that $\supp \phi \subset [a-L, a+L]$ and
\[
\lVert (H-E)\phi \rVert < \left( \epsilon + \frac C{L} (1 +\lvert E\rvert^{1/2}) \right) \lVert \phi \rVert
\]
where $C$ depends only on $\inf V$.
\end{lemma}

\begin{proof} Since $E \in \spec (H)$, we may pick $\psi \in D(H)$ such that
\[
\lVert (H-E) \psi \rVert < \frac \epsilon{\sqrt 2} \lVert \psi \rVert.
\]
Let $\eta \in C^2$ be such that $\supp \eta \subset [-1,1]$ and $\lVert \eta\rVert^2 = 1$. For $L>0$, denote
\[
\eta_L(x) = L^{-1/2} \eta(x/L).
\]
With $\eta^{(k)}$ denoting the $k$-th derivative of $\eta$,
\[
\lVert \eta_L^{(k)} \rVert^2 = L^{-2k} \lVert \eta^{(k)} \rVert^2.
\]
Denote also $\eta_{L,a}(x) = \eta_L(a-x)$. It is easy to see
\[
\int \lVert \eta_{L,a} \psi \rVert^2 da = \lVert \eta_{L} \rVert^2  \lVert \psi \rVert^2 = \lVert \psi \rVert^2
\]
since that integral is just the $L^1$-norm of the convolution $\lvert \eta_L\rvert^2  * \lvert \psi\rvert^2$, and likewise
\begin{equation}\label{ml02}
\int \lVert \eta_{L,a} (H-E)\psi \rVert^2 da = \lVert \eta_{L} \rVert^2  \lVert (H-E)\psi \rVert^2 =  \lVert (H-E)\psi \rVert^2 < \frac{\epsilon^2}2 \lVert \psi \rVert^2.
\end{equation}
Next, we compute
\[
[\eta_{L,a}, H-E] \psi = 2 \eta'_{L,a} \psi' + \eta''_{L,a} \psi.
\]
Using the inequality
\[
\lVert u + v \rVert^2 \le 2 \lVert u \rVert^2 + 2 \lVert v \rVert^2,
\]
this implies
\[
\int \lVert [\eta_{L,a}, (H-E)] \psi \rVert^2 da \le 8 \lVert \eta'_L \rVert^2 \lVert \psi' \rVert^2 + 2 \lVert \eta''_L \rVert^2 \lVert \psi \rVert^2.
\]
We can simplify that by using
\begin{equation}\label{ml01}
\lVert \psi' \rVert^2 = \langle \psi, H \psi \rangle - \int V(x) \lvert \psi(x) \rvert^2 dx < (E+\epsilon - \inf V) \lVert \psi \rVert^2,
\end{equation}
so that we conclude
\begin{equation}\label{ml03}
\int \lVert [\eta_{L,a}, (H-E)] \psi \rVert^2 da < \frac {C^2}{2 L^2} (1+ \lvert E \rvert) \lVert \psi \rVert^2,
\end{equation}
where
\[
C^2 = 20 \max( \lVert \eta' \rVert^2, \lVert \eta''\rVert^2) (1 + \lvert \inf V \rvert ).
\]
We can then combine \eqref{ml02} and \eqref{ml03} to see
\begin{align*}
\int \lVert (H-E)(\eta_{L,a} \psi) \rVert^2 da & < \left( \epsilon^2 +  \frac {C^2}{L^2}(1+ \lvert E \rvert) \right) \lVert \psi \rVert^2 \\
& = \int \left( \epsilon^2 +  \frac {C^2}{L^2}(1+ \lvert E \rvert) \right) \lVert \eta_{L,a} \psi \rVert^2 da.
\end{align*}
Thus, for some $a\in \mathbb{R}$,
\[
 \lVert (H-E)(\eta_{L,a} \psi) \rVert^2< \left( \epsilon^2 +  \frac {C^2}{L^2} (1+\lvert E\rvert) \right) \lVert \eta_{L,a} \psi \rVert^2.
\]
Taking $\phi = \eta_{L,a} \psi$ and taking the square root completes the proof.
\end{proof}

\begin{proof}[Proof of Theorem~\ref{Tdistspec}]
Let $E \in \spec (H_\omega)$. By the previous lemma, for any $\epsilon \in (0,1)$ and $L>1$ there exists $\phi \in D(H_\omega)$ such that
\[
\lVert (H_\omega - E) \phi \rVert < \left( \epsilon + \frac {C_1}L (1 +\lvert E\rvert^{1/2}) \right) \lVert \phi \rVert
\]
and $\supp \phi \subset [a-L, a+L]$. Denote by $\tilde H$ the Schr\"odinger operator with potential $\tilde V = U(\omega' (x-a) + \omega a)$; since this is just a translation of $V_{\omega'}$, obviously $\spec(\tilde H) = \spec(H_{\omega'})$. However, for $x\in [a-L, a+L]$,
\[
\lvert U(\omega' (x-a) + \omega a) - U(\omega  x) \rvert  \le \lVert \nabla U \rVert_\infty \lvert \omega' - \omega \rvert \lvert x- a \rvert \le L \lVert \nabla U \rVert_\infty \lvert \omega' - \omega \rvert,
\]
and since $\phi(x) = 0$ for $x\notin [a-L, a+L]$, this implies that
\[
\lVert (\tilde H - H_\omega) \phi \rVert = \lVert U(\omega' (x-a) + \omega a) \phi(x) - U(\omega  x) \phi(x) \rVert \le  L \lVert \nabla U \rVert_\infty \lvert \omega' - \omega \rvert  \lVert \phi \rVert.
\]
Thus,
\begin{align*}
\lVert (\tilde H - E) \phi \rVert & \le \lVert (H_{\omega} - E)\phi \rVert +  \lVert (\tilde H - H_{\omega})\phi \rVert \\
& \le \left( \epsilon + \frac {C_1}L (1 +\lvert E\rvert^{1/2}) \right) \lVert \phi \rVert + L \lVert \nabla U \rVert_\infty \lvert \omega' - \omega \rvert \lVert \phi \rVert
\end{align*}
so
\[
\dist(E, \sigma(H_{\omega'})) \le \epsilon + \frac {C_1}L (1 +\lvert E\rvert^{1/2})  + L \lVert \nabla U \rVert_\infty \lvert \omega' - \omega \rvert.
\]
We optimize this bound in $L$ by choosing
\begin{equation}\label{ml09}
L = \sqrt{ \frac{   {C_1}(1 +\lvert E\rvert^{1/2})  }{  \lVert \nabla U \rVert_\infty \lvert \omega' - \omega \rvert  }} ,
\end{equation}
which gives
\[
\dist(E, \sigma(H_{\omega'})) \le \epsilon + 2\sqrt{ C_1 \lVert \nabla U \rVert_\infty} (1 +\lvert E\rvert^{1/4}) \lvert \omega' - \omega \rvert^{1/2}.
\]
Since $\epsilon\in (0,1)$ was arbitrary, this completes the proof, with $C = 2\sqrt{ C_1 \lVert \nabla U \rVert_\infty}$. Note that the choice of $L$ in \eqref{ml09} was allowed because for $\lvert \omega - \omega' \rvert < \tilde C = C_1 / \lVert \nabla U \rVert_\infty$,
\[
L = \sqrt{ \frac{  {C_1}(1 +\lvert E\rvert^{1/2}) }{   \lVert \nabla U \rVert_\infty \lvert \omega' - \omega \rvert  }} > \sqrt{1+ \lvert E \rvert^{1/2}} > 1. \qedhere
\]
\end{proof}

\section{Convergence of the Translation Flows on the Isospectral Tori of Periodic Approximations}\label{sec:9}

In this section we discuss the differential equation for the translation dynamics $V \to V(\cdot + t)$ on the torus of isospectral periodic potentials derived by Dubrovin \cite{Du} and Trubowitz \cite{Tr} in the context of Theorem~$\tilde I$ and the rational approximation of a given Diophantine vector $\omega$.

Recall some definitions from Section~\ref{sec.4}: $G_{m} = (E_{m}^-, E_m^+)$, $m \in \mathfrak{Z}(\tilde \omega)$, $\mathcal{O} = \mathcal{O}(\tilde\omega)=\{ m : E_{m}^- < E_m^+ \}$,
\begin{equation}\label{eq:8toruspm-1}
\begin{split}
G_{m,\sigma} & = \{ (\lambda,\sigma) : \lambda \in G_m \}, \quad \sigma = \pm, \\
\mathcal{C}_{m} & = G_{m,-} \cup G_{m,+} \cup \{ E_m^-, E_m^+ \}, \\
\mathcal{C} & = \prod_{m \in \mathcal{O}} \mathcal{C}_m.
\end{split}
\end{equation}
Consider the vector-field
$\Phi(\mu) = (\Phi_n(\mu))_{n \in \mathcal{O}}$ on $\mathcal{C}$,
\begin{equation}\label{ml06}
\Phi_n(\mu) := \sigma(\mu_n) \sqrt{4 (\underline E - \mu_n)(E_n^- - \mu_n)(E_n^+ - \mu_n) \prod_{\substack{i \in \mathcal{O} \\ i \neq n}} \frac{(E_i^- - \mu_n)(E_i^+ - \mu_n)}{(\mu_i - \mu_n)^2} }.
\end{equation}
It was shown in the work of Trubowitz \cite{Tr} that $\mathcal{C}$ has a natural structure of a compact smooth manifold and the $\Phi$-flow $\mu(t)$ is well-defined.

Given $\tau > 0$, consider the set $\mathcal{O}_\tau = \{m : |G_m| \ge \tau \}$, the ``$\tau$-cutoff'' torus $\mathcal{T}_\tau = \prod_{m\in \mathcal{O}_\tau} \mathcal{C}_m$, and the vector-field
\begin{equation}\label{eq:9phicutoff}
\Phi_{\tau, n}(\mu) =  \sqrt{4 (\underline E - \mu_n)(E_n^- - \mu_n)(E_n^+ - \mu_n) \prod_{\substack{i \in \mathcal{O}_\tau \\ i \neq n}} \frac{(E_i^- - \mu_n)(E_i^+ - \mu_n)}{(\mu_i - \mu_n)^2} }, \quad n \in \mathcal{O}_\tau.
\end{equation}

\begin{remark}\label{rem:9closeODE}
$(1)$ The arguments from  \cite{Tr} apply to $\Phi_{\tau,n}(\mu)$. Our first goal is to compare the trajectories of the fields \eqref{eq:9phicutoff} and \eqref{ml06}. We need to compare the trajectories on a fixed time interval $[0,T]$, while $\tau$ can be chosen arbitrarily small.

$(2)$ We reduce the problem to finite dimensional manifolds and we quantify the stability. It seems that the most efficient way to carry out the estimation of the problem raised in $(1)$ goes via the following general setting. Let $\mathcal{X}$, $\mathcal{Y}$ be smooth manifolds, $\mathcal{M} = \mathcal{X} \times \mathcal{Y}$. We assume that $\mathcal{X}$ and $\mathcal{Y}$ are equipped with a metric $d_\mathcal{X}$ and $d_\mathcal{Y}$, respectively.  Define a metric on $\mathcal{M}$ by setting $d((x,y),(u,v)) = d_\mathcal{X}(x,u) + d_\mathcal{Y}(y,v)$. Let $\Psi(x)$ be a smooth vector-field on $\mathcal{M}$. Let us view the tangent space to $\mathcal{M}$ at $(x,y)$ as the product of the tangent space to $\mathcal{X}$ at $x$ and the tangent space to $\mathcal{Y}$ at $y$. Denote by $\Theta$ and $\Upsilon$ the respective components of $\Psi$. Assume that the following conditions hold:

$(i)$ There exists a smooth vector-field $\Xi(x)$ on $\mathcal{X}$ such that $|\Xi(x) - \Theta(x,y)| < \delta$ for any $(x,y) \in \mathcal{M}$.

$(ii)$ $|\Upsilon(x,y)| < \delta$ for any $(x,y) \in \mathcal{M}$. Here and everywhere else $|\cdot|$ stands for the $\ell^1$-norm in the tangent space, compare part $(3)$ of this remark.

Given $(x_0,y_0) \in \mathcal{M}$, let $g(x_0,y_0;t) := (u(x_0, y_0; t), v(x_0, y_0; t))$ be the $\Psi$-trajectory originating at $(x_0,y_0)$. Let $h(x_0,; t)$ be the $\Xi$-trajectory originating at $x_0$. We need to estimate $d_\mathcal{X}(u(x_0, y_0; t), h(x_0; t))$ for any $(x_0, y_0) \in \mathcal{M}$ and any $0 \le t \le T$. Of course some natural relation between the metric and the smooth structure on the manifolds is needed. Moreover the smooth structure needs to be naturally quantified. We discuss these issues below in this remark.

$(3)$ Before we proceed we want to emphasize that by convention we always evaluate the vectors via the $\ell^1$ norm, that is,
$$
|(x_1, x_2, \ldots)| = \sum _k |x_k|.
$$
We mention this here again since the quantitative stability estimates will be in terms of this norm, \textbf{as the dimension $m$ varies in what follows in the current section}. For this reason we introduce boxes $\mathcal{P} = \prod_{1 \le k \le m} (- \rho_k,\rho_k)$ with rapidly decaying $\rho_k$ as a ``standard space unit'' for the local charts, see $(4)$ below.

$(4)$ To derive the estimates for $d(u(x_0, y_0; t), h(x_0; t))$ we need to localize the problem, that is, we need to work within a local chart on the manifold. For that matter the following definition is needed. Let $\mathcal{L}$ be a compact smooth $m$-dimensional manifold. Fix a sequence $\rho_k > 0$, $\sum_k \rho_k \le 1$. Let $\mathcal{P} = \prod_{1 \le k \le m} (- \rho_k, \rho_k)$, $\mathcal{P}_\gamma = \prod_{1 \le i \le m} (-\gamma \rho_k, \gamma \rho_k)$. We say that $\mathcal{L}$ belongs to the class $\mathfrak{R}(K)$ if the following conditions hold:

$(a)$ There exist homeomorphic injections $f_j$, $j \in \mathcal{J}$ from $\mathcal{P}$ into $\mathcal{L}$ such that the sets $\mathcal{U}_j = f_j(\mathcal{Q})$, $\mathcal{Q} = \mathcal{P}_{1/2K}$ cover the manifold.

$(b)$ Let $\mathcal{D} = \mathcal{P}_{1/K}$ and $\mathcal{W}_j = f_j(\mathcal{D})$. If $\mathcal{W}_j \cap \mathcal{W}_k \neq \emptyset$, then the map $\varphi_{j,k} = f_k^{-1} \circ f_j$ is well-defined on $\mathcal{D}$, $C^1$-smooth at each point where it is defined and obeys $\sup |\partial^\alpha \varphi_{j,k}| \le K$ for any $|\alpha| = 1$.

In this case we say also that $\mathfrak{F} = \{ f_j : j \in \mathcal{J} \}$ is a $K$-atlas for $\mathcal{L}$. We also set $\mathcal{V}_j = f_j(\mathcal{P})$.

Let $d$ be a metric on $\mathcal{L}$. We say that $d$ is $A$-regular with respect to the atlas $\mathfrak{F}$ if for any $p, q \in \mathcal{P}$ and any $j$, we have $d(f_j(p), f_j(q)) \le A |p - q|$.

We want to mention that the particular choice of $\rho_k$ does not affect the estimates we develop here. For our applications we consider
\begin{equation} \label{eq:8rhokdefi}
\rho_k = k^{-2}.
\end{equation}

$(5)$ We also need the following definitions. Using the above notation, let $x(t)$, $0 \le t \le T$ be an arbitrary smooth curve in $\mathcal{L}$. There exists $1 \le j_0 \le N$ such that $x(0) \in \mathcal{U}_{j_0}$. If $x(t)$ does not leave $\mathcal{W}_{j_0}$, then we set $t_1 = T$. Otherwise, let $t_1$ be the first moment $x(t)$ leaves $\mathcal{W}_j$. There exists $j_1$ such that $x(t_1) \in \mathcal{U}_{j_1}$. If $x(t)$ does not leave $\mathcal{W}_{j_1}$, then we set $t_2 = T$. Otherwise, let $t_2$ be the first moment $x(t)$ leaves $\mathcal{W}_{j_1}$, etc. We call the collection $\mathfrak{I} = \{ j_0, \ldots; t_1, \ldots \}$ an $\mathfrak{F}$-itinerary for the curve $x(t)$. The $\mathfrak{F}$-itinerary is not defined uniquely, but this does not play any role in what follows. We call the intervals $(t_k, t_{k+1})$ the itinerary intervals. We denote by $n(\mathfrak{I})$ the total number of itinerary intervals.

We say that a given curve $x(t)$ is $B$-regular with respect to the atlas $\mathfrak{F}$ if the following condition holds. Let $\mathfrak{I} = \{ j_0, \ldots; t_1, \ldots \}$ be an $\mathfrak{F}$-itinerary for $x(t)$. Take an arbitrary itinerary interval $(t_k, t_{k+1})$. The function $t \mapsto f_{j_k}^{-1} (x(t)) \in \mathcal{P}$, which is well-defined on this interval, obeys $| \partial_t f_{j_k}^{-1}(x(t))| \le B$.

Let $F(x)$ be a vector-field on $\mathcal{L}$. Let $\mathfrak{F} = \{f_j : 1 \le j \le N \}$ be a $K$-atlas for $\mathcal{L}$. For each $j$, denote by $F \circ f^{-1}_j$ the pull-back of the vector-field $F$ to the box $\mathcal{P}$ under the map $f_j$. Let $(F \circ f^{-1}_j)_k$, $1 \le k \le m$ be the components of $F \circ f^{-1}_j$. We say that $F(x)$ is $B$-regular with respect to the atlas $\mathfrak{F}$ if for any $j, k$ and $|\alpha| \le 1$, we have $\sup |\partial^\alpha  (F \circ f^{-1}_j)_k| \le B \rho_k$.

$(6)$ The definitions in this remark and the statements following below apply to compact infinite-dimensional manifolds, that is, in the case $m = \infty$. The discussion of this case goes in the same way as for $m < \infty$. The label set $\mathcal{J}$ for the atlas $\mathfrak{F}$ in this case may be uncountable.
\end{remark}

The statements in the next two lemmas follow straight from the definitions.

\begin{lemma}\label{lem9:localchartsconnect1}
Let $\mathfrak{F} = \{ f_j : 1 \le j \le N \}$ be a $K$-atlas for $\mathcal{L}$. Let $x \in \mathcal{W}_{j} \cap \mathcal{U}_{k}$, $y \in \mathcal{V}_{j}$, and $|f^{-1}_{j}(x) - f^{-1}_{j}(y)| \le \delta < 1/2 K$. Then $y \in \mathcal{V}_k$ and $|f^{-1}_{k}(x) - f^{-1}_{k}(y)| \le K \delta$.
\end{lemma}

\begin{lemma}\label{lem9:regularcurves1}
$(1)$ Assume that $x(t)$ is $B$-regular with respect to the atlas $\mathfrak{F}$. Then each itinerary interval $(t_k, t_{k+1})$ obeys $|t_{k+1} - t_k| \ge (2B)^{-1}$. In particular, $n(\mathfrak{I}) \le 2 BT$.

$(2)$ Assume that the vector-field $F(x)$ is $B$-regular with respect to the atlas $\mathfrak{F}$. Then each $F$-trajectory $x(t)$ is $B$-regular with respect to the atlas $\mathfrak{F}$.
\end{lemma}

Before we proceed we need to recall the followings standard facts about contracting maps and ODE's.

$(i)$ Let $(\mathfrak{X}, \rho)$ be a metric space. Let $S : \mathfrak{X} \to \mathfrak{X}$ be a contracting map, that is, with some $0 < \lambda < 1$, we have for every $\mathfrak{x}, \mathfrak{y}$, $\rho(S(\mathfrak{x}), S(\mathfrak{y})) \le \lambda \rho(\mathfrak{x}, \mathfrak{y})$. We say that $S$ is $\lambda$-contracting. In this case $S$ has a unique fixed point $\mathfrak{x}_0$. Moreover, for any $\mathfrak{x}$ and $n$, we have
$$
\rho(S^n(\mathfrak{x}), \mathfrak{x}_0) \le \rho(\mathfrak{x}, S(\mathfrak{x})) (1-\lambda)^{-1} \lambda^n.
$$

$(ii)$ Using the notation from $(i)$ assume that $S_1, S_2 : \mathfrak{X} \to \mathfrak{X}$ are $\lambda$-contracting. Let $\delta = \sup_\mathfrak{x} \rho(S_1(\mathfrak{x}), S_2(\mathfrak{x}))$. Then,
$$
\sup_{\mathfrak{x}, n} \rho(S_1^n(\mathfrak{x}), S^n_2(\mathfrak{x})) \le (1 - \lambda)^{-1} \delta.
$$
Let $\mathfrak{x}^{(j)}_0$ be the fixed point of $S_j$, $j = 1, 2$. Then,
$$
\rho(\mathfrak{x}^{(1)}_0, \mathfrak{x}^{(2)}_0) \le (1 - \lambda)^{-1} \delta.
$$

$(iii)$ Let $H(x)$ be a vector-field on the box $\mathcal{P}$. Assume that $H$ is $B$-regular, that is, $\sup |\partial^\alpha H| \le B$ for any $|\alpha| = 1$.  Given $t_0 > 0$, let $\mathfrak{X}_{t_0}$ be the set of all continuous trajectories $\mathfrak{x} = (x(t))_{0 \le t \le t_0}$, $x(t) \in \mathcal{P}$. Set $\rho(\mathfrak{x}, \mathfrak{y}) = \sup_t |x(t) - y(t)|$, $\mathfrak{x}, \mathfrak{y} \in \mathfrak{X}_{t_0}$. Assume $\lambda := B t_0 \le 1/2$. Given $x^\zero \in \mathcal{P}_{1/2}$, set
\begin{equation}\label{eq:9intoper}
S_{x^\zero, H} [\mathfrak{x}](t) = x^\zero + \int_0^t H(x(s)) \, ds.
\end{equation}
This defines a map of $\mathfrak{X}_{t_0}$, which is $\lambda$-contracting. The fixed point of this map, $(x^\zero(t))_{0 \le t \le t_0}$, is the $H$-trajectory originating at $x^\zero$, that is, the unique solution of the ODE
\begin{equation}\label{eq:9ODEabst}
\dot{x}^\zero(t) = H(x^\zero(t)), \quad 0 \le t \le t_0, \quad x^\zero(0) = x^\zero.
\end{equation}
Furthermore, let $x^\zero, x^\one \in \mathcal{P}_{1/2}$. Let $x^\zero(t), x^\one(t)$ be the corresponding $H$-trajectories. Then,
\begin{equation}\label{eq:9ODEstab1}
\sup_t |x^\zero(t) - x^\one(t)| \le (1 - \lambda)^{-1} |x^\zero - x^\one|.
\end{equation}

$(iv)$ Using the notation of $(iii)$, let $H^{(j)}(x)$, $j = 1, 2$ be $B$-regular vector-fields on the box $\mathcal{P}$. Let $x^\zero \in \mathcal{P}_{1/2}$ and let $x^{(k)}(t)$ be the $H^{(k)}$-trajectory originating at $x^\zero$. Then,
\begin{equation}\label{eq:9ODEstab2}
\sup_t |x^\zero(t) - x^\one(t)| \le (1 - \lambda)^{-1} \sup_x |H^{(1)}(x) - H^{(2)}(x)|.
\end{equation}

$(v)$ Using the notation from Remark~\ref{rem:9closeODE}, let $H(x)$ be a smooth vector-field on the manifold $\mathcal{L}$. Then due to compactness, the $H$-trajectories are well defined for all $t$.

\bigskip

Now we combine these standard facts with the definitions from Remark~\ref{rem:9closeODE}.

\begin{lemma}\label{lem9:trajestability2}
Using the notation from Remark~\ref{rem:9closeODE}, let $\mathfrak{F} = \{ f_j : 1 \le j \le N \}$ be a $K$-atlas for $\mathcal{L}$. Let $H^{(k)}(x)$, $k = 1, 2$ be vector-fields on the manifold $\mathcal{L}$. Assume that $H^{(k)}$ is $B$-regular with respect to the atlas $\mathfrak{F}$. Set
\begin{equation}\label{eq:9fieldsdeviations}
\delta = \sup_{x,j} |H^{(1)} \circ f^{-1}_j - H^{(2)} \circ f^{-1}_j|.
\end{equation}
Let $x^\zero \in \mathcal{L}$ and denote by $x^{(k)}(t)$ the $H^{(k)}$-trajectory originating at $x^\zero$. Let $T$ be arbitrary and let $\mathfrak{I} = \{ j_0, \ldots; t_1, \ldots \}$ be an $\mathfrak{F}$-itinerary for the curve $x^\one(t)$ on the interval $[0,T]$. Then for any $t_r \le t \le t_{r+1}$, we have $x^{(2)}(t) \in \mathcal{V}_{j_r}$ and $|f^{-1}_{j_r}(x^{(2)}(t)) - f^{-1}_{j_r}(x^{(1)}(t))| \le \delta_r := (4K)^r \delta$. In particular, let $d$ be a metric on $\mathcal{L}$ that is $A$-regular with respect to the atlas $\mathfrak{F}$. Then,
\begin{equation}\label{eq:9trajecmetricdeviation}
d(x^{(1)}(t), x^{(2)}(t)) \le A (4K)^{BT} \delta.
\end{equation}
\end{lemma}

\begin{proof}
The proof of the first statement goes by induction over $r$, starting from $r = 0$. For $r = 0$, we consider the pull-back fields
$H^{(1)} \circ f^{-1}_{j_0}$, $k = 1,2$ and apply \eqref{eq:9ODEstab2} from $(iv)$. This yields $|f^{-1}_{j_0}(x^{(2)}(t)) - f^{-1}_{j_0}(x^{(1)}(t))| \le 2 \delta$ for any $t_0 \le t \le t_{1}$, that is, the statement holds for $r = 0$. Assume that the statement holds for $r'$ in the role of $r$ for $r' \le r-1$, $r \ge 1$. In particular, we have $|f^{-1}_{j_{r-1}}(x^{(2)}(t_r)) - f^{-1}_{j_{r-1}}(x^{(1)}(t_r))| \le \delta_{r-1}$. Recall that $x^{(1)}(t_r) \in \mathcal{U}_{j_r}$. Due to Lemma~\ref{lem9:localchartsconnect1}, $x^{(2)}(t_r) \in \mathcal{U}_{j_r}$ and $|f^{-1}_{j_{r}}(x^{(2)}(t_r)) - f^{-1}_{j_{r}}(x^{(1)}(t_r))| \le K \delta_{r-1}$. Now we consider the pull-back fields $H^{(1)} \circ f^{-1}_{j_r}$, $k = 1, 2$ and apply \eqref{eq:9ODEstab1} and \eqref{eq:9ODEstab2} from $(iv)$. This yields
$$
|f^{-1}_{j_r}(x^{(2)}(t)) - f^{-1}_{j_r}(x^{(1)}(t))| \le 2 |f^{-1}_{j_{r}}(x^{(2)}(t_r)) - f^{-1}_{j_{r}}(x^{(1)}(t_r))| + 2 \delta < \delta_r
$$
for any $t_r \le t \le t_{r+1}$, as claimed. For the second statement let us recall that due to Lemma~\ref{lem9:regularcurves1}, $n(\mathfrak{I}) \le BT$. Therefore the statement follows just from the definition of an $A$-regular metric.
\end{proof}

It is easy to see that the last lemma applies to the problem raised in Remark~\ref{rem:9closeODE}. Let $\mathcal{X}$, $\mathcal{Y}$ be smooth manifolds, $\mathcal{M} = \mathcal{X} \times \mathcal{Y}$. We assume that $\mathcal{M} \in \mathfrak{R}(K)$. Let $\mathfrak{F} = \{ f_j : 1 \le j \le N \}$ be a $K$-atlas for $\mathcal{M}$. Let $\Psi(x)$ be a smooth vector-field on $\mathcal{M}$ and let $\Theta$ and $\Upsilon$ be the respective components of $\Psi$. Assume: (i) there exists a smooth vector-field $\Xi(x)$ on $\mathcal{X}$ such that $|\Xi(x) - \Theta(x,y)| < \delta$ for any $(x,y) \in \mathcal{M}$ and $(ii)$ $|\Upsilon(x,y)| < \delta$ for any $(x,y) \in \mathcal{M}$. View $\Omega(x,y) := (\Xi(x),0)$ as a vector-field on $\mathcal{M}$. Assume that both $\Xi, \Omega$ are $B$-regular with respect to the atlas $\mathfrak{F}$. Given $(x_0, y_0) \in \mathcal{M}$, let $g(x_0, y_0; t) := (u(x_0, y_0; t), v(x_0, y_0; t))$ be the $\Psi$-trajectory originating at $(x_0, y_0)$. Let $h(x_0; t)$ be the $\Xi$-trajectory originated at $x_0$. Clearly, $q(x_0, y_0, t)=(h(x_0; t), y_0)$ is the $\Omega$-trajectory originating at $(x_0,y_0)$. We assume that $\mathcal{X}$ and $\mathcal{Y}$ are both equipped with a metric $d_\mathcal{X}$ and $d_\mathcal{Y}$, respectively. We define a metric on $\mathcal{M}$ by setting $d((x,y),(u,v)) = d_\mathcal{X}(x,u) + d_\mathcal{Y}(y,v)$. Finally, we assume that $d$ is $A$-regular with respect to the atlas $\mathfrak{F}$. Applying the last lemma, we obtain

\begin{corollary}\label{Lm9:voltera1}
With the above notation, we have for any $(x_0, y_0) \in \mathcal{M}$, any $T$, and any $0 \le t \le T$,
\begin{equation}\label{eq:9trajecmetricdeviation2-1}
d(g(x_0, y_0; t), q(x_0, y_0, t)) \le A (4K)^{BT} \delta.
\end{equation}
\end{corollary}

\begin{remark}\label{rem:9closeODE1}
$(1)$ To resolve the problem raised in part $(1)$ of Remark~\ref{rem:9closeODE} we would need to verify that the last corollary applies to the vector-fields $\Phi$, $\Phi_{\tau}$. Clearly, for that we would need to develop some estimates for the vector-fields in question. First of all let us address the issue of the boxes $\mathcal{P} = \prod_{1 \le k \le m} (-\rho_k, \rho_k)$ as a ``standard space unit.'' We need to re-enumerate $\mathcal{O}$ via integers $k = 1, 2, \ldots$. Obviously, there are many ways to do this. The only thing that is important here is that the estimates \eqref{eq:8Gnesti} below hold. The following enumeration is convenient enough for us. Recall, once again, that due to Theorem~$\tilde B$,
\begin{equation}\label{eq:8GnestitilB}
|G_\mathfrak{m}| \le 2 \ve \exp \Big( -\frac{\kappa_0}{2} |\mathfrak{m}| \Big).
\end{equation}
Let
$$
L_k = \left\{ \mathfrak{m} : 2 \ve \exp \Big( -\frac{\kappa_0}{2}(k-1) \Big) \le |G_\mathfrak{m}| < 2 \ve \exp \Big( -\frac{\kappa_0}{2}k \Big) \right\}.
$$
We enumerate the points of $L_k$ arbitrarily. We enumerate $\mathcal{O}$ lexicographically with respect to $L_k$ with increasing $k$, starting from $k = 1$. One can see that with the enumeration $n \mapsto \mathfrak{m}_n$ obtained in this way, we have
\begin{equation}\label{eq:8Gnesti}
\begin{split}
c(a_0, b_0, \kappa_0, \nu) |\mathfrak{m}_n|^{\beta(a_0, b_0, \kappa_0, \nu)} & \le |n| \le C(a_0, b_0, \kappa_0, \nu) |\mathfrak{m}_n|^{b(a_0, b_0, \kappa_0, \nu)} , \\
c(a_0, b_0, \kappa_0, \nu) \ve_0\exp(-n^{\sigma_1(a_0, b_0, \kappa_0, \nu)}) & \le |G_n| \le C(a_0, b_0, \kappa_0, \nu) \ve_0 \exp(-n^{\sigma(a_0, b_0, \kappa_0, \nu)}).
\end{split}
\end{equation}
The specific way we do the enumeration does not affect the results in what follows, as long as \eqref{eq:8Gnesti} holds.

$(2)$ The vector-fields $\Phi_n$ are not smooth with respect to the variables $\mu$ at $\mu_n = E^\pm_n$. To resolve this problem we follow \cite{Tr} and introduce below auxiliary manifolds and vector-fields. Namely, set
\begin{equation}\label{eq:8auxmanofold}
\begin{split}
I_n & = (0, \pi \xi_n/2), \quad \xi_n = \exp(-n^{\sigma(a_0, b_0, \kappa_0, \nu)}/2), \quad I_{n, \pm} = \{ (\lambda, \pm) : \lambda \in I_n \}, \\
S_n & = I_{n,+} \cup I_{n,-} \cup \{ 0, \pi \xi_n/2 \}, \quad \mathcal{M} = \prod_{n} S_{n}, \quad \mathcal{X} = \prod_{n \le N} S_{n}, \quad \mathcal{Y} = \prod_{n > N} S_{n}, \\
\mu_n(\theta) & = E^-_n+ (E^+_n - E^-_n) \sin^2 \Big( \frac{\theta_n}{\xi_n} \Big), \quad \theta = (\theta_n, \sigma_n)_n \in \mathcal{M}, \\
\Psi_n(\theta) & =  4 \xi_n \sqrt{(\mu_n(\theta) - \underline{E}) \prod_{\substack{i \\ i \neq n}} \frac{(E_i^- - \mu_n(\theta))(E_i^+ - \mu_n(\theta))}{(\mu_i(\theta) - \mu_n(\theta))^2}}, \quad \Psi = (\Psi_n)_n.
\end{split}
\end{equation}
The field $\Psi$ is the pull-back of $\Phi$ under the map $\theta \mapsto \mu(\theta)$. We verify this later in Lemma~\ref{lem:8Lml2ODEmutheta} after we have all the results on $\Psi$ in place.

$(3)$ We identify $S_n$ with a circle $C_n$ of diameter $\rho_n$ via the bijection $(\theta,\sigma) \mapsto \rho_n \exp(4 i \sigma \theta \rho_n^{-1})$. Clearly, the standard smooth structure on $C_n$ has a $K$-atlas, with $K$ being an absolute constant. This defines a $K$ atlas on $S_n$. We consider the product structure
on $\mathcal{X}$ and denote by $\mathfrak{F}$ a $K$ atlas for it. Set
\begin{equation}\label{eq:9ddistance}
d((\theta,\sigma), (\theta',\sigma')) = \rho_n |\exp(4 i \sigma \theta \rho_n^{-1}) - \exp(4 i \sigma' \theta' \rho_n^{-1})|
\end{equation}
and
\begin{equation}\label{eq:9ddistanceX}
d(((\theta_n, \sigma_n))_n, ((\theta'_n, \sigma'_n))) = \sum_{n} d((\theta_n, \sigma_n), (\theta'_n, \sigma'_n)).
\end{equation}
The metric space $(\mathcal{M}, d)$ is compact. The metric $d$ is $A$-regular with respect to the atlas $\mathfrak{F}$, with $A$ being an absolute constant. The same applies to the manifolds $\mathcal{X}$ and $\mathcal{Y}$.

$(4)$ Note that $\xi_k$ are set up so that
\begin{equation} \label{eq:8rhokdefi1}
\xi_k^{-1} |G_k| \le C \xi_k,
\end{equation}
see \eqref{eq:8Gnesti}.
\end{remark}

We turn now to the analysis of the vector-fields in question. We use the enumeration mentioned in the last remark, so that  $\Phi(\mu) = (\Phi_{n}(\mu))_{n \ge 1}$. Lemma~\ref{lem2:homogenuityplus} is instrumental here. Due to this lemma and the first estimate in \eqref{eq:8Gnesti}, we have the following:

\smallskip
$(i)$ For any $m' \neq m$ with $m' \ge m$, we have
\begin{equation}\label{eq:6noresonanceI}
\dist ([E_{m}^-, E_{m}^+], [E_{m'}^-, E_{m'}^+]) \ge a (m')^{-b}
\end{equation}
with constants $a,b$ depending on $a_0, b_0, \kappa_0, \nu$.

$(ii)$ Denote by $\mathfrak{R}_{m}$ the set of all $m' \neq m$ with $m' \leq m$ such that
\begin{equation}\label{eq:6resonance}
(E_{m'}^+ - E_{m'}^-) \ge (\dist ([E_m^-, E_m^+], [E_{m'}^-, E_{m'}^+])^4.
\end{equation}
The set $\mathfrak{R}_m$ can be enumerated as follows: $n^{(r)}_m$, $r = 1, \ldots, r_0$,
$$
n^{(r-1)}_m \ge \tau_1 \exp((n^{(r)}_m)^{\gamma_1}), \quad r = 1, \ldots, r_0,
$$
where $n^{(0)}_m := m$, $\tau_1, \gamma_1$ depend on $a_0, b_0, \kappa_0, \nu$.

$(iii)$ For any $n$, we have
\begin{equation}\label{eq:6resonancelowerE}
a |n|^{-b} \le E_n^- - \underline{E} \le 2 C |\tilde \omega|^2 (1+n)^2.
\end{equation}

Consider the following functions, see \eqref{ml06},
\begin{equation}\label{ml07F}
\begin{split}
P_n(\mu) & = \prod_{\substack{i \in \mathfrak{R}_n}} \frac{(E_i^-  - \mu_n)(E_i^+  - \mu_n)}{(\mu_i -  \mu_n)^2}, \\
Q_n(\mu) & = \prod_{\substack{i \in \mathfrak{N}_n}} \frac{(E_i^-  - \mu_n)(E_i^+  - \mu_n)}{(\mu_i -  \mu_n)^2},
\end{split}
\end{equation}
where $\mathfrak{R}_n$ is defined as in $(ii)$ above, $\mathfrak{N}_n = \{ i \neq n \} \setminus \mathfrak{R}_n$. Introduce for convenience
\begin{align*}
\gamma_i & = E_{\mathfrak{i}}^+ - E_{\mathfrak{i}}^-, \\
\eta_{i,n} & = \dist ([E_\mathfrak{i}^-, E_\mathfrak{i}^+], [E_{\mathfrak{n}}^-, E_{\mathfrak{n}}^+]).
\end{align*}
We assume that $\ve_0$ is small enough so that \eqref{eq:8Gnesti} reads
\begin{equation}\label{eq:8Gnestimod}
\gamma_i \le \ve_1 \exp(-i^\sigma)
\end{equation}
with $\ve_1 \ll 1$. Due to the definition of the set $\mathfrak{N}_n$, we have
\begin{equation}\label{eq:6resonanceAA-1}
\gamma_i < \eta_{i,n}^4, \quad i \in \mathfrak{N}_n.
\end{equation}
On the other hand, due to \eqref{eq:6noresonanceI} we have
\begin{equation}\label{eq:6resonanceAA-2}
\gamma_n < \eta_{i,n}^4, \quad i \in \mathfrak{R}_n.
\end{equation}
since here $i < n$.

\begin{lemma}\label{Lml2} $(1)$
\begin{equation}\label{ml07}
Q_n(\mu) = \prod_{\substack{i \in \mathfrak{N}_n}} (1 + \phi_{i,n} (\mu_i,\mu_n)),
\end{equation}
where $\phi_{i,n}(\mu_i, \mu_n)$ obey
\begin{equation}\label{ml10}
|\phi_{i,n} (\mu_i, \mu_n)| \le 6 \gamma_i^{3/4}.
\end{equation}

$(2)$ The function $\sqrt {Q_n(\mu)}$ is well-defined, smooth and obeys
\begin{equation}\label{eq8:ml10}
|\partial_{\mu_k}\sqrt {Q_n(\mu)}| \le C(a_0, b_0, \kappa_0, \nu).
\end{equation}
\end{lemma}

\begin{proof}
We have
\begin{equation}\label{ml071}
\frac{(E_i^- - \mu_n)(E_i^+ - \mu_n)}{(\mu_i - \mu_n)^2} = 1 + \frac{E_i^- - \mu_i}{\mu_i - \mu_n} + \frac{E_i^+ - \mu_i}{\mu_i - \mu_n} + \frac{(E_i^- - \mu_i)(E_i^+ - \mu_i)}{(\mu_i - \mu_n)^2} =: 1 + \phi_{i,n}(\mu).
\end{equation}

Thus, for $i \in \mathfrak{N}_n$ and $(\mu_i,\mu_n)$ in the mentioned domain, we have
\begin{equation}\label{ml072}
\begin{split}
|\phi_{i,n}(\mu)| & \le \frac{E_i^+ - E_i^-}{\eta_{i,n} - 2 \gamma_i} + \frac{E_i^+ - E_i^-}{\eta_{i,n} - 2 \gamma_i} + \frac{(E_i^+ - E_i^- )^2}{(\eta_{i,n} - 2 \gamma_i)^2} \\
& \le 4 (E_i^+ - E_i^-)^{3/4} + 2 (E_i^+  -E_i^-)^{3/2} \\
& \le 6 (E_i^+ - E_i^-)^{3/4}.
\end{split}
\end{equation}
This finishes $(1)$. Differentiating, one obtains
\begin{equation}\label{ml07123}
\begin{split}
|\partial_{\mu_k} \sqrt {Q_n(\mu)}| = \frac{1}{2} |Q_n(\mu)|^{-1/2} |\partial_{\mu_k} Q_n(\mu)| \le \frac{1}{2} \prod_{\substack{i \in \mathfrak{N}_n}} |1 - 6 \gamma_i^{5/6}|^{-1/2} |\partial_{\mu_k} Q_n(\mu)| \le C_1(\kappa_0, \nu) |\partial_{\mu_k} Q_n(\mu)|.
\end{split}
\end{equation}
Take $k = n$:
\begin{equation}\label{ml07124}
\begin{split}
|\partial_{\mu_n} Q_n(\mu)| & \le \sum_{\substack{j \in \mathfrak{N}_n}}\prod_{\substack{i \in \mathfrak{N}_n\setminus \{j\}}}|\frac{(E_i^-  - \mu_n)(E_i^+  - \mu_n)}{(\mu_i -  \mu_n)^2}| |\partial_{\mu_n}\frac{(E_j^- - \mu_n)(E_j^+ - \mu_n)}{(\mu_j - \mu_n)^2}| \\
& \le C_1(a_0, b_0, \kappa_0, \nu) \sum_{\substack{j \in \mathfrak{N}_n}} |\partial_{\mu_n} \frac{(E_j^- - \mu_n)(E_j^+ - \mu_n)}{(\mu_j - \mu_n)^2}|, \\
|\partial_{\mu_n}\frac{(E_j^- - \mu_n)(E_j^+ - \mu_n)}{(\mu_j - \mu_n)^2}| & \le |\partial_{\mu_n}[\frac{E_j^- - \mu_j}{\mu_i - \mu_n} + \frac{E_i^+ - \mu_i}{\mu_i - \mu_n} + \frac{(E_i^- - \mu_i)(E_j^+ - \mu_j)}{(\mu_j - \mu_n)^2}]| \\
& \le \frac{E_j^+ - E_j^-}{(\eta_{j,n} - 2 \gamma_j)^2} + \frac{E_j^+ - E_j^-}{(\eta_{j,n} - 2 \gamma_j)^2} + 2 \frac{(E_j^+ - E_j^- )^2}{(\eta_{j,n} - 2 \gamma_j)^3} \\
& \le 8 (E_j^+ - E_j^-)^{1/2} + 2 (E_j^+ - E_j^-)^{5/4} \\
& \le 10 (E_j^+ - E_j^-)^{1/2} \\
& < 10 \ve_1^{1/2} \exp(-j^\sigma/2) \\
& < \exp(-j/2^\sigma), \\
|\partial_{\mu_n} Q_n(\mu)| & \le 10 C_1(a_0, b_0, \kappa_0, \nu) \sum_{\substack{j \in \mathfrak{N}_n}} \exp(-j^\sigma/2) \\
& = C(a_0,b_0,\kappa_0,\nu).
\end{split}
\end{equation}
For $k \neq n$, the proof is similar.
\end{proof}

\begin{lemma}\label{Lml2PQ}
$|\mathfrak{R}_m| \le \log_2 \log_2 m + D(\kappa_0, \nu)$.
\end{lemma}

\begin{proof}
Let $n^{(r)}_m$, $r = 1, \ldots, r_0$ be as in $(ii)$ above. Let $R_0 = R_0(a_0, b_0, \kappa_0, \nu)$ be such that
$$
\tau a_1 \exp(R^{\gamma_1}) \ge R^2, \quad \text{for $R \ge R_0$}.
$$
If $n^{(r)}_m \le R_0$ for all $r$, then $|\mathfrak{R}_m| \le R_0$ and we are done. Assume $n^{(r)}_m > R_0$ for $r = 0, \dots, r_1$. Then
$$
n^{(r-1)}_m \ge \tau \exp(\beta n^{(r)}_m) > (n^{(r)}_m)^{2},
$$
$r = 0, \dots, r_1$. Hence
\begin{equation}\nn
\begin{split}
m = n^{(0)}_m & \ge (n^{(r_1)}_m)^{2^{r_1}} > R_0^{2^{r_1}} > 2^{2^{r_1}}, \\
r_1 & \le \log_2 \log_2 m.
\end{split}
\end{equation}
\end{proof}

Before we proceed, recall \eqref{eq:6noresonanceI} once again:
\begin{equation}\label{eq:6noresonanceIREP}
\eta_{i,n}^{-1} \le C(a_0, b_0, \kappa_0, \nu) n^{b(a_0, b_0, \kappa_0, \nu)}, \quad i \le n.
\end{equation}

\begin{lemma}\label{Lml2PQAG}
The functions $P_n(\mu)$, $\sqrt {P_n(\mu)}$ are well-defined, smooth, and obey
\begin{equation}\label{eq8:ml10A-1}
\begin{split}
C(a_0, b_0, \kappa_0, \nu)^{-\log_2 \log_2 n} n^{-b(a_0, b_0, \kappa_0, \nu) \log_2 \log_2 n} \le |P_n(\mu)| \\
\le C(a_0, b_0, \kappa_0, \nu)^{\log_2 \log_2 n} n^{b(a_0, b_0, \kappa_0, \nu) \log_2 \log_2 n}, \\
|\partial_{\mu_k} \sqrt {P_n(\mu)}| \le C(a_0, b_0, \kappa_0, \nu)^{\log_2 \log_2 n} n^{b(a_0, b_0, \kappa_0, \nu) \log_2 \log_2 n}.
\end{split}
\end{equation}
\end{lemma}

\begin{proof}
Clearly, $(E_i^-  - \mu_n)(E_i^+  - \mu_n) > 0$ for any $i \neq n$ since the gaps are disjoint. Therefore $\sqrt {P_n(\mu)}$ is well-defined. Note that
\begin{equation}\label{eq6:fracestimates}
\begin{split}
C(a_0, b_0, \kappa_0, \nu)^{-1} n^{-b(a_0, b_0, \kappa_0, \nu)} & \le 3 \min (1, \eta_{i,n}) \\
& \le \frac{\eta_{i,n}}{\eta_{i,n} + \gamma_{i} + \gamma_{n}} \\
& \le |\frac{(E_i^-  - \mu_n)(E_i^+  - \mu_n)}{(\mu_i -  \mu_n)^2}| \\
& \le 1 + \frac{\gamma_i}{\eta_{i,n}} \\
& < C(a_0, b_0, \kappa_0, \nu) n^{b(a_0, b_0, \kappa_0, \nu)};
\end{split}
\end{equation}
see \eqref{eq:6noresonanceIREP}. This implies the estimates on $|P_n(\mu)|$ in \eqref{eq8:ml10A-1}. Differentiating, one obtains
\begin{equation}\label{ml0712345}
\begin{split}
|\partial_{\mu_k} \sqrt {P_n(\mu)}| & = \frac{1}{2} |P_n(\mu)|^{-1/2} |\partial_{\mu_k} P_n(\mu)| \\
& \le C(a_0, b_0, \kappa_0, \nu)^{\log_2 \log_2 n} (\log \gamma_n^{-1})^{b(a_0, b_0, \kappa_0, \nu) \log_2 \log_2 n} |\partial_{\mu_k} P_n(\mu)|.
\end{split}
\end{equation}
Take $k = n$:
\begin{equation}\label{ml0712445}
\begin{split}
|\partial_{\mu_n} P_n(\mu)| & \le \sum_{\substack{j \in \mathfrak{R}_n}} \prod_{\substack{i \in \mathfrak{R}_n \setminus \{ j \}}}| \frac{(E_i^-  - \mu_n)(E_i^+  - \mu_n)}{(\mu_i -  \mu_n)^2}| |\partial_{\mu_n} \frac{(E_j^-  - \mu_n)(E_j^+  - \mu_n)}{(\mu_j -  \mu_n)^2}| \\
& \le C(a_0, b_0, \kappa_0, \nu)^{\log_2 \log_2 n}(\log \gamma_n^{-1})^{b(a_0, b_0, \kappa_0, \nu) \log_2 \log_2 n} |\mathfrak{R}_n| \times \\
& \qquad \times \max_{\substack{j \in \mathfrak{R}_n}} |\partial_{\mu_n} \frac{(E_j^- - \mu_n)(E_j^+ - \mu_n)}{(\mu_j -  \mu_n)^2}|, \\
|\partial_{\mu_n} \frac{(E_j^- - \mu_n)(E_j^+ - \mu_n)}{(\mu_j - \mu_n)^2}| & \le \frac{|E_j^+ - \mu_n|}{(\mu_j - \mu_n)^2} + \frac{|E_j^- - \mu_n|}{(\mu_j - \mu_n)^2} + 2 \frac{|E_j^+ - \mu_n| |E_j^- - \mu_n|}{|\mu_j - \mu_n|^3} \\
& \le 4 (1 + \frac{2}{\eta_{i,n}^3}) \\
& < C_1(a_0, b_0, \kappa_0, \nu)(\log \gamma_n^{-1})^{b_1(a_0, b_0, \kappa_0, \nu)}.
\end{split}
\end{equation}
Combining \eqref{ml0712345} and \eqref{ml0712445}, we obtain the estimates on $|\partial_{\mu_n} P_n(\mu)|$ in \eqref{eq8:ml10A-1}. For $k \neq n$, the proof is similar.
\end{proof}

\begin{lemma}\label{Lml2PQAGlowerE}
The functions $\sqrt{\mu_n - \underline{E}}$ are well-defined, smooth, and obey
\begin{equation}\label{eq8:ml10A-2}
\begin{split}
a n^{-b} \le \sqrt{\mu_n - \underline{E}} & \le \mu_n \le C |\tilde \omega|^2 (1+n)^2, \\
|\partial_{\mu_k} \sqrt {P_n(\mu)}| & \le C n^b,
\end{split}
\end{equation}
where $a, b, C$ depend on $a_0, b_0, \kappa_0, \nu$.
\end{lemma}

\begin{proof}
The statement follows from \eqref{eq:6resonancelowerE}.
\end{proof}

\begin{corollary}\label{cor:8Psiregular}
The functions $\Psi_n(\theta)$ in \eqref{eq:8auxmanofold} are smooth and obey
\begin{equation}\label{eq8:Psismooth}
|\partial^\alpha \Psi_n(\theta)| \le C(a_0, b_0, \kappa_0, \nu) \xi_n^{1/2} \le B(a_0, b_0, \kappa_0, \nu) n^{-2}, \quad \alpha \le 1.
\end{equation}
\end{corollary}

\begin{proof}
It follows from the definition \eqref{eq:8auxmanofold} and Lemmas~\ref{Lml2}--\ref{Lml2PQAGlowerE} that
\begin{equation}\label{eq:8auxmanofoldiff}
\begin{split}
|\Psi_n(\theta)| & = 4 \rho_n \sqrt{(\mu_n(\theta) - \underline{E}) P_n(\mu(\theta)) Q_n(\mu(\theta))} \\
& \le C(a_0, b_0, \kappa_0, \nu)^{\log_2 \log_2 n} n^{b(a_0, b_0, \kappa_0, \nu) \log_2 \log_2 n} |\tilde \omega|^2 \xi_n, \\
|\partial_{\theta_n} \Psi_n(\theta)| & = |\partial_{\mu_n} \Psi_n(\theta) \partial_{\theta_n} \mu_n| \le
2 |\partial_{\mu_n} \Psi_n(\theta)| \gamma_n \rho_n^{-1} \\
& \le C(a_0, b_0, \kappa_0, \nu)^{\log_2 \log_2 n} n^{b(a_0, b_0, \kappa_0, \nu) \log_2 \log_2 n} |\tilde \omega|^2 \gamma_n, \\
|\partial_{\theta_k} \Psi_n(\theta)| & = |\partial_{\mu_k} \Psi_n(\theta) \partial_{\theta_k} \mu_k| \le 2 |\partial_{\mu_k} \Psi_n(\theta)| \gamma_k \xi_k^{-1} \\
& \le C(a_0, b_0, \kappa_0, \nu)^{\log_2 \log_2 n} n^{b(a_0, b_0, \kappa_0, \nu) \log_2 \log_2 n} |\tilde \omega|^2 \xi_n.
\end{split}
\end{equation}
The statement follows now from \eqref{eq:8rhokdefi1}, \eqref{eq:8Gnestimod} combined with the adjustment in \eqref{eq:8rhokdefi1}:
$\xi_k^{-1}\gamma_k\le C\xi_k$.
\end{proof}

This in its turn implies the following statement.

\begin{corollary}\label{cor:8PsiregularM}
The vector-field $\Psi = (\Psi_n)_n$ is $B$-regular with respect to the atlas $\mathfrak{F}$ with $B = B(a_0, b_0, \kappa_0, \nu) > 0$.
\end{corollary}

Consider the manifold $\mathcal{X}$ in \eqref{eq:8auxmanofold} and the vector-field on $\mathcal{X}$ defined via
\begin{equation}\label{eq:8auxmanofoldN}
\Psi_{N,n}(\theta) :=  4 \xi_n \sqrt{(\mu_n(\theta) - \underline{E}) \prod_{\substack{i \\ i \neq n, \quad i \le N}} \frac{(E_i^- - \mu_n(\theta))(E_i^+  - \mu_n(\theta))}{(\mu_i(\theta) - \mu_n(\theta))^2} }, \quad n \le N, \quad \Psi^{(N)} = (\Psi_{N,n})_n.
\end{equation}
Clearly, the same analysis applies to this case. So, we have the following statement:

\begin{corollary}\label{cor:8PsiregularX}
The vector-field $\Psi^{(N)} = (\Psi_{N,n})_n$ is $B$-regular with respect to the atlas $\mathfrak{F}_\mathcal{X}$ on $\mathcal{X}$, with $B = B(a_0, b_0, \kappa_0, \nu)$.
\end{corollary}

Finally, we have the following:

\begin{prop}\label{prop:8.stabil1}
Consider the vector-field $\Psi^{(N,O)} = (\Psi^{(N)},0)$ on $\mathcal{M}$.

$(1)$ The vector-field $\Psi^{(N,O)}$ is $B$-regular with respect to the atlas $\mathfrak{F}$ on $\mathcal{M}$, with $B = B(a_0, b_0, \kappa_0, \nu) > 0$.

$(2)$ $|\Psi^{(N,O)}_n - \Psi_n| \le C \xi_n^{1/2} \exp(-N^{\sigma}) := \xi^{1/2} \delta $, with $C, \sigma$ depending on $a_0, b_0, \kappa_0, \nu$.

$(3)$ Let $\theta^\zero \in \mathcal{M}$ be arbitrary. Let $\theta^\zero (t)$ and $\theta^{(N,0)}(t)$ be the $\Psi$-trajectory, resp.\ $\Psi^{(N)}$-trajectory originating at $\theta = \theta^\zero$. Then
\begin{equation}\label{eq:9trajecmetricdeviation2-2}
\max_{0 \le t \le T} d(\theta^\zero (t), \theta^{(N,0)}(t)) \le A (4K)^{BT} \delta.
\end{equation}

$(4)$ Let $\theta^\zero, \theta^\one \in \mathcal{M}$. Let $\theta^{(j)}(t)$ be the $\Psi$-trajectory originating at $\theta = \theta^{(j)}$, $j = 1, 2$. Then
\begin{equation}\label{eq:9trajecmetricdeviation2theta}
\max_{0 \le t \le T} d(\theta^\zero (t),\theta^{(1)}(t)) \le A (4K)^{BT} d(\theta^\zero, \theta^\one).
\end{equation}
\end{prop}

\begin{proof}
The verification of $(1)$ is completely similar to the one we did for $\Psi$.

Furthermore, for $n \le N$, we have
\begin{equation}\label{eq:8auxmanofoldNcompare1}
\begin{split}
|\Psi_{N,n}(\theta) - \Psi_n(\theta)| = 4 \xi_n \sqrt{(\mu_n(\theta) - \underline{E}) \prod_{\substack{i \\ i \neq n, \quad i \le N}} \Big| \frac{(E_i^-  - \mu_n(\theta))(E_i^+ - \mu_n(\theta))}{(\mu_i(\theta) - \mu_n(\theta))^2} \Big| } \times \\
\Big| 1 - \prod_{\substack{i \\ i \neq n, \quad i > N}} \frac{(E_i^- - \mu_n(\theta))(E_i^+ - \mu_n(\theta))}{(\mu_i(\theta) - \mu_n(\theta))^2} \Big|.
\end{split}
\end{equation}
Just as in the proof of Corollary~\ref{cor:8Psiregular}, we have
\begin{equation}\label{eq:8auxmanofoldiff111}
4 \xi_n \sqrt{(\mu_n(\theta) - \underline{E}) \prod_{\substack{i \\ i \neq n, \quad i \le N}} |\frac{(E_i^- - \mu_n(\theta))(E_i^+ - \mu_n(\theta))}{(\mu_i(\theta) - \mu_n(\theta))^2} }| \le C \xi_n^{1/2}.
\end{equation}
Note that for $i > N$ and $n \le N$, we have $i \in \mathfrak{N}_n$. Therefore, Lemma~\ref{Lml2} applies. This yields
\begin{equation}\label{eq:8auxmanofoldNcompare2}
\Big| 1 - \prod_{\substack{i \\ i \neq n, \quad i > N}} \frac{(E_i^- - \mu_n(\theta))(E_i^+ - \mu_n(\theta))}{(\mu_i(\theta) - \mu_n(\theta))^2} \Big| \le \sum_{i > N} \gamma_i^{1/2} \le C(a_0, b_0, \kappa_0, \nu) \exp(-N^{\sigma_0(a_0, b_0, \kappa_0, \nu)}).
\end{equation}
So, for $n \le N$, the statement in $(2)$ holds. For $n > N$, we have
\begin{equation}\label{eq:8auxmanofoldNcompare3}
|\Psi_{N,n}(\theta) - \Psi_n(\theta)| = |\Psi_n(\theta)| \le C \xi_n^{1/2} \exp(-N^{\sigma}).
\end{equation}

Part $(3)$ follows from $(2)$ due to Corollary~\ref{Lm9:voltera1}.

Part $(4)$ is due to Corollary~\ref{Lm9:voltera1}.
\end{proof}

\begin{remark}\label{rem:9closeODE1CONV}
$(1)$  Let $\omega = (\omega_1, \ldots, \omega_\nu) \in \mathbb{R}^\nu$ and $\tilde \omega^{(r)} = (\tilde \omega_1^{(r)}, \ldots, \tilde \omega_\nu^{(r)})$ be as in Section~\ref{sec.5}. For each $r$, let
\begin{equation}\label{eq:52VtildeCONV}
V^\ar (x) = \sum_{n \in \zv} c(n) e^{2 \pi i x n \tilde \omega^\ar}\ , \quad x \in\mathbb{R},
\end{equation}
\begin{equation}\label{eq:8Fourier}
\begin{split}
\overline{c(n)} & = c(-n), \quad n \in \zv , \quad c(0) = 0,\\
|c(n)| & \le \ve \exp(-\kappa_0|n|^{\alpha_0}), \quad n \in \zv ,
\end{split}
\end{equation}
\begin{equation} \label{eq:517-1CONV}
[H^\ar y](x) = -y''(x) + V^\ar(x) y(x) = Ey(x), \quad x, \ve \in \IR;
\end{equation}
compare with the notation in Remark~\ref{rem:5PA1}. Let $\mathfrak{Z}^\ar = \mathfrak{Z}(\tilde \omega^{(r)})$, and let $G^\ar_{m} = (E_{m}^{r,-}, E_m^{r,+})$, $m \in \mathfrak{Z}^\ar$ be the gaps in the spectrum $\mathcal{S}^\ar$ of $H^\ar$. Let $\mathcal{M}^\ar$, $\mathcal{X}^\ar$ be the manifolds defined in part $(2)$ of Remark~\ref{rem:9closeODE1} with $\tilde \omega^\ar$ in the role of $\tilde \omega$. Let $\Psi^\ar$ and $\Psi^{(r,N)}$ be the vector-fields defined in \eqref{eq:8auxmanofold} and \eqref{eq:8auxmanofoldN}, respectively. Our ultimate goal in this section is to show \textbf{the convergence of the trajectories} of the vector-fields $\Psi^\ar$ as $r \to \infty$.

$(2)$ Let
\begin{equation}\label{eq:52VtildeCONVL}
V (x) = \sum_{n \in \zv} c(n) e^{2 \pi i x n \omega}\ , \quad x \in \mathbb{R},
\end{equation}
\begin{equation} \label{eq:517-1CONVL}
[H y](x) = -y''(x) + V(x) y(x) = Ey(x), \quad x, \ve \in \IR.
\end{equation}
Let $G_{m} = (E_{m}^{-}, E_m^{+})$, $m \in \mathbb{Z}^\nu$ be the gaps in the spectrum $\mathcal{S}$ of $H$. We define the manifolds $\mathcal{M}$, $\mathcal{X}$ just as in part $(2)$ of Remark~\ref{rem:9closeODE1}. Let $\Psi$ and $\Psi^{(N)}$ be the vector-fields defined in \eqref{eq:8auxmanofold} and \eqref{eq:8auxmanofoldN}, respectively. Proposition~\ref{prop:8.stabil1} applies just as before.

$(3)$ In Remark~\ref{rem:9closeODE1CONV1} below we set up the definitions in part $(2)$ of Remark~\ref{rem:9closeODE1} so that
\begin{equation} \label{eq:517-1specidentL}
\mathcal{X}^\ar = \mathcal{X}.
\end{equation}
To get this identity we just need to enumerate the ``relatively large gaps'' in $\mathcal{S}$ and $\mathcal{S}^\ar$ so that this enumeration will actually define a bijection between these two sets of gaps. Naturally, Theorem~\ref{Tdistspec} is instrumental for this task. The identity \eqref{eq:517-1specidentL} allows us to show the convergence of the trajectories.
\end{remark}

\begin{lemma}\label{Lml2PQloverupperbound}
$(1)$ Let $\underline{E} = \inf \mathcal{S}$, $\underline{E}^\ar = \inf \mathcal{S}^\ar$. Then $\underline{E}, \underline{E}^\ar \ge -1$.

$(2)$ $E^+_m \le C(\kappa_0) |\omega|^2 (\log |G_m|^{-1})^2$, $E^{r,+}_m \le C(\kappa_0) |\omega|^2 (\log |G^\ar_m|^{-1})^2$.
\end{lemma}

\begin{proof}
Due to \eqref{eq:8Fourier}, for $0 < \ve \le \ve_0$, we have $|V(x)|, |V^\ar(x)| \le 1$ for any $x$. It is well known that this implies the estimate in $(1)$.

Recall that $E^+ = E^+(k_m)$ and $E(k) \le C |\omega|^2 |m|^2$, see Theorem~$A$. On the other hand, due to Theorem~$A$, $|G_m| \le \exp(-\frac{\kappa_0}{2} |m|)$. This implies $(2)$ for $H$. The argument for $H^\ar$ is completely similar.
\end{proof}

\begin{lemma}\label{lem:8enumerinj}
There exists $C_0(\kappa_0)$ such that if for some $\tau$, we have $\lambda'_r := C_0(\kappa_0) |\omega|^2 (\log \tau^{-1})^2 \lambda_r < \tau/4$ with $\lambda_r = \lvert \omega - \tilde \omega^\ar \rvert^{1/2}$, then there is an injection $m \mapsto \mathfrak{n}(m) \in \mathfrak{Z}^\ar$ defined for all $m$ with $|G_m| \ge \tau$, such that $|E^\pm_m - E^{r,\pm}_{\mathfrak{n}(m)}| < \lambda'_{r}$.
\end{lemma}

\begin{proof}
Due to Theorem~\ref{Tdistspec},
\begin{equation}\label{eq:8specdistest}
d (\mathcal{S} \cap (-E,E), \mathcal{S}^\ar \cap (-E,E)) \le C (1 + \lvert E \rvert^{1/4}) \lambda_r,
\end{equation}
where $d$ denotes Hausdorff distance.

Let $|G_m| = E^+_m - E^-_m \ge \tau$. Due to Lemma~\ref{Lml2PQloverupperbound}, we have $E^+_m \le C(\kappa_0) |\omega|^2 (\log \tau^{-1})^2$. Since $\lambda'_{r} < |G_m|/4$, due to \eqref{eq:8specdistest} we have
\begin{equation}\label{eq:8specdistest2}
[E^-_m + \lambda'_{r}, E^+_m - \lambda'_{r}] \subset \bigcup_\mathfrak{n} G^\ar_{\mathfrak{n}}.
\end{equation}
Take $E^-_m + \lambda_{r,m} < E' < E^+_m - \lambda'_{r}$. Since the gaps do not overlap, there exists a unique $G^\ar_{\mathfrak{n}(m)}$ such that $E' \in G^\ar_{\mathfrak{n}(m)}$. Moreover, $[E^-_m + \lambda_{r,m}, E^+_m - \lambda_{r,m}] \subset G^\ar_{\mathfrak{n}(m)}$. Indeed, assume that $[E^-_m + \lambda'_{r}, E^+_m - \lambda'_{r}] \cap G^\ar_{\mathfrak{n}}$ for some $\mathfrak{n} \neq \mathfrak{n}(m)$. Since the gaps do not overlap this would contradict \eqref{eq:8specdistest2}. We can exchange the roles of $G_m$ and $G^\ar_{\mathfrak{m}(n)}$ in this argument. Indeed, the only thing we need for this is the lower estimate for $|G^\ar_m|$. Clearly, $|G^\ar_m| \ge \tau/2$. One can see that with this lower estimate the above argument still works. This proves that the map $m \mapsto \mathfrak{n}(m)$ is injective and $|E^\pm_m - E^{r,\pm}_{\mathfrak{n}(m)}| < \lambda'_{r}$.
\end{proof}

\begin{remark}\label{rem:9closeODE1CONV1}
The previous lemma enables us to define $\mathcal{X}^\ar, \mathcal{X}$ so that \eqref{eq:517-1specidentL} holds. Indeed, let us enumerate the gaps in the spectrum $\mathcal{S}$ via $G_n$, $n = 1, 2, \ldots$, just as in Remark~\ref{rem:9closeODE1}. The statement in Lemma~\ref{lem:8enumerinj} does not specify the labeling of
the gaps. So, we use the same injection $n \mapsto \mathfrak{m}(n)$ with $n = 1, 2, \ldots$. Set $G^\ar_n := G^\ar_{\mathfrak{m}(n)}$. In Lemma~\ref{lem:8enumerinj11} below we verify conditions \eqref{eq:8Gnesti} in part $(1)$ of Remark~\ref{rem:9closeODE1}. Due to these conditions Proposition~\ref{prop:8.stabil1} applies, as was mentioned in part $(1)$ of Remark~\ref{rem:9closeODE1}. Thus, indeed we can define $\mathcal{X}^\ar, \mathcal{X}$ being the same.
\end{remark}

\begin{lemma}\label{lem:8enumerinj11}
Set $G^\ar_n := G^\ar_{\mathfrak{m}(n)}$. Then,
\begin{equation}\label{eq:8GnestiAP1}
\begin{split}
c'(a_0, b_0, \kappa_0, \nu) |\mathfrak{m}(n)|^{\beta'(a_0, b_0, \kappa_0, \nu)} \le |n| \le C'(a_0, b_0, \kappa_0, \nu) |\mathfrak{m}(n)|^{b'(a_0, b_0, \kappa_0, \nu)}, \\
c'(a_0, b_0, \kappa_0, \nu) \exp(-n^{\sigma'_1(\kappa_0, \nu)}) \le |G^\ar_{\mathfrak{m}(n)}| \le C'(a_0, b_0, \kappa_0, \nu) \ve_0 \exp(-n^{\sigma'(\kappa_0, \nu)}).
\end{split}
\end{equation}
\end{lemma}

\begin{proof}
Recall that due to \eqref{eq:8Gnesti},
\begin{equation}\label{eq:8GnestiAP}
c(a_0, b_0, \kappa_0, \nu) \exp(-n^{\sigma_1(a_0, b_0, \kappa_0, \nu)}) \le |G_n| \le C(a_0, b_0, \kappa_0, \nu) \ve_0 \exp(-n^{\sigma(a_0, b_0, \kappa_0, \nu)}).
\end{equation}
Since $|G^\ar_{\mathfrak{m}(n)}|/2 \le |G_n| \le 2 |G^\ar_{\mathfrak{m}(n)}|$, the second line in \eqref{eq:8GnestiAP1} holds. Since $|G^\ar_{\mathfrak{m}}| \le C(a_0, b_0, \kappa_0, \nu) \exp(-|m|^{\sigma(a_0, b_0, \kappa_0, \nu)})$, the first line follows from the second.
\end{proof}

\begin{lemma}\label{lem:8enumerinj114}
With $\lambda'_r$ as in Lemma~\ref{lem:8enumerinj}, we have $|\underline{E}^\ar - \underline{E}| \le \lambda'_r \to 0$ as $r \to \infty$.
\end{lemma}

\begin{proof}
The statement follows from \eqref{eq:8specdistest} combined with part $(1)$ of Lemma~\ref{Lml2PQloverupperbound}.
\end{proof}

Let $\Psi_{N,n}(\theta)$ be as in Proposition~\ref{prop:8.stabil1}, see \eqref{eq:8auxmanofoldN}. Consider the vector-fields on $\mathcal{X}$ defined via
\begin{equation}\label{eq:8auxmanofoldN111}
\begin{split}
\Psi^\ar_{N,n}(\theta) & =  4 \xi_n \sqrt{(\mu^\ar_n(\theta) - \underline{E}^\ar) \prod_{\substack{i \\ i \neq n, \quad i \le N}} \frac{(E_i^{r,-} - \mu^\ar_n(\theta))(E_i^{r,+} - \mu^\ar_n(\theta))}{(\mu^\ar_i(\theta) - \mu^\ar_n(\theta))^2} }, \quad n \le N, \\
\Psi^{(r,N)} & = (\Psi^\ar_{N,n})_n.
\end{split}
\end{equation}

\begin{lemma}\label{lem:8enumerinj113}
We have
\begin{equation}\label{eq:8GnestiAP13}
\lim_{r \to \infty} \Psi^\ar_{N,n}(\theta) = \Psi_{N,n}(\theta), \quad 1 \le n \le N,
\end{equation}
uniformly in $\theta$.
\end{lemma}

\begin{proof}
Recall that $\mu_n(\theta) = E^-_n + (E^+_n - E^-_n) \sin^2 (\frac{\theta_n}{\rho_n})$, $\mu^\ar_n(\theta) = E^{r,-}_n + (E^{r,+}_n - E^{r,-}_n) \sin^2 (\frac{\theta_n}{\rho_n})$. Due to Lemma~\ref{lem:8enumerinj}, $|E^{r,\pm}_n - E^{\pm}_n| \to 0$ as $r \to \infty$. Recall also that $\eta_{i,n} = \dist(G_i, G_n) > 0$. This implies $\lim_{r \to \infty} \mu^\ar_n(\theta) = \mu_n(\theta)$, uniformly in $\theta$. Due to Lemma~\ref{lem:8enumerinj114}, $|\underline{E}^\ar - \underline{E} | \to 0$ as $r \to \infty$. Recall that $\underline{E}^\ar < E^-_n \le \mu_n(\theta)$ for any $n$ and $\theta$. The statement follows.
\end{proof}

\begin{remark}\label{rem:9PhiPsiconnection}
We discuss now the fields $\Phi$, see \eqref{eq:8toruspm-1}, \eqref{ml06}. Using the same enumeration as above, set
\begin{equation}\label{eq:8toruspmPsi}
\begin{split}
G_{n,\sigma} & = \{ (\mu,\sigma) : \mu \in G_n \},\quad G^\ar_{n,\sigma} = \{ (\mu,\sigma) : \mu \in G^\ar_m \}, \quad \sigma = \pm, \\
\mathcal{C}_{n} & = G_{n,-} \cup G_{n,+} \cup \{ E_n^-, E_n^+ \},\quad \mathcal{C} = \prod_{n} \mathcal{C}_n, \\
\mathcal{C}^\ar_{n} & = G^\ar_{n,-} \cup G^\ar_{n,+} \cup \{ E_n^{r,-}, E_n^{r,+} \}, \quad \mathcal{C}^\ar = \prod_{n} \mathcal{C}^\ar_n.
\end{split}
\end{equation}
The tori $\mathcal{C}$, $\mathcal{C}^\ar$ have natural smooth structures. We denote the points on these tori by $\mu=(\mu_n)_n$ with $\sigma(\mu_n) \in \{+,-\}$ being suppressed from the notation. Consider the vector-fields
\begin{equation}\label{eq:8toruspmPsivf}
\begin{split}
\Phi_n(\mu) & = \sigma(\mu_n) \sqrt{4 ( \underline E - \mu_n)(E_n^- - \mu_n)(E_n^+ - \mu_n) \prod_{\substack{i \neq n}} \frac{(E_i^- - \mu_n)(E_i^+ - \mu_n)}{(\mu_i - \mu_n)^2} }, \\
\Phi^\ar_n(\mu) & = \sigma(\mu_n) \sqrt{4 (\underline E^\ar - \mu_n)(E_n^{r,-} - \mu_n)(E_n^{r,+} - \mu_n) \prod_{\substack{i \neq n}} \frac{(E_i^{r,-} - \mu_n)(E_i^{r,+} - \mu_n)}{(\mu_i - \mu_n)^2}}.
\end{split}
\end{equation}
Recall \eqref{eq:8auxmanofold}:
\begin{equation}\label{eq:8auxmanofoldPhiPsi}
\begin{split}
\mu_n(\theta) = E^-_n + (E^+_n - E^-_n) \sin^2 \Big(\frac{\theta_n}{\xi_n}\Big), \quad \theta := (\theta_n,\sigma_n)_n \in \mathcal{M}, \quad \xi_n = \exp(-n^{\sigma(\kappa_0, \nu)}/2), \\
\mu^\ar_n(\theta) = E^{r,-}_n + (E^{r,+}_n - E^{r,-}_n) \sin^2 \Big(\frac{\theta_n}{\xi_n}\Big), \quad \theta := (\theta_n, \sigma_n)_n \in \mathcal{M}, \quad \xi_n = \exp(-n^{\sigma(\kappa_0, \nu)}/2).
\end{split}
\end{equation}
\end{remark}

\begin{lemma}\label{lem:8Lml2ODEmutheta}
The function $\mu(t)$ obeys
\begin{equation}\label{ml07origODE}
\dot{\mu} = \Phi(\mu), \quad \text{resp., $\dot{\mu} = \Phi^\ar(\mu)$}
\end{equation}
if and only if $\mu(t) = \mu(\theta(t))$, resp.\ $\mu(t) = \mu^\ar(\theta(t))$, with
\begin{equation}\label{ml07origODEtheta}
\dot{\theta} = \Psi(\theta), \quad \text{resp., $\dot{\theta} = \Psi^\ar(\theta)$.}
\end{equation}
In particular, for any given $\mu^0 \in \mathcal{C}$ (resp., $\mu^0 \in \mathcal{C}^\ar$), there exists a unique trajectory originating at $\mu^0$.
\end{lemma}

\begin{proof}
We have
\begin{equation}\label{ml07origODEthetaverif}
\begin{split}
\Phi_n(\mu) & = \sigma_n \sqrt{(\mu_n(\theta_n) - E_n^-)(E_n^+ - \mu_n(\theta_n))} \Psi_n(\theta) = \sigma_n (E^+_n - E^-_n) \sin \Big( \frac{\theta_n}{\xi_n} \Big) \cos \Big( \frac{\theta_n}{\xi_n} \Big) \Psi_n(\theta), \\
\dot{\mu}_n & = \frac{d\mu_n}{d\theta_n} \dot{\theta}_n = \frac{2(E^+_n - E^-_n)}{\xi_n} \sin \Big( \frac{\theta_n}{\xi_n} \Big) \cos \Big( \frac{\theta_n}{\xi_n} \Big) \dot{\theta}_n.
\end{split}
\end{equation}
This verifies the statement for $\Phi$. The verification for $\Phi^\ar$ is the same.
\end{proof}

Now we can state the main results of this section. Let $\mathcal{M}, \mathcal{X}, \mathcal{Y}, \mathcal{M}^\ar, \mathcal{X}^\ar, \mathcal{Y}^\ar, \Psi, \Psi^\ar, \Psi^{(N)}, \Psi^{(r,N)}$ be as above. As above we have a setup with $\mathcal{X} = \mathcal{X}^\ar$. The manifolds $\mathcal{Y}$ and $\mathcal{Y}^\ar$ are not related to each other. Let $\theta^\zero = (\theta_n^\zero)_n \in \mathcal{M}$ be arbitrary. Let $\theta^\zero (t) = (\theta_n^\zero(t))_n$ be the $\Psi$-trajectory originating at $\theta^\zero$. We have $\theta^{N,0} := (\theta_n^\zero)_{n \le N} \in \mathcal{X}^\ar$. Take an arbitrary $\gamma^\ar \in \mathcal{Y}^\ar$. Let $\theta^\zero (t; \gamma^\ar) = (\theta_n^\zero (t; \gamma^\ar))_n$ be the $\Psi^\ar$-trajectory originating at $(\theta^{N,0}, \gamma^\ar)$.

\begin{prop}\label{prop:8.stabilM}
We have
\begin{equation}\label{eq:9trajecmetricdeviation234}
\begin{split}
\limsup_{r \to \infty} \left[ \sup_{\gamma^\ar} \max_{0 \le t \le T} \sum_{1 \le n \le N} |\theta^\zero_n (t) - \theta_n^{(0)}(t; \gamma^\ar)| \right] \le A (4K)^{BT} \delta, \\
\limsup_{r \to \infty} \left[ \sup_{\gamma^\ar} \max_{0 \le t \le T} \sum_{1 \le n \le N} |\mu_n(\theta^\zero_n (t)) - \mu^\ar_n(\theta_n^{(0)}(t; \gamma^\ar))| \right] \le A (4K)^{BT} \delta,
\end{split}
\end{equation}
with $A, B, K$ depending on $a_0, b_0, \kappa_0, \nu$ and $\delta = \exp(-N^{\sigma})$.
\end{prop}

\begin{proof}
Let $\Psi_{N,n}(\theta), \Psi^\ar_{n}(\theta)$ be as in Lemma~\ref{lem:8enumerinj113}, $\Psi^{N} = (\Psi_{N,n}(\theta))_{1 \le n \le N}, \Psi^{r,N}(\theta) = (\Psi^\ar_{n}(\theta))_{1 \le n \le N}$. Let $\theta^{(0)} (t;N) = (\theta_n^\zero (t;N))_{1 \le n \le N}$ be the $\Psi^N$-trajectory originating at $(\theta_n^{(0)})_{1 \le n \le N}$. Let $\theta^\zero (t;r,N) = (\theta_n^\zero (t; r, N))_{1 \le n \le N}$ be the $\Psi^{r,N}$-trajectory originating at $(\theta_n^{(0)})_{1 \le n \le N}$. Due to Lemma~\ref{lem:8enumerinj113},
\begin{equation}\label{eq:9trajecmetricdeviation234lim}
\lim_{r \to \infty} \Psi^\ar_{N,n}(\theta) = \Psi^\ar_{n}(\theta), \quad 1 \le n \le N,
\end{equation}
uniformly in $\theta$. Recall also that due to Corollary~\ref{cor:8PsiregularM} the vector-fields $\Psi^{N}, \Psi^{r,N}(\theta)$ are $B$-regular with respect to the atlas $\mathfrak{F}$, with $B = B(a_0, b_0, \kappa_0, \nu) > 0$. Combining this fact with \eqref{eq:9trajecmetricdeviation234lim}, we conclude that
\begin{equation}\label{eq:9trajecmetricdeviationinter}
\lim \left[ \sum_{1 \le n \le N} |\theta^\zero_n (t;N) - \theta_n^{(0)}(t;r,N)| \right] = 0
\end{equation}
$($this is a well-known general fact, but one can invoke Corollary~\ref{eq:9trajecmetricdeviation2-1}$)$. Due to Proposition~\ref{prop:8.stabil1},
\begin{equation}\label{eq:9trajecmetricdeviation2AG}
\begin{split}
\max_{0 \le t \le T} \sum_{1 \le n \le N} |\theta_n^\zero (t) - \theta_n^{(0)}(t;N)| & \le A (4K)^{BT} \delta, \\
\max_{0 \le t \le T} \sum_{1 \le n \le N} |\theta_n^\zero (t;\gamma_r) - \theta_n^{(0)}(t;r,N)| & \le A (4K)^{BT} \delta.
\end{split}
\end{equation}
Combining \eqref{eq:9trajecmetricdeviation2AG} with \eqref{eq:9trajecmetricdeviationinter}, we obtain the first line in \eqref{eq:9trajecmetricdeviation234}. We have
\begin{equation}\label{eq:8auxmanofoldPhiPsi11}
\begin{split}
\sum_{1 \le n \le N} |\mu_n(\theta_n) - \mu_n(\hat{\theta}_n)| & \le \sum_{1 \le n \le N} (E^+_n - E^-_n) |\theta_n - \hat{\theta}_n| \xi_n^{-1} \le C \sum_{1 \le n \le N} |\theta_n - \hat{\theta}_n|, \\
\sum_{1 \le n \le N} |\mu_n(\theta_n) - \mu_n^\ar(\theta_n)| & \le \sum_{1 \le n \le N} [|E^-_n - E^{r,-}| + |E^+_n - E^{r,+}|] \to 0 \text{ as $r \to \infty$}.
\end{split}
\end{equation}
Therefore the second line in \eqref{eq:9trajecmetricdeviation234} follows from the first one.
\end{proof}

\section{Proof of Theorem~$I$}\label{sec.10}

The next lemma is well known.

\begin{lemma}\label{Lml:8quasibasis}
Let
\begin{equation}\label{ml07qbasis}
Q(t) = \sum_{n \in \zv} d(n) e^{2 \pi i t n \omega}, \quad t \in \mathbb{R}
\end{equation}
with
\begin{equation}\label{eq:8-4decay-1}
|d(n)| \le \exp(-\kappa_0 |n|).
\end{equation}
Let $T > 0$ be arbitrary and let $\delta_T := \max_{0 \le t \le T} |Q(t)|$. Let $\beta_0 := (2 b_0)^{-1}$. If $|n| \le T^{\beta_0}$, then
\begin{equation}\label{eq:qpbasis4}
|d(n)| \le \delta_T + \frac{C(a_0, b_0, \kappa_0, \nu)}{T^{1/2}} + C(\kappa_0, \nu) \exp(-\kappa_0 T^{\beta_0}).
\end{equation}
In particular, if $Q(t) = 0$ for all $t \in \mathbb{R}$, then $d(n) = 0$ for all $n \in \zv$.
\end{lemma}

\begin{proof}
Take an arbitrary $n \in \zv$. We have
\begin{equation}\label{eq:qpbasis1}
\Big| d(n) + \sum_{m \in \zv, \, m \neq n} d(m) \frac{1}{T} \int_0^T e^{2 \pi i t (m-n) \omega} \, dt \Big| = \Big| \frac{1}{T} \int_0^T Q(t) e^{-2 \pi i t n \omega} \, dt \Big| \le \delta_T.
\end{equation}
Note that for any $m \neq n$, we have
\begin{equation}\label{eq:qpbasis2}
\frac{1}{T} \Big| \int_0^T e^{2 \pi i t (m-n) \omega} \, dt \Big| \le \min \Big( 1, \frac{1}{\pi T |(m-n) \omega|} \Big) \le \min \Big( 1, \frac{|m-n|^{b_0}}{\pi a_0 T} \Big).
\end{equation}
Assume $|n| \le T^{\beta_0}$. Then
\begin{equation}\label{eq:qpbasis3}
\begin{split}
\sum_{m \neq n, \, |m| \le T^{\beta_0}} |d(m)| \frac{1}{T} \Big| \int_0^T e^{2 \pi i t (m-n) \omega} \, dt \Big|
& \le \sum_{|m| \le T^{\beta_0}} |d(m)| \frac{(|n|+|m|)^{b_0}}{\pi a_0 T} \le \frac{C(a_0, b_0, \kappa_0, \nu)}{T^{1/2}}, \\
\sum_{m \neq n, \, |m| > T^{\beta_0}} |d(m)| \frac{1}{T} \Big| \int_0^T e^{2 \pi i t (m-n) \omega} \, dt \Big| & \le
\sum_{|m| > T^{\beta_0}} |d(m)| \le C(\kappa_0, \nu) \exp(-\kappa_0 T^{\beta_0}).
\end{split}
\end{equation}
Combining this with \eqref{eq:qpbasis1}, we obtain \eqref{eq:qpbasis4}. If $Q(x) = 0$ for all $x$, then taking $T \to \infty$, we conclude that $d(n) = 0$ for all $n$.
\end{proof}

Let $\omega$ and $\tilde \omega = \tilde \omega^{(r)} = (\tilde \omega_1^{(r)}, \ldots, \tilde \omega_\nu^{(r)})$ be as in Lemma~\ref{lem:PAomegas}; see Section~\ref{sec.5}. Let
\begin{equation}\label{ml07qbasisexp1}
Q^{(r)}(t) = \sum_{n \in \zv} c^{(r)}(n) e^{2 \pi i t n \tilde \omega^{(r)}}, \quad t \in \mathbb{R},
\end{equation}
with
\begin{equation}\label{eq:8-4decay-2}
|c^{(r)}(n)| \le  \ve \exp(-\kappa_0|n|), \quad n \in \zv.
\end{equation}

\begin{lemma}\label{Lml:8quasibasisconverge1}
Assume that the limit $Q(t) = \lim_{r \to \infty} Q^{(r)}(t)$ exists for every $t$. Then the limits $c(n) = \lim_{r \to \infty} c^{(r)}(n)$ exist for every $n \in \zv$. Moreover,
\begin{equation}\label{eq:8-4decay123-1}
\begin{split}
Q(t) & = \sum_{n \in \zv} c(n) e^{2 \pi i t n\omega}, \quad t \in \mathbb{R}, \\
|c(n)| & \le \ve \exp(-\kappa_0|n|), \quad n \in \zv.
\end{split}
\end{equation}
\end{lemma}

\begin{proof}
Assume that the limit $\lim_{r \to \infty} c^{(r)}(n^\zero)$ does not exist for some $n^\zero \in \zv$. Using the condition \eqref{eq:8-4decay-2}, we can find two sequences $r'_s$ and $r''_s$, $s = 1, 2, \ldots$ such that the limits $c'(n) = \lim_{s \to \infty} c^{(r'_s)}(n)$,  $c''(n) = \lim_{s \to \infty} c^{(r''_s)}(n)$ exist for every $n \in \zv$, and $c'(n^\zero) \neq c''(n^\zero)$. Moreover,
\begin{equation}\label{eq:8-4decay125-1}
|c'(n)|, |c''(n)| \le \ve \exp(-\kappa_0 |n|), \quad n \in \zv.
\end{equation}
Once again, using the condition \eqref{eq:8-4decay-2} and the estimates \eqref{eq:8-4decay125-1}, we obtain
\begin{equation}\label{eq:8-4decay126-1}
\begin{split}
\lim_{s \to \infty} \sum_{n \in \zv} c^{(r'_s)}(n) e^{2 \pi i t n\omega^{(r'_s)}} & = \sum_{n \in \zv} c'(n) e^{2 \pi i t n \omega}, \quad t \in \mathbb{R}, \\
\lim_{s \to \infty} \sum_{n \in \zv} c^{(r''_s)}(n) e^{2 \pi i t n \omega^{(r''_s)}} & = \sum_{n \in \zv} c''(n) e^{2 \pi i t n \omega}, \quad t \in \mathbb{R}
\end{split}
\end{equation}
$($recall that $|\omega - \omega^{(r)}| \to 0$ when $r \to \infty)$. Due to the assumptions of the lemma, we have
\begin{equation}\label{eq:8-4decay126-2}
\sum_{n \in \zv} c'(n) e^{2 \pi i t n \omega} = \sum_{n \in \zv} c''(n) e^{2 \pi i t n \omega}, \quad t \in \mathbb{R}.
\end{equation}
Due to the previous lemma this implies $c'(n) = c''(n)$ for any $n \in \zv$. The contradiction we get proves the first statement in the lemma. The second one follows from it.
\end{proof}
\begin{remark}\label{rem:9fortheorI}
$(1)$ In the proof of Theorem $I$ we need to invoke yet another fundamental fact from the theory of Hill operators.
Let $q(x)$ be a continuous real $T$-periodic function,
\begin{equation} \label{eq:9potentialR}
q(x) = \sum_{n \in \mathbb{Z} \setminus \{ 0 \}} c(n) e^{\frac{2\pi i nx}{T}}.
\end{equation}
Consider the Hill equation
\begin{equation} \label{eq:9HillR-1}
[H_qy](x) = -y''(x) + q(x) y(x) = E y(x), \quad x \in \IR.
\end{equation}
Denote by $\mu_1(q) < \mu_2(q) < \cdots$ the Dirichlet eigenvalues on $[0,T]$. We view $\mu_n(q)$ as a functional of $q \in C[0,T]$. The following estimate holds,
\begin{equation} \label{eq:9HillR-2}
|\mu_n(q) - \mu_n(p)| \le \max_{0 \le x \le T} |q(x)-p(x)|.
\end{equation}
For $T = 1$, we may refer to \cite[p.~34]{PoTr}. For arbitrary $T$, we can employ the standard re-scaling
$$
\tilde q(x) = T^2 q(xT), \quad 0 \le x \le 1,
$$
which defines a $1$-periodic function $\tilde q$ with $\mu_n(\tilde q) = T^2 \mu_n(q)$.

$(2)$ In the proof we also invoke the following results of Craig \cite{Cr}. Let $Q(x)$ be a bounded continuous real function, $t \in \mathbb{R}$. Consider the Schr\"odinger operator
\begin{equation} \label{eq:9craig1}
[H y](x) = - y''(x) +  Q(x) y(x), \quad x \in \IR.
\end{equation}
Let $\mathcal{S}$ be the spectrum of $H$. Assume that
$$
\mathcal{S} = [\underline{E}, \infty) \setminus \bigcup_{n \ge 1} (E^-_n,E^+_n),
$$
where the gaps $G_n := (E^-_n, E^+_n)$ obey
\begin{equation} \label{eq:9craig2}
\sum_n \gamma_n < \infty, \quad \sum_{i, n : i \neq n} \frac{\gamma_i^{1/2} \gamma_n^{1/2}}{\eta_{i,n}} < \infty,
\end{equation}
$\gamma_n := |G_n|$, $\eta_{i,n} := \dist (G_i,G_n)$. Finally, assume that $H$ is reflectionless, see Remark~\ref{rem:reflectionles}. Then,
\begin{equation} \label{eq:craig3}
Q(t) = \sum_{n} E_n^+ + E^-_n - 2\mu_n(t),
\end{equation}
where the functions $\mu_n(t)$ obey the differential equation
\begin{equation}\label{eq:9craig4}
\dot{\mu}_n = \Phi_n(\mu) = \sigma(\mu_n) \sqrt{4 (\underline E - \mu_n)(E_n^- - \mu_n)(E_n^+ - \mu_n) \prod_{\substack{i \neq n}} \frac{(E_i^- - \mu_n)(E_i^+ - \mu_n)}{(\mu_i - \mu_n)^2} },
\end{equation}
$\sigma_n \in \{ +, - \}$ and $\mu_n(t) \in [E_n^-, E_n^+]$. The ODE \eqref{eq:9craig4} is defined on the manifold $\mathcal{C}$ just as for the Hill operator.

$(3)$ In the proof of Theorem~$I$ we invoke Theorem~$\tilde I$. This gives the uniform exponential estimates on the Fourier coefficients which are crucial for the derivation of Theorem~$I$. It is easy to see that the estimates \eqref{eq:8coeffexpdecay} known in the setting of a general Hill equation with analytic potential are ineffective as the period $T$ grows to $\infty$.
\end{remark}

\begin{proof}[Proof of Theorem~$I$]
Let $V$ be as in Theorem~$I$. We use Proposition~\ref{prop:8.stabilM}, see also Remark~\ref{rem:9closeODE1CONV}. Take an arbitrary $\theta^\zero = (\theta^\zero_n) \in \mathcal{M}$ and set
\begin{equation} \label{eq:9traceformula1-1}
\begin{split}
Q(t) & = Q(t; \theta^\zero) = \sum_{n} E_n^+ + E^-_n - 2 \mu_n(\theta^\zero_n (t)), \\
Q^\ar(t, \gamma^\ar) & = Q^\ar(t; \theta^\zero, \gamma^\ar) = \sum_{n} E_n^{r,+} + E^{r,-}_n - 2 \mu^\ar_n(\theta_n^{(0)}(t; \gamma^\ar)).
\end{split}
\end{equation}
Due to the McKean-van Moerbeke-Trubowitz trace formula, $Q^\ar(t,\gamma^\ar)$ is a periodic function with the same period as
$V^\ar$ and with the same spectrum, that is, $Q \in \mathcal{ISO}(V)$; see Theorem~$\tilde I$. By Theorem~$\tilde I$ we have
\begin{equation}\label{eq:9VtildeCONV}
\begin{split}
Q^\ar(t, \gamma^\ar) & = \sum_{n \in \zv} d^\ar(n,\gamma^\ar) e^{2 \pi i t n\tilde \omega^\ar}\ , \quad t \in\mathbb{R}, \\
|d^\ar(n, \gamma^\ar)| & \le \sqrt{2\ve} \exp \Big( -\frac{\kappa_0}{2} |n| \Big), \quad n \in \zv \setminus \{ 0 \}.
\end{split}
\end{equation}
Using Lemmas~\ref{lem:8enumerinj} and \ref{lem:8enumerinj11} and \eqref{eq:8Gnesti}, we obtain
\begin{equation} \label{eq:9traceformula2}
\begin{split}
\sum_{n} |E_n^+ + E^-_n - E_n^{r,+} + E^{r,-}_n| & = \sum_{n \le N} |E_n^+ + E^-_n - E_n^{r,+} + E^{r,-}_n| + \sum_{n > N} |E_n^+ + E^-_n - E_n^{r,+} + E^{r,-}_n| \\
& \le C N \lambda^\ar + C \sum_{n > N} \exp(-n^{\sigma(\kappa_0, \nu)}).
\end{split}
\end{equation}
Combining Proposition~\ref{prop:8.stabilM} with \eqref{eq:9traceformula2} we conclude that
\begin{equation} \label{eq:9traceformula1-2}
\lim_{r \to \infty} \max_{0 \le t \le T} |Q(t) - Q^\ar(t, \gamma^\ar)| = 0
\end{equation}
for any $T$. Due to Lemma~\ref{Lml:8quasibasisconverge1},
\begin{equation}\label{eq:8-4coeflim}
d(n) = \lim_{r \to \infty} d^\ar(n, \gamma^\ar)
\end{equation}
exists for every $n \in \zv$. Moreover,
\begin{equation}\label{eq:8-4decay123-2}
\begin{split}
Q(t) & = \sum_{n \in \zv} d(n) e^{2 \pi i t n\omega}, \quad t \in \mathbb{R}, \\
|d(n)| & \le \ve^{1/2} \exp \Big( -\frac{\kappa_0}{2} |n| \Big), \quad n \in \zv.
\end{split}
\end{equation}
Using the notation in Theorem~$I$, this shows that $Q \in \mathcal{P}(\omega, \ve^{1/2}, \kappa_0/2)$.

Set
\begin{equation}\label{eq:8-4decay124}
\begin{split}
\hat{Q}^\ar(t) & = \sum_{n \in \zv} d^\ar(n, \gamma^\ar) e^{2 \pi i t n \omega}, \quad t \in \mathbb{R}, \\
[H_Q y](t) & = - y''(t) +  Q(t) y(t), \\
[\hat{H}_Q^\ar y](t) & = - y''(t) + \hat{Q}^\ar(t) y(t), \\
[H_Q^\ar y](t) & = - y''(t) + Q^\ar(t, \gamma^\ar) y(t).
\end{split}
\end{equation}
Let $\mathcal{S}_Q$, $\hat{\mathcal{S}}^\ar_Q$ and $\mathcal{S}^\ar_Q$ be the spectrum of $H_Q$, $\hat{H}_Q^\ar$ and $H_Q^\ar$,  respectively. Clearly,
\begin{equation}\label{eq:8-4decay125-2}
\sup_t |\hat{Q}^\ar(t) - Q(t)| \to 0 \quad \text{as $r \to \infty$}.
\end{equation}

Fix some compact interval $I$. By \eqref{eq:8-4decay125-2}, it follows that
\begin{equation}\label{eq:8-4decay126-3}
\begin{split}
d (\mathcal{S}_Q\cap I, \hat{\mathcal{S}}^\ar_Q\cap I) \to 0 \quad \text{as $r \to \infty$},
\end{split}
\end{equation}
where $d$ again stands for the Hausdorff distance between the sets. Moreover, by Theorem~\ref{Tdistspec},
\begin{equation}\label{eq:8-4decay127}
d (\mathcal{S}^\ar_Q\cap I, \hat{\mathcal{S}}^\ar_Q\cap I) \le C |\omega - \tilde \omega^\ar|^{1/2} \to 0 \quad \text{as $r \to \infty$}.
\end{equation}
Thus,
\begin{equation}\label{eq:8-4decay128}
d (\mathcal{S}_Q\cap I, \mathcal{S}^\ar_Q\cap I) \to 0 \quad \text{as $r \to \infty$}.
\end{equation}
for any fixed compact interval $I$.

Recall that $\mathcal{S}^\ar_Q = [\underline{E}^\ar, \infty) \setminus \bigcup (E^{r,-}_{n}, E^{r,+}_{n})$. Due to Lemma~\ref{lem:8enumerinj}, $|E^\pm_n - E^{r,\pm}_{n}| < \lambda'_{r}$ for any $1 \le n \le N$, any $N$, provided $r \ge r(N)$. Here,
$$
\mathcal{S} = [\underline{E}, \infty) \setminus \bigcup_{n \ge 1} (E^{-}_{n}, E^{+}_{n})
$$
is the spectrum for $V$ as in Proposition~\ref{prop:8.stabilM}. Due to Lemma~\ref{lem:8enumerinj114}, $|\underline{E}^\ar - \underline{E}| \le \lambda'_r \to 0$ as $r \to \infty$. So,
\begin{equation}\label{eq:8-4decay129}
d(\mathcal{S}\cap I, \mathcal{S}^\ar_Q\cap I) \to 0 \quad \text{as $r \to \infty$}.
\end{equation}
for any fixed compact interval $I$. This implies
\begin{equation}\label{eq:8-4decay1210}
\mathcal{S}_Q = \mathcal{S},
\end{equation}
that is, $Q$ is isospectral with $V$. Thus, we have defined a map $\Phi$, given by
\begin{equation} \label{eq:9traceformula11}
\mathcal{M} \ni \theta^\zero \mapsto \Phi(\theta^\zero) := Q(t, \theta^\zero) := \sum_{n} E_n^+ + E^-_n - 2 \mu_n(\theta^\zero_n (t)) \in \mathcal{ISO}(V)
\cap \mathcal{P}(\omega, \sqrt{2\ve}, \kappa_0/2).
\end{equation}
Let us verify that this map is injective. Let $Q(t, \theta^\zero) = Q(t, \theta^{(1)})$, $\theta^\zero = (\theta^\zero_n)_n$, $\theta^\one = (\theta^\one_n)_n$. Let $Q^\ar(t; \theta^\zero, \gamma^\ar)$, $Q^\ar(t; \theta^\one, \gamma^\ar)$ be defined as in \eqref{eq:9traceformula1-1} with arbitrarily chosen $\gamma^\ar$.
Due to \eqref{eq:9VtildeCONV} we have
\begin{equation}\label{eq:9VtildeCONV111}
\begin{split}
Q^\ar(t, \theta^{(j)} \gamma^\ar) & = \sum_{n \in \zv} d^{(r,j)}(n, \gamma^\ar) e^{2 \pi i t n \tilde \omega^\ar}, \quad t \in \mathbb{R}, \\
|d^{(r,j)}(n, \gamma^\ar)| & \le \sqrt{2\ve} \exp \big( -\frac{\kappa_0}{2} |n| \big), \quad n \in \zv \setminus \{ 0 \},
\end{split}
\end{equation}
$j = 0, 1$. Due to \eqref{eq:8-4coeflim},
\begin{equation}\label{eq:8-4coeflim111}
\lim_{r \to \infty} d^{(r,j)}(n,\gamma^\ar) = d^{(j)}(n),
\end{equation}
where
\begin{equation}\label{eq:8-4decay12311}
Q(t,\theta^{(j)}) = \sum_{n \in \zv} d^{(j)}(n) e^{2 \pi i t n\omega},
\end{equation}
$j = 0, 1$; see \eqref{eq:8-4decay123-2}. Since $Q(t, \theta^\zero) = Q(t, \theta^{(1)})$, we conclude that $d^{(0)}(n) = d^{(1)}(n)$ for any $n \in \zv$. Thus,
\begin{equation}\label{eq:8-4coeflim1121}
\lim_{r \to \infty} |d^{(r,0)}(n,\gamma^\ar) - d^{(r,1)}(n,\gamma^\ar)| = 0.
\end{equation}
Combining \eqref{eq:9VtildeCONV111} with \eqref{eq:8-4coeflim1121}, we conclude that
\begin{equation}\label{eq:9VtildeCONV11121}
\lim_{r \to \infty} \sup_{t \in \mathbb{R}} |Q^\ar(t, \theta^{(0)} \gamma^\ar) - Q^\ar(t, \theta^{(1)} \gamma^\ar)| = 0.
\end{equation}
Recall that $\mu_n^r(\theta^{(k)}_n)$ is a Dirichlet eigenvalue for the Hill equation with potential $Q^\ar(t, \theta^{(k)} \gamma^\ar)$. Due to \eqref{eq:9HillR-2},
\begin{equation} \label{eq:9HillR1}
|\mu_n^r(\theta^{(0)}_n) - \mu_n^r(\theta^{(1)}_n)| \le \sup_{t \in \mathbb{R}} |Q^\ar(t, \theta^{(0)} \gamma^\ar) - Q^\ar(t, \theta^{(1)} \gamma^\ar)|.
\end{equation}
Recall that
\begin{equation}\label{eq:9auxmanofold}
\mu^\ar_n(\theta) = E^{r,-}_n + (E^{r,+}_n-E^{r,-}_n) \sin^2 \Big( \frac{\theta_n}{\xi_n} \Big), \quad \xi_n = \exp(-n^{\sigma(a_0, b_0, \kappa_0, \nu)}/2),
\end{equation}
see \eqref{eq:8auxmanofold}. Recall also that $0 \le \theta^{(k)} \le \pi \xi_n/2$. So,
\begin{equation} \label{eq:9HillR122}
\begin{split}
|\mu_n^r(\theta^{(0)}_n) - \mu_n^r(\theta^{(1)}_n)| & = (E^{r,+}_n - E^{r,-}_n) \Big| \sin^2 \Big( \frac{\theta^\zero_n}{\xi_n} \Big) - \sin^2 \Big( \frac{\theta^\one_n}{\xi_n} \Big) \Big| \\
& \ge c(E^{r,+}_n - E^{r,-}_n) |\theta^{(0)}_n - \theta^{(1)}_n|,
\end{split}
\end{equation}
$c = c(\frac{\theta^{(0)}_n}{\xi_n},\frac{\theta^{(1)}_n}{\xi_n})$. Once again,
\begin{equation} \label{eq:9HillR12233}
\lim_{r \to \infty} (E^{r,+}_n - E^{r,-}_n) = E^{+}_n - E^{-}_n > 0.
\end{equation}
Combining \eqref{eq:9HillR122} with \eqref{eq:9HillR1}, \eqref{eq:9VtildeCONV11121}, and  \eqref{eq:9HillR12233}, we conclude that $\theta^{(0)}_n = \theta^{(1)}_n$.

Let $Q(x)$ be a bounded continuous real function, $x \in \R$. Consider the Schr\"odinger operator
\begin{equation} \label{eq:9craig1ThI}
[H_Q y](x) = - y''(x) +  Q(x) y(x), \quad x \in \IR.
\end{equation}
Let $\mathcal{S}_Q$ be the spectrum of $H_Q$. Assume that $\mathcal{S}_Q = \mathcal{S}$ and that $H_Q$ is reflectionless. Recall that due to \eqref{eq:8Gnestimod}--\eqref{eq:6resonanceAA-2}, we have
\begin{equation}\label{eq:8GnestimodThI}
\gamma_n \le \exp(-n^\sigma), \quad n \ge 1, \quad \gamma_n < \eta_{i,n}^4, \quad i < n.
\end{equation}
This implies the conditions \eqref{eq:9craig2} in Remark~\ref{rem:9fortheorI}. So,
\begin{equation} \label{eq:craig3ThI}
Q(t) = \sum_{n} E_n^+ + E^-_n - 2\mu_n(t),
\end{equation}
where the functions $\mu_n(t)$ obey the differential equation
\begin{equation}\label{eq:9craig4ThI}
\dot{\mu}_n = \Phi_n(\mu) = \sigma(\mu_n) \sqrt{4 (\underline E - \mu_n)(E_n^- - \mu_n)(E_n^+ - \mu_n) \prod_{\substack{i \neq n}} \frac{(E_i^- - \mu_n)(E_i^+ - \mu_n)}{(\mu_i - \mu_n)^2}}
\end{equation}
and $\mu_n(t) \in [E_n^-, E_n^+]$. Due to Lemma~\ref{lem:8Lml2ODEmutheta}, $\mu(t) = \mu(\theta(t))$ with
\begin{equation}\label{ml07origODEthetaThI}
\dot{\theta} = \Psi(\theta).
\end{equation}
Here $\mu_n(\theta_n)$ is the same as above. Just as above we derive that
\begin{equation}\label{eq:8-4decay123-3}
\begin{split}
Q(t) & = \sum_{n \in \zv} d(n) e^{2 \pi i t n\omega}, \quad t \in \mathbb{R}, \\
|d(n)| & \le \sqrt{2\ve} \exp \big( -\frac{\kappa_0}{2} |n| \big), \quad n \in \zv;
\end{split}
\end{equation}
see the derivation of \eqref{eq:8-4decay123-2}.

Let $Q \in \mathcal{ISO}(V) \cap \mathcal{P}(\omega, \sqrt{2\ve}, \kappa_0/2)$. We assume that $\ve_0$ (and hence $\sqrt{2\ve}$) is small enough so that the existing theory implies that the spectrum of $H_Q$ is purely absolutely continuous, see \cite{DG, El}. In particular $H_Q$ is reflectionless; compare \cite{R07}. Since we also have by assumption that $\mathcal{S}_Q = \mathcal{S}$, it follows from the arguments above that $Q = \Phi(\theta^\zero)$ for some $\theta^\zero \in \mathcal{M}$.

Thus, $\Phi$ is a bijection from $\mathcal{M}$ onto the $Q \in \mathcal{ISO}(V) \cap \mathcal{P}(\omega, \sqrt{2\ve}, \kappa_0/2)$, as claimed in Theorem~$I$. Finally, let us verify that $\Phi$ is a homeomorphism.

Let $\theta^\zero, \theta^\one \in \mathcal{M}$. Let $\theta^{(j)}(t)$ be the $\Psi$-trajectory originating at $\theta = \theta^{(j)}$, $j = 0, 1$. Then due to Proposition~\ref{prop:8.stabil1}, for any $T$, we have
\begin{equation}\label{eq:9ation2theta}
\max_{0 \le t \le T} d(\theta^\zero (t), \theta^{(1)}(t)) \le A (4K)^{BT} d(\theta^\zero,\theta^\one),
\end{equation}
where $A, B, K$ depend on $a_0, b_0, \kappa_0, \nu$. Let
\begin{equation} \label{eq:9traceformula11222}
Q(t, \theta^{(j)}) := \sum_{n} E_n^+ + E^-_n - 2 \mu_n(\theta^{(j)}_n (t)), \quad j = 0, 1.
\end{equation}
Then,
\begin{equation} \label{eq:9traceformula11223}
\begin{split}
|Q(t, \theta^\zero) - Q(t, \theta^\one)| & \le \sum_{n} 2 |\mu_n(\theta^{(0)}_n (t)) - \mu_n(\theta^{(1)}_n (t))| \\
& \le C \sum_{n} \frac{E_n^+- E^-_n}{\xi_n} |\theta^{(0)}_n (t) - \theta^{(1)}_n (t)| \\
& \le C(a_0, b_0, \kappa_0, \nu) \, d(\theta^\zero (t), \theta^{(1)}(t)).
\end{split}
\end{equation}
Combining this with \eqref{eq:9ation2theta} we obtain
\begin{equation}\label{eq:9ation2theta111}
\max_{0 \le t \le T} |Q(t, \theta^\zero) - Q(t, \theta^\one)| \le A_1 (4K)^{BT} d(\theta^\zero, \theta^\one) := \delta_T(\theta^\zero, \theta^\one).
\end{equation}
We have
\begin{equation}\label{eq:8-4decay123Fin}
\begin{split}
Q^{(j)}(t) & = \sum_{n \in \zv} g^{(j)}(n) e^{2 \pi i t n \omega}, \quad t \in \mathbb{R}, \\
|g^{(j)}(n)| & \le \ve^{1/2} \exp \Big( -\frac{\kappa_0}{2} |n| \Big), \quad n \in \zv, \quad j = 0, 1.
\end{split}
\end{equation}
Let $\beta_0 := (2 b_0)^{-1}$. Due to Lemma~\ref{Lml:8quasibasis}, for $|n| \le T^{\beta_0}$, we have
\begin{equation}\label{eq:9qpbasis4}
|g^\zero(n) - g^\one(n)| \le \delta_T (\theta^\zero, \theta^\one) + \frac{C(a_0, b_0, \kappa_0, \nu)}{T^{1/2}} + C(\kappa_0, \nu) \exp \Big( -\frac{\kappa_0}{2} T^{\beta_0} \Big).
\end{equation}
Therefore,
\begin{equation} \label{eq:9Fourdist}
\begin{split}
d((g^\zero(n))_n, (g^\one(n)_n)_n) & = \sum_{n \in \zv} |g^\zero(n) - g^\one(n)| \\
& = \sum_{|n| \le T^{\nu^{-1} \beta_0}} |g^\zero(n) - g^\one(n)| + \sum_{|n| > T^{\nu^{-1} \beta_0}} |g^\zero(n) - g^\one(n)| \\
& \le C (\nu_0) T^{\nu \beta_0} \left[ \delta_T(\theta^\zero, \theta^\one) + \frac{C(a_0, b_0, \kappa_0, \nu)}{T^{1/2}} + C(\kappa_0, \nu) \exp \Big( -\frac{\kappa_0}{2} T^{\beta_0} \Big) \right] \\
& \qquad + \sum_{|n|> T^{\nu^{-1}\beta_0}} |(g^\zero(n)| + |g^\one(n)|) \\
& \le A_1 (4K_1)^{BT} d(\theta^\zero, \theta^\one) + \frac{C_1(a_0, b_0, \kappa_0, \nu)}{T^{1/4}}.
\end{split}
\end{equation}
It follows from \eqref{eq:9Fourdist} that the map $\Phi : (\mathcal{M},d) \to (\mathcal{ISO}(V) \cap \mathcal{P}(\omega, \sqrt{2\ve}, \kappa_0/2), d)$ is continuous. Since $(\mathcal{M},d)$ is compact and $\Phi$ is injective, the inverse is also continuous.
\end{proof}

\end{document}